\documentclass[a4paper,11pt]{article}
\usepackage[utf8]{inputenc}
\usepackage[english]{babel}
\usepackage{amsmath}
\usepackage{amsthm}
\usepackage{amssymb}
\usepackage{mathtools}
\mathtoolsset{showonlyrefs}
\usepackage{tikz}
\usepackage{verbatim}
\usepackage{fancyvrb}
\usepackage{array}
\usepackage{adjustbox}
\usepackage{changepage}
\usepackage{subcaption}
\usepackage{enumerate}
\usepackage{wrapfig}
\usepackage{bbm}
\usepackage{bm}
\usepackage{enumitem}
\usepackage{csquotes}
\usepackage{hyperref}
\usepackage{framed}
\newtheorem{defi}{Definition}[section]
\newtheorem{thm}[defi]{Theorem}
\newtheorem{lem}[defi]{Lemma}

\newtheorem{prop}[defi]{Proposition}

\newtheorem{cor}[defi]{Corollary}
\newtheorem{rem}[defi]{Remark}

\newcommand{\N}{\mathbb{N}}
\newcommand{\Z}{\mathbb{Z}}

\newcommand{\C}{\mathbb{C}}
\newcommand{\compC}{\mathbb{C}}
\newcommand{\contC}{\mathcal{C}}

\newcommand{\gap}{S}
\newcommand{\Khat}{\widetilde{K}}
\newcommand{\Mconst}{\mathbb{M}}
\newcommand{\R}{\mathbb{R}}
\newcommand{\realR}{\mathbb{R}}

\newcommand{\eps}{\epsilon}

\newcommand{\om}{\omega}

\newcommand{\bigO}{\mathcal{O}}

\newcommand{\spn}[1]{\operatorname{span}\{#1\}}

\newcommand{\MF}{\mathcal{P}}

\newcommand{\E}{\mathbb E}

\DeclareMathOperator{\U}{U}
\DeclareMathOperator{\diag}{diag}

\newcommand{\lambdascaled}{\hat{\lambda}}
\newcommand{\lambdascaledM}{\tilde{\lambda}}
\newcommand{\etascaled}{\hat{\eta}}
\newcommand{\xiscaled}{\hat{\xi}}
\newcommand{\zetascaled}{\hat{\zeta}}

\newcommand{\gfrac}[2]{\genfrac{}{}{0pt}{}{#1}{#2}}

\newcommand{\hyperF}[5]{{}_{#1}F_{#2} \left( \gfrac{#3}{#4} ; {#5} \right)}

\DeclareMathOperator{\Ai}{Ai}
\DeclareMathOperator{\arcosh}{arcosh}
\DeclareMathOperator{\arccot}{arccot}
\DeclareMathOperator{\ballistic}{\mathrm{I}}
\DeclareMathOperator{\Bes}{Bes}
\DeclareMathOperator{\bulk}{bulk}
\DeclareMathOperator{\crit}{crit}

\DeclareMathOperator{\crossover}{\mathrm{II}}
\DeclareMathOperator{\diffusive}{\mathrm{met}}
\DeclareMathOperator{\hard}{hard}
\DeclareMathOperator{\Inner}{inn}
\DeclareMathOperator{\insulator}{\mathrm{ins}}
\DeclareMathOperator{\Left}{lef}
\DeclareMathOperator{\Local}{loc}
\DeclareMathOperator{\Lower}{low}
\DeclareMathOperator{\Middle}{mid}
\DeclareMathOperator{\Outer}{out}
\DeclareMathOperator{\Right}{rig}
\DeclareMathOperator{\PV}{P.V.}
\DeclareMathOperator{\soft}{soft}
\DeclareMathOperator{\Upper}{upp}
\DeclareMathOperator{\Var}{Var}

\title{Invariant pseudo-unitary matrix ensembles and local statistics and conductance fluctuations of the DMPK ensemble}
\author{Dong Wang\footnote{School of Mathematical Sciences, University of Chinese Academy of Sciences. \\ E-mail: \texttt{wangdong[at]wangd-math.xyz}.} \and Thomas Wolfs\footnote{Department of Mathematics, KTH Royal Institute of Technology. \\ E-mail: \texttt{wolfs[at]kth.se}.}} 
\date{}
\begin{document}
\maketitle

\begin{abstract}
We consider the DMPK ensemble, which models quantum transport in a disordered wire, and describe its limiting local statistics as the number of particles tends to infinity. In the metallic regime, as well as at crossover I, we obtain sine statistics in the bulk, Bessel statistics at the hard edge and Airy statistics at the soft edge. In the insulating regime, we prove localization around certain explicit deterministic positions with Gaussian fluctuations. This confirms and quantifies the predicted universal and localized behavior in the physics literature. At crossover II, we obtain explicit limiting kernels in the bulk and at both edges, which interpolate between these regimes. The hard edge kernel involves Legendre functions and, to the best of our knowledge, is new, while the bulk and soft edge kernels agree with those derived for the squared singular values of critically scaled products of independent Ginibre matrices. To explain these similarities, we construct a class of invariant pseudo-unitary matrices that is closed under taking products and use it to interpret the DMPK ensemble as the joint density of the spectral variables of a product of such random matrices. More generally, the latter takes the form of a polynomial ensemble of derivative type involving a second order differential operator compatible with the Legendre functions. Finally, we obtain uniform asymptotic expansions of the mean and variance of the conductance and prove that the DMPK ensemble admits universal conductance fluctuations, i.e. the variance of the conductance tends to $1/15$ throughout the whole metallic regime. This again confirms the predictions in the physics literature.
\end{abstract}

\tableofcontents

\section{Introduction}

Let $n\in\N$ and $M>0$. We will consider the probability density
\begin{equation}\label{eq:DMPK-density}
    \mathfrak{p}_{n,M}(x)
    =\frac{1}{Z_{n,M}}\prod_{1\leq j<k\leq n}(x_k-x_j)
    \det[w_{j,M}(x_k)]_{j,k=1}^n,
    \qquad x\in[1,\infty)^n,
\end{equation}
where $Z_{n,M}\neq0$ is the normalization constant and
\begin{equation}\label{w_j}
    w_{j,M}(x)
    =\int_0^\infty e^{-\frac{M}{n}s^2}s^{2j-1}
    \tanh(\pi s)P_{is-\frac12}(x)ds,
    \qquad x\in [1,\infty), \quad j=1,\dots,n,
\end{equation}
is defined in terms of the Legendre function of the first kind
\begin{equation}\label{eq:Leg}
    P_{is-\frac12}(x)
    =\hyperF{2}{1}{\frac12-is,\frac12+is}{1}{\frac{1-x}{2}},
    \qquad x\in [1,\infty),
\end{equation}
see \cite[Eq.~14.3.6]{Boisvert-Clark-Lozier-Olver10}. The density \eqref{eq:DMPK-density} arises in the theory of quantum transport in a disordered wire with $n$ propagating channels. We refer to \cite{Beenakker97} in the physics literature for an extensive review of the physical background. After the change of variables $x_k\mapsto \cosh \xi_k$ for $k=1,\dots,n$, the density is the exact solution of the Dorokhov-Mello-Pereyra-Kumar (DMPK) equation for $\beta=2$, with ballistic initial condition
\[
    \lim_{M\downarrow0}\mathfrak{p}_{n,M}(\xi)
    =\prod_{k=1}^n\delta(\xi_k),
\]
in the sense of distributions. We will therefore refer to \eqref{eq:DMPK-density} as the DMPK ensemble and the transplanted variables $\xi_1,\dots,\xi_n\in[0,\infty)$ as the physical variables. The DMPK equation is a Fokker-Planck equation introduced by Dorokhov in \cite{Dorokhov82} and independently by Mello, Pereyra and Kumar in \cite{Mello-Pereyra-Kumar88}, because it provided a model for quantum transport in which, according to the predictions, a single dynamic parameter controls whether the particles fall into a diffusive (metallic) or localized (insulating) regime. Its exact solution for $\beta=2$ was derived by Beenakker and Rejaei in \cite{Beenakker-Rejaei93,Beenakker-Rejaei94} and the determinantal structure was subsequently analyzed by Frahm in \cite{Frahm95}. In the notation of \eqref{eq:DMPK-density}, the role of this dynamic parameter is being played by $M>0$. It corresponds to $s=L/l$ in \cite[Eq. (3.21)]{Beenakker97} and can be seen as the wire length in units of the mean free path. The distinguished physical regimes are the ballistic regime $M\ll1$, the metallic or diffusive regime $1\ll M\ll n$, and the insulating or localized regime $n\ll M$. We will also study the two intermediate regimes: the crossover I regime, where $M>0$ is fixed, and the crossover II regime, where $M=\alpha n$ with fixed $\alpha>0$.
\medbreak

The DMPK equation also has an interpretation as a model for Brownian motion on pseudo-unitary matrices, see \cite{Bachmann-Butz-De_Roeck12}. We will come back to the definition of pseudo-unitary matrices later. In particular, the ordered particles associated with \eqref{eq:DMPK-density} can be interpreted as the positions at time $M$ of $n$ non-intersecting paths with entrance configuration $x_1(0)=\cdots=x_n(0)=1$. 
\begin{comment}
They satisfy
\begin{equation*}
    dx_j(M)
    = \sqrt{\frac{2}{n}\bigl(x_j(M)^2-1\bigr)}\,dB_j(M)
    + \frac{2}{n}\Big(
        x_j(M)
        + \sum_{\substack{k=1\\k\neq j}}^n
          \frac{x_j(M)^2-1}{x_j(M)-x_k(M)}
      \Big)dM,
    \qquad j=1,\dots,n,
\end{equation*}
where $B_1,\dots,B_n$ are independent standard real Brownian motions.
\end{comment}
Figure~\ref{fig1} illustrates the simulated paths in the physical variables $\xi_1(M),\dots,\xi_n(M)\in[0,\infty)$ for $n=5$ and $M\in[0,10]$, together with the different regimes. General properties of non-intersecting path models have been studied in \cite{Cepa-Lepingle97,Graczyk-Malecki14,Duits18,Assiotis23}. Local limits for non-intersecting Brownian and squared Bessel paths have also received considerable attention, see, e.g., \cite{Daems-Kuijlaars-Veys08,Johansson13,Delvaux-Kuijlaars-Zhang11,Neuschel-Venker24} and \cite{Kuijlaars-Martinez_Finkelshtein-Wielonsky09,Kuijlaars-Martinez_Finkelshtein-Wielonsky11,Delvaux-Kuijlaars-Roman-Zhang12}, respectively. 
\medbreak

\begin{figure}[h]
    \centering
    \includegraphics[width=0.6\linewidth]{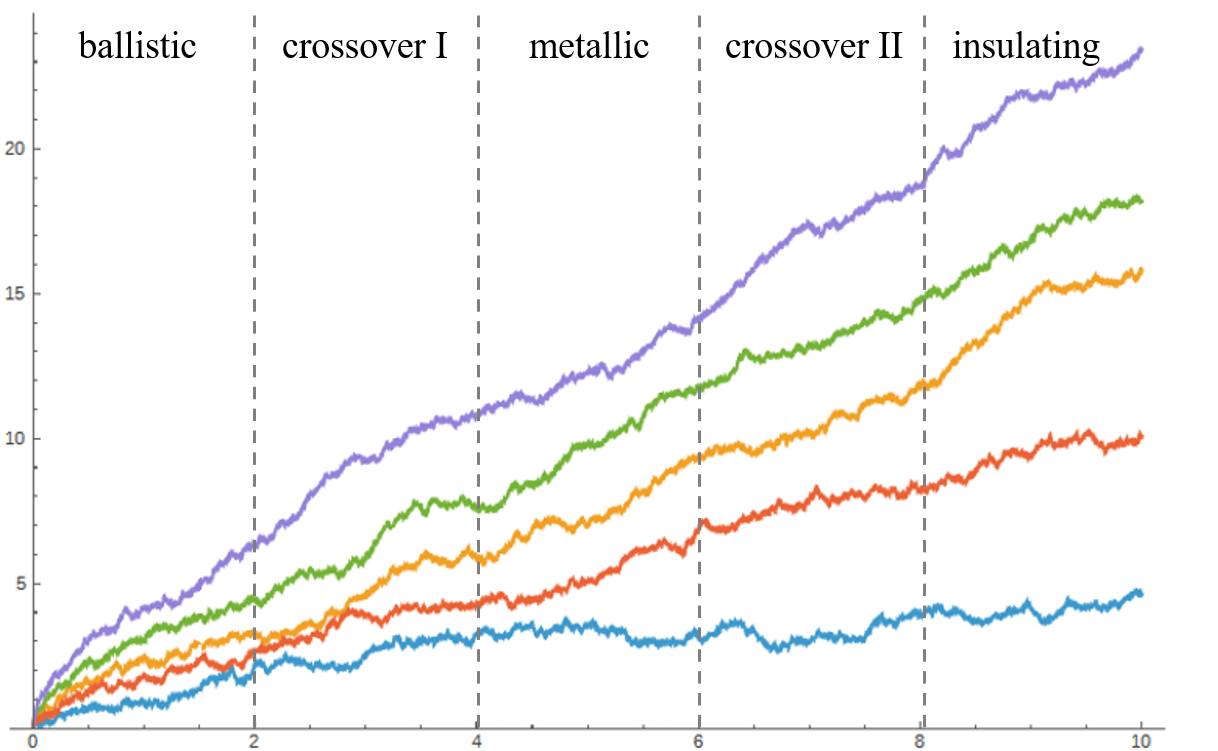}
    \caption{DMPK ensemble as a non-intersecting path model.} \label{fig1}
\end{figure}

\medbreak

Our first objective is to describe the limiting local statistics of the DMPK ensemble as $n\to\infty$ and to quantify the predictions from physics in the metallic and insulating regime. To the best of our knowledge, no predictions about the local statistics in the crossover regimes have been made so far. In the metallic regime, we will establish limiting sine statistics in the bulk, Bessel statistics at the hard edge and Airy statistics at the soft edge. This confirms the universal diffusive behavior predicted in \cite{Macedo94a} and \cite[Sect.~III.B]{Beenakker97}. The same three limits occur in the crossover I regime. In the insulating regime, the particles are predicted to be localized around deterministic positions, see \cite[Sect.~III.C]{Beenakker97}. We will identify these positions and prove that the particles exhibit Gaussian fluctuations around them, with asymptotically empty regions between these positions. In the crossover II regime, we obtain explicit limiting kernels in the bulk and at both edges, describing the intermediate local behavior. 

Since the crossover II regime acts as the transition between the metallic regime with classical random matrix local limits and the insulating regime with Gaussian local limits, the limiting local statistics in crossover II regime will be fairly non-standard. In spirit, similar transitions appear in connection to the one-dimensional Anderson Hamiltonian, see \cite{Dumaz-Labbe20, Dumaz-Labbe23, Dumaz-Labbe24}. Curiously, such transitions also appear in connection to products of independent Ginibre matrices. A Ginibre matrix is a matrix whose entries are independent standard complex Gaussian random variables. The local statistics of the squared singular values of products of independent Ginibre matrices have been studied extensively so far, see \cite{Akemann-Kieburg-Wei13,Kuijlaars-Zhang14,Liu-Wang-Zhang14}, and the review \cite{Akemann-Ipsen15} in the physics literature. It was shown in \cite{Liu-Wang-Wang18} that when the number of factors and matrix dimension tend to infinity together, their ratio governs a similar transition between the classical random matrix sine and Airy limits and Gaussian fluctuations. This confirmed the earlier predictions in \cite{Akemann-Burda-Kieburg18,Akemann-Burda-Kieburg20}. We will show that the limiting bulk and soft edge kernels for the DMPK ensemble at crossover II agree with those described for Ginibre products when the number of factors is proportional to the matrix dimension. Remarkably, it was shown in \cite{Ahn23} that the same soft edge limit arises for critical products within a broad class of right unitarily invariant random matrices, and a connection was given to the edge statistics of a Brownian motion on $\mathrm{GL}(n,\C)$. Further properties of this kernel were studied in \cite{Berezin-Strahov22}. At the hard edge, we obtain a limiting kernel involving Legendre functions, which to the best of our knowledge has not appeared before.
\medbreak

Our second objective is to identify a product structure underlying the DMPK ensemble, explaining the previous similarities with products of Ginibre matrices. We call a matrix $g\in\mathrm{GL}(2n,\C)$ pseudo-unitary if it satisfies
\[
g^*Jg=J,\quad J=\begin{pmatrix}I_n&0\\0&-I_n\end{pmatrix}.
\]
The set of such matrices is denoted by
\begin{equation} \label{eq:def-U(n,n)}
    \U(n,n)=\{g\in\mathrm{GL}(2n,\C):g^*Jg=J\},
\end{equation}
and forms a group $G$ with respect to matrix multiplication. We will construct a class of random matrices in the set of pseudo-unitary matrices $\U(n, n)$ whose laws are invariant under left and right multiplication with matrices in 
\begin{equation} \label{eq:defn_subgroup_K}
    K=\{\diag(u_1,u_2):u_1,u_2\in\U(n)\}
    \cong\U(n)\times\U(n),
\end{equation}
and which is closed under taking products of independent matrices. The relevant spectral variables will be referred to as the radial variables. The joint density of the radial variables in this class has the form
\begin{equation}\label{eq:Leg-DT}
    \rho_{\mathrm{rv}}(x)
    =\frac{1}{Z_n(\om)}\prod_{1\leq j<k\leq n}(x_k-x_j)
    \det[(D^{j-1}\om)(x_k)]_{j,k=1}^n,
    \qquad x\in(1,\infty)^n,
\end{equation}
where $\om$ is sufficiently regular, $Z_n(\om)\neq0$ is the normalization constant and
\begin{equation}\label{D}
    (D\om)(x)
    =(1-x^2)\om''(x)-2x\om'(x)-\tfrac14\om(x).
\end{equation}
We will show that the DMPK ensemble \eqref{eq:DMPK-density} belongs to this class and that, for $M\in\N$, it is the joint density of the radial variables of a product of $M$ independent, identically distributed pseudo-unitary matrices. In order to prove this, we will make use of some tools from harmonic analysis, particularly the spherical transform of the pair $(\U(n,n),K)$.

Ensembles with similar closure properties have already been constructed for products of bi-unitarily invariant invertible complex matrices in \cite{Kieburg-Kosters16,Kieburg-Kosters19}, sums of unitarily invariant Hermitian matrices in \cite{Kuijlaars-Roman19} and products of conjugation invariant unitary matrices in \cite{Kieburg-Li-Zhang-Forrester23}. All of these constructions fit into the theory of polynomial ensembles of derivative type. In each setting, the appropriate differential operator depends on the underlying matrix space.
\medbreak

An important linear statistic in the context of the DMPK ensemble \eqref{eq:DMPK-density} is the conductance
\begin{equation}\label{eq:defn_conductance}
    G=\sum_{k=1}^n\frac{2}{1+x_k},
\end{equation}
because the variables $T_k=2/(1+x_k)$ represent the transmission eigenvalues of the underlying physical model. Our third and final objective is to prove that the DMPK ensemble \eqref{eq:DMPK-density} admits universal conductance fluctuations throughout the metallic regime, i.e.
\begin{equation}\label{eq:UCF}
    \lim_{n\to\infty}\Var_{n,M_n}(G)=\frac{1}{15},
    \qquad \text{if}\quad M_n\to\infty,\quad M_n/n\to0,
\end{equation}
see \cite[Eq.~(3.35)]{Beenakker97}, and \cite[Conj.~2]{Bachmann-Butz-De_Roeck12} with $\beta=2$. We will obtain asymptotic expansions of $\E_{n,M}[G]$ and $\Var_{n,M}(G)$ that hold uniformly for $M\in (0,n-1)$, and from which \eqref{eq:UCF} will readily follow. Both expansions match the predictions in \cite[Eqs.~(3.34) \& (3.35)]{Beenakker97}. For a fixed $M$, the expansion of the mean $\E_{n,M}[G]$ was already proven rigorously before in \cite[Thm.~2]{Bachmann-Butz-De_Roeck12} using different techniques. Our expansion of the variance matches the one obtained in \cite{Wang23} for a model proposed in \cite{Beenakker-Rejaei93,Beenakker-Rejaei94} that was predicted to serve as an appropriate approximation to the DMPK ensemble at crossover I. Remarkably, our bulk limits at crossover I also match the ones obtained in \cite{Wang23}, thereby confirming the predicted accurateness of the appropriating model.
\medbreak

Essential in our asymptotic analysis will be a double integral representation of the correlation kernel. Recall that the DMPK ensemble \eqref{eq:DMPK-density} is a biorthogonal ensemble and hence a determinantal point process, see \cite{Johansson06} for an introduction. Its correlation kernel can be written as
\begin{equation}\label{K}
    K_n(x,y)=\sum_{k=0}^{n-1}f_k(x)g_k(y),
\end{equation}
where the functions $f_0,\dots,f_{n-1}:[0,\infty)\to\R$ form a basis of the space of polynomials of degree at most $n-1$, the functions $g_0,\dots,g_{n-1}:[0,\infty)\to\R$ form a basis of $\spn{w_{1,M},\dots,w_{n,M}}$ defined in \eqref{w_j}, and satisfy the biorthogonality relations
\begin{equation}\label{eq:biorthogonality}
    \int_1^\infty f_j(x)g_k(x)dx=\delta_{j,k},
    \qquad j,k=0,\dots,n-1.
\end{equation}
A particular biorthogonal system for \eqref{eq:DMPK-density} was already derived in \cite[Eq. (16)]{Frahm95}, see also \cite[Eq. (3.79)]{Beenakker97}, but is not appropriate for further rigorous analysis. We will derive a different biorthogonal system adapted to the derivative type structure in \eqref{eq:Leg-DT}. This choice will yield a double integral representation for the correlation kernel of the general class, which we will then specialize to the DMPK ensemble. Similar phenomenon also occurred for other kinds of polynomial ensembles of derivative type, see \cite{Claeys-Zhang26} for an overview. As explained there, the resulting double integral representations often open up the road for an asymptotic analysis through a steepest descent analysis. The latter has been proven to be successful in, e.g., \cite{Johansson01a,Baik-Ben_Arous-Peche05,Liu-Wang-Zhang14,Liu-Wang-Wang18}.
\medbreak

The remainder of this work is organized as follows. Section~\ref{sec:MR} contains our main results. In Section~\ref{sec:IPUM}, we develop the harmonic analysis for products of invariant pseudo-unitary matrices, construct the ensembles of derivative type and derive double integral representations for their correlation kernels. In Section~\ref{sec:conductance}, we prove the uniform expansions of the mean and variance of the conductance. Finally, in Section~\ref{sec:LS}, we derive the local limits in the insulating, crossover II, crossover I and metallic regimes. The appendix collects the properties and asymptotic estimates of Legendre functions used in the proofs.

\section{Main results} \label{sec:MR}

Recall the group of pseudo-unitary matrices $\U(n, n)$ defined in \eqref{eq:def-U(n,n)}. Since it is a connected reductive, and thus unimodular, Lie group, see \cite[Sects. 1.4 \& 1.5]{Anker-Orsted05}, it admits a Haar measure $dg$ which is both left and right invariant and is unique up to multiplication by a positive constant. In Section \ref{S:IPUM-prod}, we will fix the precise normalization of $dg$. We will consider probability measures $d\mu(g)=\rho(g)dg$ on $\U(n,n)$ whose densities posses left and right multiplication invariance with respect to the subgroup $K$ defined in \eqref{eq:defn_subgroup_K}. The precise definition is as follows.

\begin{defi}\label{def:K-biinvariant-ensemble}
A $K$-invariant ensemble on $\U(n,n)$ is an absolutely continuous probability measure $d\mu(g)=\rho(g)dg$ with respect to Haar measure $dg$ on $\U(n,n)$ whose density $\rho:\U(n,n)\to\R_{\geq 0}$ is bi-$K$-invariant, i.e. $\int_{\U(n, n)}\rho dg=1$ and $\rho(k_1gk_2)=\rho(g)$ for all $k_1,k_2\in K$. A random matrix with law $\mu$ is said to be a $K$-invariant pseudo-unitary random matrix.
\end{defi}

We will focus on a subclass of ensembles that only depends on a single function $\om:\R_{\geq 1}\to\R$ and where the $n$-dimensional structure is created by consecutive applications of the differential operator
\begin{equation*}
(D\om)(x)= (1-x^2)\om''(x) -2x \om'(x) - \tfrac{1}{4}\om(x),\qquad x\in\R_{\geq1}.
\end{equation*}
For technical convenience in the proofs, we will restrict to functions from the following set
\begin{equation} \label{eq:def-W}
    \mathcal{W}(\R_{\geq 1}) = \Big\{f \in C^{\infty}(\R_{\geq 1}) : \int_1^\infty |(D^kf)(x)| x^m dx < \infty \text{ for all } k,m\in\Z_{\geq 0} \Big\},
\end{equation}
which is certainly not optimal. To every pseudo-unitary matrix, we can associate a matrix of (transplanted) radial variables of the form
\begin{equation}\label{eq:a-x}
\alpha_{x}=\frac1{\sqrt2}
\begin{pmatrix}
(\Lambda_{x}+I_n)^{1/2}&(\Lambda_{x}-I_n)^{1/2}\\
(\Lambda_{x}-I_n)^{1/2}&(\Lambda_{x}+I_n)^{1/2}
\end{pmatrix},\quad \Lambda_{x}=\diag(x_1,\ldots,x_n),
\end{equation}
and it will be sufficient to define $K$-invariant ensembles on those matrices. We will carefully explain this in Section \ref{S:IPUM-prod}.

\begin{defi}\label{def:derivative-type-ensemble}
A $K$-invariant ensemble on $\U(n,n)$ with density $\rho$ is said to be of derivative type if there exists a function $\omega\in\mathcal{W}(\R_{\geq 1})$ and constant $Z_{n}(\om)\neq 0$ such that
\begin{equation}\label{eq:rho-omega-density}
\rho(\alpha_{x})=
\frac{\det[(D^{j-1}\omega)(x_k)]_{j,k=1}^n}{Z_{n}(\om)\Delta_n(x)},\qquad \text{ a.e. }x\in(1,\infty)^n.
\end{equation}
\end{defi}
We want to highlight that positivity of the right hand side of \eqref{eq:rho-omega-density} is a separate condition from membership in $\mathcal{W}(\R_{\geq 1})$, even if $\omega\ge0$.
\medbreak

Our first result is that the joint density of the unordered radial variables of a $K$-invariant random pseudo-unitary matrix whose law is of derivative type possesses the structure of a polynomial ensemble of derivative type (in the broader sense explained in the introduction). We will prove this in Section \ref{S:IPUM-prod}.

\begin{prop} \label{prop:imup-rv}
    The joint density of the unordered radial variables of a $K$-invariant random pseudo-unitary matrix whose law is of derivative type with respect to $\om\in\mathcal{W}(\R_{\geq 1})$ is given by
    \begin{equation}\label{eq:radial-density-DT}
    \rho_{\rm rv}( x)=\frac{1}{Z_{n}(\om)}\Delta_n(x)
    \det[(D^{j-1}\omega)(x_k)]_{j,k=1}^n,\qquad \text{a.e. } x\in(1,\infty)^n.
    \end{equation}
\end{prop}

More importantly, the precise derivative type structure allows us to prove the following closure property under taking products. The stated convolution $\ast_{\mathcal{P}}$ is such that
    $$ (\mathcal{P}[\om_1\ast_{\mathcal{P}}\om_2])(s) = \mathcal{P}\om_1(s)\cdot\mathcal{P}\om_2(s),\qquad s\in\C, $$
in terms of the Mehler-Fock transform
    $$ (\mathcal{P}\om)(s) = \int_1^\infty \om(x) P_{is-\frac{1}{2}}(x) dx,\qquad s\in\C. $$
We will give a more precise definition in Section \ref{S:IPUM-DT}, which also contains the proof of the closure property.

\begin{thm}\label{thm:closure-products}
Let $g_1,g_2$ be two independent $K$-invariant random pseudo-unitary matrices whose laws are of derivative type with respect to $\omega_1,\omega_2\in\mathcal{W}(\R_{\geq 1})$ respectively. Then $g_1g_2$ is a random pseudo-unitary matrix whose law is of derivative type with respect to $\omega_1\ast_{\MF}\omega_2$.
\end{thm}

Specializing Theorem \ref{thm:closure-products} to the setting of the DMPK ensemble \eqref{eq:DMPK-density} leads to the result below. We will give a proof in Section \ref{S:IPUM-DMPK}.

\begin{cor} \label{cor:DPMK-prod}
    Let $M\in\N$ and let $w_{1,1}$ and $w_{1,M}$ be as in \eqref{w_j}. Suppose that $\mathfrak{g}_1,\dots,\mathfrak{g}_{M}$ are independent $K$-invariant random pseudo-unitary matrices whose laws are of derivative type with respect to $w_{1,1}$. Then their product $\mathfrak{g}_1\cdots\mathfrak{g}_{M}$ is a $K$-invariant random pseudo-unitary matrix whose law is of derivative type with respect to $w_{1,M}$. In particular, its radial variables have joint probability density \eqref{eq:radial-density-DT} with $\om=w_{1,M}$, which is the DMPK ensemble \eqref{eq:DMPK-density}.
\end{cor}
\begin{rem}
    By slightly modifying the proof of Corollary \ref{cor:DPMK-prod}, we can also show that the conclusion continues to hold for any product $\mathfrak{g}_1\cdots\mathfrak{g}_{r}$ of independent $K$-invariant random pseudo-unitary matrices $\mathfrak{g}_1,\dots,\mathfrak{g}_{r}$ whose laws are of derivative type with respect to $w_{1,M_1},\dots,w_{1,M_r}$ respectively, where all $M_j>0$ and $M_1+\dots+M_r=M$.
\end{rem}

The correlation kernel of the joint density of unordered radial variables in \eqref{eq:radial-density-DT} admits the following double integral representation, opening up the road to further asymptotic analysis. It involves both the Legendre functions $P_{it-\frac{1}{2}}$ of the first kind and the Legendre functions $Q_{is-\frac{1}{2}}$ of the second kind, see \cite[Ch. 14]{Boisvert-Clark-Lozier-Olver10}. We will prove this in Section \ref{S:IPUM-DIR}.

\begin{thm} \label{DIR}
Let $\contC$ be a counterclockwise contour that encloses $\{\pm i(k+{\textstyle\frac{1}{2}})\}_{k=0}^{n-1}$ once and let $\Sigma$ be a contour from $-\infty$ to $\infty$ in $\{s\in\C:\Im s \leq 0\}$ that doesn't intersect $\contC$. Suppose that $\om\in\mathcal{W}(\R_{\geq 1})$ and that $\mathcal{P}\om$ is non-vanishing on a neighborhood of $\mathcal{C}_{-}$ and its interior. Then the correlation kernel of the ensemble in \eqref{eq:radial-density-DT} can be written as
\begin{equation} \label{K_gen}
    K_n(x,y) = \int_{\Sigma} \oint_{\contC} \frac{(\mathcal{P}\om)(s)}{(\mathcal{P}\om)(t)} \frac{(\frac{1}{2}-is)_{n}(\frac{1}{2}+is)_{n}}{(\frac{1}{2}-it)_{n}(\frac{1}{2}+it)_{n}} \frac{P_{it-\frac{1}{2}}(x) Q_{is-\frac{1}{2}}(y)}{t-s} \frac{tdt}{\pi i} \frac{ds}{2\pi i}.
    \end{equation}
\end{thm}

Specializing Theorem \ref{DIR} to the setting of the DMPK ensemble \eqref{eq:DMPK-density} leads to the result below. We will give a proof in Section \ref{S:IPUM-DIR}.

\begin{cor} \label{DIR_DMPK}
Let $\contC$ be a counterclockwise contour that encloses $\{\pm i(k+{\textstyle\frac{1}{2}})\}_{k=0}^{n-1}$ once and let $\Sigma$ be a contour from $-\infty$ to $\infty$ in $\{s\in\C:\Im s \leq 0\}$ that doesn't intersect $\contC$. Then the correlation kernel of the DMPK ensemble \eqref{eq:DMPK-density} can be written as
\begin{equation} \label{eq:defn_K_in_PQ_simplified_further}
  K_{n,M}(x,y) = \int_{\Sigma} \oint_{\contC}  \frac{e^{-\frac{M}{n} (s^2-t^2)}}{t-s} \frac{(\frac{1}{2}-is)_{n}(\frac{1}{2}+is)_{n}}{(\frac{1}{2}-it)_{n}(\frac{1}{2}+it)_{n}}   P_{it-\frac{1}{2}}(x) Q_{is-\frac{1}{2}}(y) \frac{tdt}{\pi i} \frac{ds}{2\pi i}.
\end{equation}
\end{cor}

The previous result will be the basis of our asymptotic analysis of the DMPK ensemble. To present its limiting local statistics, it will be more convenient to consider the correlation kernel in the dilated physical variables $\xi=M\xiscaled,\eta=M\etascaled \in \R_{\geq 0}$, i.e.
\begin{equation} \label{eq:standard_K_kernel}
  \mathbb{K}_{n,M}(\xiscaled,\etascaled) := M \sqrt{\sinh(M\xiscaled)\sinh(M\etascaled)} \, K_{n,M}(\cosh(M\xiscaled),\cosh(M\etascaled)),
\end{equation}
Evidently, if the particles $x_1,\dots,x_n\in[1,\infty)$ have a law with joint probability density \eqref{eq:DMPK-density}, then $\{ \lambdascaled_j = M^{-1} \arcosh x_j \}^n_{j = 1}$ is a determinantal point process with correlation kernel \eqref{eq:standard_K_kernel}. 

We will present the limits up to cocycles: when writing that $K_n(x, y) \to K(x, y)$ up to cocycles, we mean that there exists non-negative functions $f_n:\R_{\geq 1}\to\R$ such that 
\[
\lim_{n \to \infty} \frac{f_n(x)}{f_n(y)} K_n(x, y) = K(x, y). 
\]
The specific cocycles appear explicitly in the corresponding proofs. 
\medbreak

The limiting local statistics in the metallic regime, which has $M = M_n \to \infty$ and $M_n/n \to 0$ as $n \to \infty$, are given below. We will prove this in Section \ref{subsec:diffusive_regime}. This confirms and quantifies the universal behavior predicted in \cite{Macedo94a} and \cite[Sect.~III.B]{Beenakker97}. We briefly recall the following universal limiting correlation kernels in random matrix theory: the Bessel kernel (with parameter $\alpha > -1$),
\begin{equation} \label{eq:defn_K_Bes}
  K_{{\rm Bes},\alpha}(x,y) = \frac{J_{\alpha}(\sqrt{x}) \sqrt{y} J'_{\alpha}(\sqrt{y}) - \sqrt{x} J'_{\alpha}(\sqrt{x}) J_{\alpha}(\sqrt{y})}{2(x - y)} = \frac{1}{4} \int_0^1 J_{\alpha}(\sqrt{xt}) J_{\alpha}(\sqrt{yt}) dt,
\end{equation}
where $J_\alpha$ denotes the Bessel function of the first kind, see see \cite{Tracy-Widom94b}, the Airy kernel
\begin{equation} \label{eq:defn_K_Ai}
  K_{\rm Ai}(x,y) = \frac{\Ai(x) \Ai'(y) - \Ai'(x) \Ai(y)}{x - y}
  = i\int_{\Sigma^{(\Ai)}}\frac{du}{2\pi i} \int_{\contC^{(\Ai)}} \frac{dv}{2\pi i} \frac{1}{u - v} \frac{e^{\frac{i}{3} v^3 + iy v}}{e^{\frac{i}{3} u^3 + ix u}},
\end{equation}
where ${\rm Ai}$ denotes the Airy function, $\Sigma^{(\Ai)}$ is the contour from $e^{-\frac{5}{6} \pi i}\infty$ to $e^{-\frac{1}{6} \pi i}\infty$ and $\contC^{(\Ai)}$ is the contour from $e^{\frac{5}{6} \pi i}\infty$ to $e^{\frac{1}{6} \pi i}\infty$, see \cite{Tracy-Widom94} and \cite[Prop. 2.3]{Johansson03}, and the sine kernel
\begin{equation}
  K_{\sin}(x, y) = \frac{\sin \left(\pi (x - y) \right)}{\pi(x - y)},
\end{equation}
see \cite{Mehta04}.

\begin{thm}[Local statistics in the metallic regime] \label{thm:metallic}
    Take $M_n=\alpha_n n$ with all $\alpha_n>0$ such that $\lim_{n\to\infty} \alpha_n = 0$ and $\lim_{n \to \infty} M_n = \infty$.  
    \begin{enumerate}
    \item \label{enu:thm:metallic_1} 
      bulk: let $\eps>0$, take $r\in(\eps,1-\eps)$, and set $\xiscaled_n=2(r+\frac{\tau}{n})$ and $\etascaled_n = 2(r + \frac{\sigma}{n})$, then, up to cocycles,
      $$\frac{2}{n} \mathbb{K}_{n,M_n}(\xiscaled_n,\etascaled_n) \to K_{\sin}(\tau, \sigma), $$
    \item \label{enu:thm:metallic_2} 
      hard edge: set $\xiscaled_n=\frac{2\tau}{n}$ and $\etascaled_n = \frac{2\sigma}{n}$, then, up to cocycles,
      $$ \frac{2}{n} \mathbb{K}_{n,M_n}(\xiscaled_n,\etascaled_n) \to 2\pi^2  \sqrt{\tau \sigma} K_{\Bes, 0}((\pi\tau)^2,(\pi\sigma)^2),$$
    \item \label{enu:thm:metallic_3} 
      soft edge: set $\xiscaled_n = M^{-1}_n \zeta_{\max} + 2^{\frac{1}{3}} n^{-\frac{2}{3}} M^{-\frac{1}{3}}_n \tau$ and $\etascaled_n = M^{-1}_n \zeta_{\max} + 2^{\frac{1}{3}} n^{-\frac{2}{3}} M^{-\frac{1}{3}}_n \sigma$, where $\zeta_{\max}=2M_n+\bigO(\log M_n)$ as $n\to\infty$ is given by \eqref{eq:defn_zeta_max}, then, up to cocycles,
      $$2^{\frac{1}{3}} n^{-\frac{2}{3}} M^{-\frac{1}{3}}_n \mathbb{K}_{n,M_n}(\xiscaled_n,\etascaled_n) \to K_{\rm Ai}(\sigma,\tau).$$
    \end{enumerate}
\end{thm}

The limiting local statistics in the insulating regime, which has $M=M_n$ with $M_n \to \infty$ and $M_n/n \to \infty$ as $n \to \infty$, are as follows. We will give a proof in Section \ref{subsec:insulator_regime}. This confirms and quantifies the localized behavior predicted in \cite[Sect.~III.C]{Beenakker97}: the particles $\{ \lambdascaled_j \}^n_{j = 1}$ admit Gaussian fluctuations around certain classical positions, which we define as
\begin{equation} \label{eq:defn_classical_position}
        \zetascaled_{n, m} = 
            \frac{2m-1}{n} + \frac{\log(n + m - \frac{1}{2}) - \psi(n - m + 1)}{M_n}, \qquad m\in\{1,\dots,n\},
    \end{equation}
where $\psi$ denotes the digamma function.

\begin{thm}[Local statistics in the insulating regime] \label{thm:insulator}
    Take $M_n=\alpha_n n$ with all $\alpha_n>0$ and $\lim_{n\to\infty} \alpha_n = \infty$.    
    \begin{enumerate}
    \item \label{enu:thm:insulator_1}
      normal distribution around the $\zetascaled_{n, m}$: set $\xiscaled_n=\zetascaled_{n,m} + \sqrt{2} n^{-\frac{1}{2}} M_n^{-\frac{1}{2}} \tau$ and $\etascaled_n=\zetascaled_{n,m} + \sqrt{2} n^{-\frac{1}{2}} M_n^{-\frac{1}{2}} \sigma$, then, up to cocycles,
\begin{equation} \label{eq:insulator_fluct}
  \sqrt{2} n^{-\frac{1}{2}} M^{-\frac{1}{2}}_n \mathbb{K}_{n,M_n}(\xiscaled_n,\etascaled_n) \to \frac{1}{\sqrt{2\pi}} e^{-\frac{\sigma^2}{2}},
\end{equation}
      \item \label{enu:thm:insulator_2}
        gaps between the $\zetascaled_{n, m}$: for all $\epsilon > 0$, there exists $L> 0$, such that as $n\to\infty$, for all $m\in\{0,\dots,n+1\}$,
         \begin{equation}
             \mathbb{P}_{n,M_n}\left\{\mathrm{some}\ \lambdascaled_j \in (\zetascaled_{n, m - 1} + \frac{L}{\sqrt{n M_n}}, \zetascaled_{n, m} - \frac{L}{\sqrt{n M_n}} \right\} < \epsilon,
         \end{equation}
         where we set $\zetascaled_{n,0}+L/\sqrt{n M_n} =0$ and $\zetascaled_{n,n+1}-L/\sqrt{n M_n} =\infty$.
    \end{enumerate}
\end{thm}

\begin{rem}
Part \ref{enu:thm:insulator_1} implies that asymptotically there is exactly one particle around $\zetascaled$, which admits Gaussian fluctuations, see \cite[Thm.~1.1, Part 2]{Liu-Wang-Wang18} for a similar discussion. If $n - m \geq \epsilon n$ for some $\epsilon > 0$, then the classical position $\zetascaled_{n,m}$ in \eqref{eq:defn_classical_position} can simply be defined as 
\begin{equation} \label{eq:simplified_zetascaled}
   \zetascaled_{n, m} = \frac{2m-1}{n}.  
\end{equation}
In fact, whenever $(\log n)^2/\alpha_n \to 0 $ as $n\to\infty$, we can define all $\zetascaled_{n, m}$ by \eqref{eq:simplified_zetascaled}. This is because in the cases above, the discrepancy between the definitions of $\zetascaled_{n, m}$ is either $\bigO(M^{-1}_n)$ or $\bigO(\log(n) M^{-1}_n)$ as $n\to\infty$, which is much smaller than the scale of correlation $\bigO(n^{-1/2} M^{-1/2}_n)$ in \eqref{eq:insulator_fluct}.
\end{rem}

We obtain the following limiting local statistics at crossover II. We will prove this in Section \ref{subsec:crossoverII}. To the best of our knowledge, the limiting kernel that appears at the hard edge
\begin{equation}
  K_{\hard}(\tau, \sigma; \alpha) = \int^{i \infty}_{-i \infty} \frac{du}{2\pi} \oint_{\contC_{-\infty} \cup \contC_{\infty}} \frac{vdv}{\pi} \frac{\cos(\pi u)}{\cos(\pi v)} \frac{e^{\alpha u^2}}{e^{\alpha v^2}} \frac{P_{v - \frac{1}{2}}(\tau) Q_{u - \frac{1}{2}}(\sigma)}{v - u},
\end{equation}
where $\mathcal{C}_{\infty}$ be a counterclockwise contour in $\{z\in\C : \Re{z} > 0\}$ enclosing $[\frac{1}{2}, \infty)$ and $\mathcal{C}_{-\infty}=-\mathcal{C}_{\infty}$, is new. The limiting kernel at the soft edge
\begin{equation}
  K_{\soft}(\tau, \sigma; \alpha) = \int^{1+i\infty}_{1-i\infty} \frac{du}{2\pi i} \oint_{\mathcal{C}_{-\infty, 0}} \frac{dv}{2\pi i} \frac{\Gamma(v)}{\Gamma(u)} \frac{e^{\alpha u^2-\sigma u}}{e^{\alpha v^2-\tau v}} \frac{1}{u - v},
\end{equation}
where $\mathcal{C}_{-\infty, 0}$ is a counterclockwise contour in $\{z\in\C : \Re{z} < 1\}$ enclosing $(-\infty,0]$, matches the limiting correlation kernel $K_{\crit}(x, y; \gamma)$ defined in \cite[Eq. (1.15)]{Liu-Wang-Wang18} with $x = \tau$, $y = \sigma$ and $\gamma = 2\alpha$. See also \cite[Eq.~(19)]{Akemann-Burda-Kieburg18} and \cite[Eqs.~(109)--(111)]{Akemann-Burda-Kieburg20} for equivalent formulas. The latter is the soft edge limiting kernel for the squared singular values of products of $M$ independent $N \times N$ Ginibre matrices as $M, N \to \infty$ and $M/N \to \gamma$, or more generally of products within broad class of more general right unitary invariant random matrices, see \cite{Ahn23}. The latter also provides an interpretation in terms of the limiting edge statistics of a Brownian motion on $\mathrm{GL}(n,\mathbb C)$. After a residue computation, the limiting bulk kernel
\begin{equation}
  K_{\bulk}(\tau, \sigma; \alpha) = \int^{i \infty}_{-i \infty} \frac{dv}{2\pi} \oint_{\contC_{-\infty} \cup \contC_{\infty}} \frac{du}{2\pi} \frac{\cos(\pi v)}{\cos(\pi u)} \frac{e^{\alpha v^2 - \tau v}}{e^{\alpha u^2 - \sigma u}} \frac{1}{u - v}
\end{equation}
is equal to $e^{(\tau - \sigma)/2} K^{(\bulk)}_{\crit}(\tau - \alpha, \sigma - \alpha; 2\alpha)$, where $K^{(\bulk)}_{\crit}(\xi, \eta; \gamma')$ is defined in \cite[Eq. (3.12)]{Liu-Wang-Wang18} and arises again through products of $M$ independent $N \times N$ Ginibre matrices as $M, N \to \infty$ and $M/N \to \gamma$. The latter was already predicted in \cite[Eq. (13)]{Akemann-Burda-Kieburg18} and \cite[Eq. (79)]{Akemann-Burda-Kieburg20}.

\begin{thm}[Local statistics at crossover II] \label{thm:crossover}
    Take $M_n=\alpha n$ with $\alpha>0$. 
    \begin{enumerate}
    \item \label{enu:thm:crossover_1}
      bulk: let $\eps>0$ be sufficiently small, take $m\in\Z$ with $\eps n < m < (1-\eps)n$, and set $\xiscaled_n = \frac{2m}{n} + \frac{1}{\alpha n} (\log \frac{n + m}{n - m} + \tau)$ and $\etascaled_n = \frac{2m}{n} + \frac{1}{\alpha n} (\log \frac{n + m}{n - m} + \sigma)$, then, up to cocycles,
        \begin{equation} \label{eq:thm:crossover_bulk}
         \frac{1}{\alpha n} \mathbb{K}_{n,\alpha n}(\xiscaled_n,\etascaled_n) \to K_{\bulk}(\tau, \sigma; \alpha), 
         \end{equation}
      \item \label{enu:thm:crossover_2}
        hard edge: set $\xiscaled_n = \frac{\tau}{\alpha n}$ and $\etascaled_n = \frac{\sigma}{\alpha n}$, then, up to cocycles,
         \begin{equation} \label{eq:thm:crossover_hard}
         \frac{1}{\alpha n} \mathbb{K}_{n,\alpha n}(\xiscaled_n,\etascaled_n) \to \sqrt{\sinh \tau \sinh \sigma} K_{\hard}(\cosh\tau,\cosh\sigma; \alpha), 
         \end{equation}
       \item \label{enu:thm:crossover_3}
         soft edge: set $\xiscaled_n=2+\frac{\log(2n)}{\alpha n}+\frac{\tau - \alpha}{\alpha n}$ and $\etascaled_n=2+\frac{\log(2n)}{\alpha n}+\frac{\sigma - \alpha}{\alpha n}$, then, up to cocycles,
            \begin{equation} \label{eq:thm:crossover_soft}
            \frac{1}{\alpha n} \mathbb{K}_{n,\alpha n}(\xiscaled_n,\etascaled_n) \to K_{\soft}(\tau, \sigma; \alpha).  
            \end{equation}
    \end{enumerate}
\end{thm}

The limiting local statistics at crossover I are given below and are similar to those obtained in the metallic regime. The difference is that the limiting density of particles in the bulk becomes non-trivial. We will give a proof in Section \ref{subsec:ballisitic_regime}. 

\begin{thm}[Local statistics at crossover I] \label{thm:ballistic}
    Take $M>0$.
    \begin{enumerate}
        \item \label{enu:thm:ballistic_1}
        bulk: let $r \in \left( 0, r_{\max} \right)$ where $r_{\max} = \sqrt{1 + M^{-1}} + \frac{1}{2M} \log \frac{\sqrt{1 + M^{-1}} + 1}{\sqrt{1 + M^{-1}} - 1}$, and set $\xiscaled_n = 2r + \frac{\pi}{M \Re(z(r)) n}\tau$ and $\etascaled_n = 2 r + \frac{\pi}{M \Re(z(r))n} \sigma$ where $z(r)$ is the unique solution of
        \begin{equation} \label{eq:defn_z+_intro}
          2M z + (i2r M - \pi) + i \log \frac{1 - iz}{1 + iz}=0,
        \end{equation}
        in the fourth quadrant. Then, up to cocycles,
        $$\frac{\pi}{M \Re(z(r))n} \mathbb{K}_{n,M}(\xiscaled_n,\etascaled_n) \to K_{\sin}(\tau,\sigma), $$
        
      \item \label{enu:thm:ballistic_2}
      hard edge: set $\xiscaled_n = \frac{\pi}{M z(0) n} \tau$ and $\etascaled_n = \frac{\pi}{M z(0) n} \sigma$, where $z(0)$ is the unique solution of \eqref{eq:defn_z+_intro} on the positive real line with $r = 0$, then, up to cocycles,
        \begin{equation}
          \frac{\pi}{M z(0) n} \mathbb{K}_{n,M}(\xiscaled_n,\etascaled_n) \to 2\pi^2 \sqrt{\tau \sigma} K_{\Bes, 0}((\pi \tau)^2, (\pi \sigma)^2),
        \end{equation}
        \item \label{enu:thm:ballistic_3}
        soft edge: set $\xiscaled_n = 2 r_{\max} + 2^{\frac{1}{3}} n^{-\frac{2}{3}} M^{-\frac{1}{2}} (1 + M)^{\frac{1}{6}} \tau$ and $\etascaled_n = 2 r_{\max} + 2^{\frac{1}{3}} n^{-\frac{2}{3}} M^{-\frac{1}{2}} (1 + M)^{\frac{1}{6}} \sigma$, then, up to cocycles,
      $$2^{\frac{1}{3}} M^{-\frac{1}{2}} (1 + M)^{\frac{1}{6}} n^{-\frac{2}{3}} \mathbb{K}_{n,M}(\xiscaled_n,\etascaled_n) \to K_{\rm Ai}(\sigma,\tau).$$
    \end{enumerate}
\end{thm}

As discussed in the introduction, it is worth noting that the above implies that the local bulk limits of the transplanted kernel $K^{\ast}_{n, M}(x, y) = 2y\mathbb{K}_{n, M}(x^2, y^2)$ matches those of the kernel $K_n(x, y)$ defined in \cite[Eqs (1.9) \& (1.10)]{Wang23} with $\mathbb{M} = M$. We will give a more detailed comparison in Remark \ref{rem:comparison}. 
\medbreak

Finally, we present the uniform asymptotic expansions of the mean and variance of the conductance in the crossover I and metallic regime. We will prove this in Section \ref{sec:conductance}. Recall that $G$ is defined by \eqref{eq:defn_conductance} as a particular linear statistic of particles $x_1, \dotsc, x_n \in [1,\infty)$ with joint probability density function \eqref{eq:DMPK-density}. 

\begin{thm} \label{thm:cond-mean-var}
  If $n \to \infty$, then uniformly for $M \in [0, n-1]$,
\begin{align}
  \E_{n,M}[G]&=\frac{n}{M+1} + \bigO \left(\frac{M+1}{n}\right), \label{meanexp} \\
  \Var_{n,M}G&=\frac1{15}\left(1 - \frac{6M+1}{(M+1)^6}\right)+O\!\left(\frac{M+1}{n}\right).\label{varianceexp}
\end{align}
\end{thm}

Specializing to crossover I and the metallic regime then gives the following limits.

\begin{cor}\label{limits}
If $M\geq 0$ is fixed, then
\begin{align} \label{eq:crossover_I_limit_conductance}
  \lim_{n\to\infty} \frac{1}{n} \E_{n,M}[G] = {}& \frac1{M+1}, & \lim_{n\to\infty} \Var_{n,M} G = {}& \frac1{15}\left(1-\frac{6M+1}{(M+1)^6}\right).
\end{align}
Furthermore, if $M=M_n\to\infty$ and $M_n/n\to0$ as $n\to\infty$, then the DMPK ensemble \eqref{eq:DMPK-density} admits the universal conductance fluctuations \eqref{eq:UCF}.
\end{cor}

The limiting mean of the conductance in \eqref{eq:crossover_I_limit_conductance} was proven rigorously before in \cite[Thm. 2]{Bachmann-Butz-De_Roeck12} under the name of Ohm's law. The limiting variance of conductance in \eqref{eq:crossover_I_limit_conductance} was predicted in \cite[Eqs. (3.35)]{Beenakker97} and conjectured in \cite[Conj. 2]{Bachmann-Butz-De_Roeck12} with $\beta=2$.

\section{Invariant pseudo-unitary matrices} \label{sec:IPUM}

The goal in this section will be to provide more details on Definitions \ref{def:K-biinvariant-ensemble} and \ref{def:derivative-type-ensemble}, and to prove Proposition \ref{prop:imup-rv}, Theorems \ref{thm:closure-products} and \ref{DIR} and Corollaries \ref{cor:DPMK-prod} and \ref{DIR_DMPK}.

\subsection{Product law} \label{S:IPUM-prod}
Recall that the set of pseudo-unitary matrices $\U(n,n)$, see \eqref{eq:def-U(n,n)}, forms a reductive Lie group with respect to matrix multiplication. Its maximal compact subgroup $K$ is defined in \eqref{eq:defn_subgroup_K}. Since $K$ is compact, it admits a unique Haar probability measure $dk$, which is both left and right invariant.
\medbreak

In order to study products of the $K$-invariant pseudo-unitary random matrices introduced in Definition \ref{def:K-biinvariant-ensemble}, it will be convenient to consider the Cartan, or polar, decomposition of $\U(n,n)$. This decomposition is given by $\U(n,n)=K\mathfrak{a}^+K$, where
$\mathfrak{a}^+=\{a_{\lambda}:\lambda_1\ge\cdots\ge\lambda_n\ge0\}$, and
\begin{equation}\label{eq:a-lambda}
a_{\lambda}=
\begin{pmatrix}
\cosh\Lambda_{\lambda}&\sinh\Lambda_{\lambda}\\
\sinh\Lambda_{\lambda}&\cosh\Lambda_{\lambda}
\end{pmatrix},\quad
\Lambda_{\lambda}=\diag(\lambda_1,\ldots,\lambda_n).
\end{equation}
Although $\U(n,n)$ is not semisimple, the decomposition follows from the corresponding result for $\mathrm{SU}(n,n)$, where $\mathrm{SU}(n,n)=\{g\in\U(n,n):\det g=1\}$ denotes the special pseudo-unitary group, see \cite[Thm. 7.1]{Edelman-Sungwoo23} and \cite[Ch.~IX, Sect.~1]{Helgason01}. Indeed, every $g\in\U(n,n)$ can be written as $g=(zI_{2n})h$ with $|z|=1$ and $h\in\mathrm{SU}(n,n)$, by choosing $z^{2n}=\det g$. Since $zI_{2n}\in K$, this scalar factor can be absorbed into the left $K$-factor. Moreover,
\[
\U(n,n)/K\cong\mathrm{SU}(n,n)/(K\cap\mathrm{SU}(n,n)),
\]
so the radial analysis can be carried out on the latter symmetric space.

This decomposition is analogous to the singular value decomposition of a complex matrix. Indeed, if $g=k_1a_{\lambda}k_2$ with $k_1,k_2\in K$, then its singular values are $e^{\pm\lambda_1},\ldots,e^{\pm\lambda_n}$, since $k_1,k_2$ are unitary and $a_{\lambda}$ is positive definite with these eigenvalues. The ordered vector $\lambda$ is thus uniquely determined by $g$, although the factors $k_1,k_2$ need not be unique. In what follows, we call $\lambda_1,\ldots,\lambda_n$ the radial variables of $g$.
\medbreak

Let $g_1,g_2$ be two independent $K$-invariant pseudo-unitary random matrices whose laws have densities $\rho_1,\rho_2$ with respect to Haar measure $dg$ on $\U(n,n)$, respectively. Their product $g=g_1g_2$ is again a $K$-invariant pseudo-unitary random matrix whose law has a density $\rho_1*\rho_2:\U(n,n)\to\R_{\geq0}$ given, for almost every $g\in\U(n,n)$, by the convolution
\begin{equation}\label{eq:group-convolution}
(\rho_1*\rho_2)(g)=\int_{\U(n,n)}\rho_1(h)\rho_2(h^{-1}g)dh.
\end{equation}
Indeed, the substitution $g=g_1g_2$ in the expectation of a bounded measurable test function, followed by left invariance of Haar measure, gives \eqref{eq:group-convolution}. Non-negativity, unit mass and bi-$K$-invariance are preserved. The standard way to analyze the product ensemble $\rho_1\ast\rho_2$ is by means of the spherical transform of the pair $(G,K)$. We will follow the exposition in \cite[Sect.~2]{Kuijlaars-Roman19} to give an overview of its key properties. A spherical function of $(G,K)$ is a continuous bi-$K$-invariant function $\varphi:G\to\C$ with $\varphi(e)=1$ and
\begin{equation}\label{eq:spherical-product-formula}
\int_K\varphi(gkh)dk=\varphi(g)\varphi(h),\qquad g,h\in G.
\end{equation}
The spherical transform of a bi-$K$-invariant function $f\in L^1(G)$ is the function on the set of bounded spherical functions defined by
\begin{equation}\label{eq:abstract-spherical-transform}
(\mathcal{S}f)(\varphi)=\int_G f(g)\varphi(g^{-1})dg.
\end{equation}
This integral converges absolutely since $\varphi$ is bounded. The spherical transform is injective on the space of bi-$K$-invariant functions in $L^1(G)$, see \cite[Thm.~3.3]{Stanton-Tomas79}. The convolution theorem for the spherical transform says that 
\begin{equation}\label{eq:spherical-multiplicativity}
(\mathcal{S}[f_1*f_2])(\varphi)=(\mathcal{S}f_1)(\varphi)\cdot(\mathcal{S}f_2)(\varphi),
\end{equation}
for all bounded spherical functions $\varphi$, see, e.g., \cite[Ch.~IV, Sect.~1]{Helgason00}. 
\medbreak

In order to give a description of the relevant spherical functions in our setting, we require some intermediate results. We start with the change of variables formula below.

\begin{prop}\label{prop:polar-integration-x}
For every $f\in L^1(\U(n,n))$,
\begin{equation}\label{eq:polar-lambda}
\int_{\U(n,n)} f(g)dg
=2^{n^2}\int_{(0,\infty)^n}\int_K\int_K
f(k_1a_{\lambda}k_2)J(\lambda)
dk_1dk_2d\lambda,
\end{equation}
where
\begin{equation}\label{eq:J-lambda}
J(\lambda)=
\prod_{i<j}(\sinh^2\lambda_i-\sinh^2\lambda_j)^2
\prod_{i=1}^n\sinh(2\lambda_i).
\end{equation}
\end{prop}
\begin{proof}
For the symmetric space $\mathrm{SU}(n,n)/(K\cap\mathrm{SU}(n,n))$, the restricted root system is of type $C_n$, see \cite[Sect.~1]{Hoogenboom82}. The positive restricted roots and multiplicities are
\[
\Sigma^+=\{e_i-e_j,e_i+e_j:i<j\}\cup\{2e_i:1\le i\le n\},\quad m_{e_i-e_j}=m_{e_i+e_j}=2,\quad m_{2e_i}=1.
\]
The polar Jacobian is proportional to
$\prod_{\alpha\in\Sigma^+}|\sinh\alpha(\lambda)|^{m_\alpha}$.
The identity
\[
\sinh(\lambda_i-\lambda_j)\sinh(\lambda_i+\lambda_j)
=\sinh^2\lambda_i-\sinh^2\lambda_j,
\]
then yields \eqref{eq:J-lambda}. Passing from the ordered chamber to the full set only changes the scalar coefficient, which is fixed by our normalization of $dg$.
\end{proof}

A further change of variables leads to an expression that resembles Weyl's integration formula for unitary invariant ensembles on Hermitian matrices, now in terms of the (transplanted) radial variables in $(1,\infty)^n$ rather than the eigenvalues in $\R^n$.

\begin{prop}\label{prop:polar-x-coordinates}
For every $f\in L^1(\U(n,n))$,
\begin{equation}\label{eq:polar-x}
\int_{\U(n,n)} f(g)dg=
\int_{(1,\infty)^n}\int_K\int_K
f(k_1\alpha_{ x}k_2)\Delta_n( x)^2
dk_1dk_2dx,
\end{equation}
where $\alpha_x$ is as in \eqref{eq:a-x}.
\end{prop}
\begin{proof}
The change of variables $x_i\mapsto\cosh(2\lambda_i)=1+2\sinh^2\lambda_i$ gives $dx_i=2\sinh(2\lambda_i)d\lambda_i$ and $x_i-x_j=2(\sinh^2\lambda_i-\sinh^2\lambda_j)$. Consequently
\begin{equation}\label{eq:radial-change-measure}
\Delta_n( x)^2dx
=2^{n^2}J(\lambda)d\lambda.
\end{equation}
It then remains to apply Proposition~\ref{prop:polar-integration-x}.
\end{proof}

The joint density of the unordered (transplanted) radial variables of a $K$-invariant pseudo-unitary random matrix therefore takes the following form.

\begin{prop}\label{prop:radial-density}
Let $g$ be a $K$-invariant random pseudo-unitary matrix whose law has density $\rho$. Then the joint density of the unordered radial variables of $g$ is given by
\begin{equation}\label{eq:radial-density}
\rho_{\rm rv}( x)=\Delta_n( x)^2
\rho(\alpha_{ x}),\quad \text{a.e. } x\in(1,\infty)^n.
\end{equation}
Conversely, every symmetric probability density $p$ on $(1,\infty)^n$
is the unordered radial density for the unique $K$-invariant ensemble defined by $\rho(\alpha_x)=p(x)/\Delta_n(x)^2$ for a.e. $x\in(1,\infty)^n$.
\end{prop}
\begin{proof}
This follows immediately from Proposition \ref{prop:polar-x-coordinates} and the fact that a $K$-invariant ensemble is bi-$K$-invariant.
\end{proof}

The proof of Proposition \ref{prop:imup-rv} is now straightforward.
\begin{proof}[Proof of Proposition \ref{prop:imup-rv}]
    Combine \eqref{eq:radial-density} and \eqref{eq:rho-omega-density}.
\end{proof}

The result below is the most important result of this section. It states that on matrices $\alpha_{x}$, the spherical functions of $(\U(n,n),K)$ can be expressed as determinants of Legendre functions.

\begin{prop}\label{prop:det-spherical-function}
The spherical functions of $(\U(n,n),K)$ can be parametrized by $\varphi_s$ with $s\in\C^n$ where
\begin{equation}\label{eq:spherical-determinant}
\varphi_{s}(\alpha_{ x})
=\frac{c_n\det[P_{(is_j-1)/2}(x_k)]_{j,k=1}^n}
{\Delta_n(s^2)\Delta_n(x)},\qquad x\in(1,\infty)^n,
\end{equation}
and 
\begin{equation*}
    c_n=(-1)^{\frac{n(n-1)}{2}}2^{\frac{3}{2}n(n-1)}\prod_{m=0}^{n-1}(m!)^2.
\end{equation*}
The right hand side of \eqref{eq:spherical-determinant} is interpreted as a confluent limit whenever some $s_k=\pm s_l$ or $x_k=x_l$.
\end{prop}
\begin{proof}
The Berezin-Karpelevich determinant formula in \cite[Thm.~3]{Hoogenboom82} with $k=0$ gives the stated determinant up to a
constant. Differentiating \eqref{eq:Leg} and taking $x=1$ yields
\begin{equation}\label{eq:p-derivative-at-one}
\frac{1}{m!}P_{(is-1)/2}^{(m)}(1)
=\frac{(-1)^m}{2^{3m}(m!)^2}
\prod_{l=0}^{m-1}\bigl(s^2+(2l+1)^2\bigr).
\end{equation}
The confluent limit is therefore
\begin{align}\label{eq:derivative-determinant-at-one}
\lim_{ x\to 1}
\frac{\det[P_{(is_j-1)/2}(x_k)]}{\Delta_n( x)}
=\det\Big[\frac{1}{(k-1)!}P_{(is_j-1)/2}^{(k-1)}(1)\Big]_{j,k=1}^n
=\frac{(-1)^{\frac{n(n-1)}{2}}\Delta_n(s_1^2,\ldots,s_n^2)}
{2^{\frac{3}{2}n(n-1)}\prod_{m=0}^{n-1}(m!)^2}.
\end{align}
The condition $\varphi_{s}(e)=1$ then leads to the stated expression for $c_n$.
\end{proof}

We can now use the functions $\varphi_s$ to build a spherical transform.

\begin{cor}\label{cor:radial-spherical-transform}
Let $f\in L^1(\U(n,n))$ be bi-$K$-invariant. The spherical transform $(\mathcal{S} f)(s) := (\mathcal{S} f)(\varphi_s)$ associated with $\varphi_s$ in \eqref{eq:spherical-determinant} with $s\in\C^n$ is given by
\begin{equation}\label{eq:radial-spherical-transform}
 (\mathcal{S} f)(s)=\frac{c_n}{\Delta_n(s_1^2,\ldots,s_n^2)}
\int_{(1,\infty)^n}f(\alpha_{ x})\Delta_n( x)
\det[P_{(is_j-1)/2}(x_k)]_{j,k=1}^ndx,
\end{equation}
whenever this integral exists (e.g. when $\varphi_s$ is bounded).
\end{cor}
\begin{proof}
Conjugation by $J\in K$ sends $a_{\lambda}$ to
$a_{-\lambda}=a_{\lambda}^{-1}$.
Thus $g^{-1}$ has the same radial coordinates as $g$. It then remains to insert
\eqref{eq:spherical-determinant} into \eqref{eq:polar-x}.
\end{proof}

\subsection{Derivative type ensembles} \label{S:IPUM-DT}

The spherical transform \eqref{eq:radial-spherical-transform} gives rise to a more extensively studied one-dimensional integral transform called the Mehler-Fock transform, see, e.g., \cite[Sect. 1.9]{Glaeske-Prudnikov-Skornik06} for an overview of its properties. We first recall its definition.

\begin{defi}
    The Mehler-Fock transform of a function $f:\R_{\geq 1}\to \C$ is given by
    $$ (\mathcal{P}f)(s) = \int_1^\infty f(x) P_{is-\frac{1}{2}}(x) dx,\qquad s\in\C, $$
    whenever this integral exists.
\end{defi}

Due to the connection to the spherical transform the Mehler-Fock transform has several useful properties similar to, e.g., the Laplace/Fourier transform. Many of these properties are proven in a self-contained way in \cite[Sect. 1.9]{Glaeske-Prudnikov-Skornik06}. We will recall some of them below, see \cite[Thm. 1.9.52, Thm. 1.9.53 \& Prop. 1.9.39]{Glaeske-Prudnikov-Skornik06} in particular. First, recall that $\mathcal{P}f$ is well-defined on $\R$ for $f\in L^1(\R_{\geq 1})$. Second, whenever $f\in L^1(\R_{\geq 1})$ and $\mathcal{P}f\in L^1(\R_{\geq 0};s^{\frac{1}{2}}ds)$, there is an inverse transform, namely
\begin{equation}\label{MF_inv}
    f(x) = \int_0^\infty (\mathcal{P}f)(s) \tanh(\pi s) P_{is-\frac{1}{2}}(x) s ds,\qquad \text{a.e. } x\in\R_{>1}.
\end{equation}

For technical convenience, we will restrict to functions from the subspace $\mathcal{W}(\R_{\geq 1})$ of $L^1(\R_{\geq1})$ defined in \eqref{eq:def-W}. We also recall the differential operator $D$ defined in \eqref{D}. The set $\mathcal{W}(\R_{\geq 1})$ is essentially such that the differentiation property in $ii)$ below is guaranteed to hold. This property will play a crucial role in what follows.

\begin{prop} \label{prop:admissable-weight-prop}
    Suppose that $f\in\mathcal{W}(\R_{\geq 1})$. Then, for all $k\in\N$,
    \begin{itemize}
        \item[$i)$] $(\MF[D^k f])(s)$ converges absolutely for all $s\in\C$ and defines an entire function,
        \item[$ii)$] $(\MF[D^k f])(s)=s^{2k}(\MF f)(s)$ for all $s\in\C$,
        \item[$iii)$] for all $A\geq0$ and $m\in\N$,
        \[
            \sup_{|\Im s|\leq A}
            (1+|s|)^m|(\MF f)(s)|<\infty.
        \]
    \end{itemize}
\end{prop}
\begin{proof}
We first establish that all $D^kf$ and $(D^kf)'$ decay sufficiently fast to apply integration by parts. By \eqref{eq:def-W}, both $D^kf$ and $D^{k+1}f$ have all absolute polynomial moments. Observe that 
\[
D^{k+1}f(x) = \bigl((1-x^2)(D^kf)'(x)\bigr)' - \frac{1}{4} D^k f(x),
\]
and $\lim_{x\to\infty} (D^{k+1}f(x)+\frac{1}{4} D^k f(x)) = 0$ by integrability. Therefore,
\[
    (1-x^2)(D^kf)'(x)
    =-\int_x^\infty\Big(D^{k+1}(t)+\frac14(D^kf)(t)\Big)\,dt.
\]
Hence, for every $m\in\N$, 
\begin{align*}
|(1-x^2)(D^kf)'(x)|\, &\leq x^{-m}\int_x^\infty
    t^m\left(|D^{k+1}f(t)|+\frac14|D^kf(t)|\right)\,dt \\
    &\leq x^{-m}\int_1^\infty
    t^m\left(|Dg(t)|+\frac14|g(t)|\right)\,dt.
\end{align*}
The function $(D^kf)'$ therefore decreases faster than all $x\mapsto x^m$. Since $g(x)=-\int_x^\infty (D^kf)'(t)\,dt$, the same then holds for $D^kf$.

Recall from Proposition \ref{prop:Leg-geom-est} that
\[
    |P_{is-\frac12}(x)|
    \leq C_A x^{A+\frac12},
    \qquad x\geq1,\quad |\Im s|\leq A.
\]
It then follows from \eqref{eq:def-W} that
\[
    (\MF[D^k f])(s)
    =\int_1^\infty (D^k f)(x)P_{is-\frac12}(x)\,dx,
\]
converges absolutely and locally uniformly in $s\in\C$. Furthermore, the integral representation \eqref{eq:Leg-int-repr} shows that
$s\mapsto P_{is-\frac12}(x)$ is entire for every fixed $x\geq1$, hence the same holds for $\MF[D^k f]$. This proves $i)$.

To prove $ii)$, let $s\in\C$. Similarly as before, now using the fact that $DP_{is-\frac12}=s^2P_{is-\frac12}$ and $|P_{is-\frac{1}{2}}(1)|\,\leq 2^{A+\frac{1}{2}}$, we obtain
\[
    (1-x^2)P_{is-\frac12}'(x)
    = \Big(s^2+\frac14\Big)
    \int_1^x P_{is-\frac12}(t)\,dt.
\]
Together with the previous bound, this shows that $P_{is-\frac12}$ and its derivative grow at most polynomially at infinity. Integration by parts therefore yields
\[
    \begin{aligned}
        &\int_1^\infty (D^{j+1}f)(x)P_{is-\frac12}(x)\,dx
        -\int_1^\infty (D^jf)(x)(DP_{is-\frac12})(x)\,dx\\
        &\qquad=
        \left[
        (1-x^2)\left(
        (D^jf)'(x)P_{is-\frac12}(x)
        -(D^jf)(x)P_{is-\frac12}'(x)
        \right)
        \right]_1^\infty.
    \end{aligned}
\]
The boundary term at $1$ vanishes by smoothness and the factor $1-x^2$, while the boundary term at infinity vanishes by the rapid decay established above. Hence,
\[
    (\MF[D^{j+1}f])(s)
    =s^2(\MF[D^j f])(s).
\]
Consecutive use of this identity then proves $ii)$.

Finally, let $A\geq0$ and $m\in\N$, and take $k\in\N$ such that $2k\geq m$. By $ii)$ and the bound on the Legendre function,
\[
    \begin{aligned}
        |s|^{2k}|(\MF f)(s)|
        =|(\MF[D^k f])(s)|
        \leq
        2^{A+\frac12}\int_1^\infty
        |(D^k f)(x)|x^{A+\frac12}\,dx,
        \qquad |\Im s|\leq A.
    \end{aligned}
\]
The right hand side is finite by \eqref{eq:def-W}, which then proves $iii)$.
\end{proof}

The previous result can be used to show that the normalization constant $Z_{n}(\om)$ of the derivative type ensemble in \eqref{eq:rho-omega-density} admits an explicit representation in terms of the Mehler-Fock transform of $\om$.

\begin{prop}  \label{prop:Z-exact}
In the setting of Definition \ref{def:derivative-type-ensemble}, the constant $Z_{n}(\om)$ is given by
\begin{equation}  \label{eq:Z-exact}
Z_{n}(\om)=(-2)^{\frac{n(n-1)}{2}}n!\prod_{k=0}^{n-1}\left[(k!)^2(\MF\omega)\bigl(-i(k+{\textstyle\frac12})\bigr)\right].
\end{equation}
\end{prop}
\begin{proof}
By \eqref{eq:radial-density} and \eqref{eq:rho-omega-density}, the normalization $\int_{\U(n,n)}\rho dg=1$ gives
\[
Z_{n}(\om)
=\int_{(1,\infty)^n}\Delta_n(x)
  \det[(D^{j-1}\omega)(x_k)]_{j,k=1}^n\,dx.
\]
The leading coefficient of the Legendre polynomial $P_k$ is $(2k)!/(2^k(k!)^2)$, see \eqref{eq:Leg}. Hence, by elementary row operations,
\[
\Delta_n(x)
=\left(\prod_{k=0}^{n-1}\frac{2^k(k!)^2}{(2k)!}\right)
  \det[P_{j-1}(x_k)]_{j,k=1}^n.
\]
It then follows from Andr\'eief's identity that
\[
Z_{n}(\om)
=n!\left(\prod_{k=0}^{n-1}\frac{2^k(k!)^2}{(2k)!}\right)
  \det\Big[
    \int_1^\infty P_j(x)(D^k\omega)(x)\,dx
  \Big]_{j,k=0}^{n-1}.
\]
Proposition~\ref{prop:admissable-weight-prop}, applied at
$s=-i(j+\frac12)$, gives
\[
    \begin{aligned}
        \int_1^\infty P_j(x)(D^k\omega)(x)\,dx
        =\bigl(-(j+{\textstyle\frac12})^2\bigr)^k
        (\MF\omega)\bigl(-i(j+{\textstyle\frac12})\bigr).
    \end{aligned}
\]
Factoring the values of the Mehler-Fock transform out of the rows leaves the Vandermonde determinant
\[
\begin{aligned}
\det\Big[\bigl(-(j+{\textstyle\frac12})^2\bigr)^k\Big]_{j,k=0}^{n-1}
=\prod_{0\leq j<k\leq n-1}
  \left((j+{\textstyle\frac12})^2-(k+{\textstyle\frac12})^2\right).
\end{aligned}
\]
Using the identity $(j+{\textstyle\frac12})^2-(k+{\textstyle\frac12})^2 = (j+k+1)(j-k)$, we can write this as
\[
\det\Big[\bigl(-(j+{\textstyle\frac12})^2\bigr)^k\Big]_{j,k=0}^{n-1} = (-1)^{\frac{n(n-1)}{2}}\prod_{k=0}^{n-1}(2k)!.
\]
Combining these results then yields \eqref{eq:Z-exact}.
\end{proof}

The following characterization of derivative type ensembles will play an essential role in identifying the convolution of such ensembles, as it essentially reduces it to a one-dimensional problem. 

\begin{prop}\label{thm:derivative-transform}
Suppose that a density $\rho$ defines a $K$-invariant ensemble $\mu$ on $\U(n,n)$ and let $\omega\in\mathcal{W}(\R_{\geq 1})$. Then the following are equivalent:
\begin{enumerate}
\item[$i)$] $\mu$ is of derivative type with respect to $\omega$,
\item[$ii)$] the spherical transform of $\rho$ factorizes as
\begin{equation}\label{eq:rho-omega-sharp}
(\mathcal{S}\rho)(s)=
\frac{\prod_{j=1}^n(\MF\omega)(s_j/2)}
{\prod_{k=0}^{n-1}(\MF\omega)\bigl(-i(k+\frac12)\bigr)},\qquad s\in\R^n.
\end{equation}
\end{enumerate}
\end{prop}
\begin{proof}
We first prove the implication from $i)$ to $ii)$. Proposition~\ref{prop:Z-exact} and $Z_n(\om)\neq0$ imply that the denominator in \eqref{eq:rho-omega-sharp} is non-zero.
For $s\in\R^n$ with pairwise distinct $s_1^2,\ldots,s_n^2$, \eqref{eq:radial-spherical-transform},
\eqref{eq:rho-omega-density}, and Andr\'eief's identity give
\[
        (\mathcal{S}\rho)(s)
        =\frac{n!c_n}{Z_n(\om)\Delta_n(s^2)}
        \det\Big[
            \int_1^\infty (D^{j-1}\omega)(x)
            P_{(is_k-1)/2}(x)\,dx
        \Big]_{j,k=1}^n.
\]
Hence, by Proposition~\ref{prop:admissable-weight-prop},
\[
        (\mathcal{S}\rho)(s)
        =\frac{n!c_n}{Z_n(\om)\Delta_n(s^2)}
        \det\Big[
            \Big(\frac{s_k^2}{4}\Big)^{j-1}
            (\MF\omega)(s_k/2)
        \Big]_{j,k=1}^n
        =\frac{n!c_n}{2^{n(n-1)}Z_n(\om)}
        \prod_{k=1}^n(\MF\omega)(s_k/2).
\]
Substituting the value of $c_n$ and using \eqref{eq:Z-exact} proves \eqref{eq:rho-omega-sharp} for these parameters. Since both sides are continuous on $\R^n$, the equality extends to all $s\in\R^n$, proving $ii)$.

We now prove the implication from $ii)$ to $i)$. Consider the bi-$K$-invariant function $\widetilde\rho$ determined by
\[
    \widetilde\rho(\alpha_x)
    =\frac{\det[(D^{j-1}\omega)(x_k)]_{j,k=1}^n}
    {Z_n(\om)\Delta_n(x)},
    \qquad x\in(1,\infty)^n,\quad \Delta_n(x)\neq0,
\]
where $Z_n(\om)$ is defined by the right hand side of \eqref{eq:Z-exact} and is non-zero because the denominator in \eqref{eq:rho-omega-sharp} is non-zero. Whenever $\Delta_n(x)=0$, we define $\widetilde\rho(\alpha_x)$ to be equal to zero. Evidently, by \eqref{eq:polar-x} and \eqref{eq:def-W}, we have $\widetilde\rho\in L^1(\U(n,n))$. The preceding transform calculation therefore applies and yields
\[
    (\mathcal{S}\widetilde\rho)(s)
    =\frac{\prod_{j=1}^n(\MF\omega)(s_j/2)}
    {\prod_{k=0}^{n-1}(\MF\omega)(-i(k+\frac12))}
    =(\mathcal{S}\rho)(s),
    \qquad s\in\R^n.
\]
Injectivity of the spherical transform then gives $\widetilde\rho=\rho$ almost everywhere on $\U(n,n)$. Hence $\rho$ satisfies \eqref{eq:rho-omega-density} with this $\omega$ and nonzero $Z_n(\om)$, proving $i)$.
\end{proof}

The one-dimensional version of the convolution in \eqref{eq:group-convolution}, which should factorize the Mehler-Fock transform, was already identified in \cite[Lem.~1.9.10]{Glaeske-Prudnikov-Skornik06} in a self-contained way (without mentioning the underlying spherical transform). For $x,y,z>1$, denote
\begin{equation}\label{eq:MF-endpoints}
\kappa_\pm(y,z)=yz\pm\sqrt{(y^2-1)(z^2-1)},
\end{equation}
and define
\begin{equation}\label{eq:MF-kernel-corrected}
\mathsf K(x,y,z)=\begin{cases}
\displaystyle\frac1{\pi\sqrt{2xyz+1-x^2-y^2-z^2}},
&\kappa_-(y,z)<x<\kappa_+(y,z),\\[4pt]
0,&\text{otherwise}.
\end{cases}
\end{equation}
Then, the authors showed that
\begin{equation}\label{eq:MF-product-formula}
P_{is-\frac{1}{2}}(y)P_{is-\frac{1}{2}}(z)
=\int_1^\infty\mathsf K(x,y,z)P_{is-\frac{1}{2}}(x)dx,\qquad s\in\C.
\end{equation}
This identity corresponds to \eqref{eq:spherical-product-formula} for $n=1$. It leads to the definition below.

\begin{defi}\label{def:MF-convolution}
Given $\omega_1,\omega_2\in L^1(\R_{\geq 1})$, we define
\begin{equation}\label{eq:MF-convolution}
(\omega_1*_{\MF}\omega_2)(x)=
\int_1^\infty\int_1^\infty\mathsf K(x,y,z)
\omega_1(y)\omega_2(z)dydz,\qquad x\in\R_{\geq 1},
\end{equation}
whenever this integral exists.
\end{defi}

The Mehler-Fock transform of $\omega_1*_{\MF}\omega_2$ has the following multiplicative property, which can also be anticipated from \eqref{eq:spherical-multiplicativity}.

\begin{prop}\label{prop:transform-multiplicativity}
If $\omega_1,\omega_2\in L^1(\R_{\geq 1})$, then $\omega_1\ast_{\MF}\omega_2\in L^1(\R_{\geq 1})$ and
\begin{equation}\label{eq:MF-convolution-transform}
\MF[\omega_1*_{\MF}\omega_2](s)
=(\MF\omega_1)(s)(\MF\omega_2)(s),\qquad s\in \R.
\end{equation}
\end{prop}
\begin{proof}
Note that
$2xyz+1-x^2-y^2-z^2
=(\kappa_+(y,z)-x)(x-\kappa_-(y,z)).
$
The change of variables $x\mapsto yz+\sqrt{(y^2-1)(z^2-1)}\cos\theta$ therefore gives
\begin{equation}\label{eq:MF-kernel-mass}
\int_1^\infty\mathsf K(x,y,z)dx
=\frac1\pi\int_{\kappa_-(y,z)}^{\kappa_+(y,z)}
\frac{dx}{\sqrt{(\kappa_+(y,z)-x)(x-\kappa_-(y,z))}}
=\frac1\pi\int_0^\pi d\theta=1.
\end{equation}
It then follows from Tonelli's theorem that
\[
\int_1^\infty\int_1^\infty\int_1^\infty
\mathsf K(x,y,z)|\omega_1(y)\omega_2(z)|
dydzdx
=\|\omega_1\|_1\|\omega_2\|_1,
\]
and therefore $\omega_1\ast_{\MF}\omega_2\in L^1(\R_{\geq 1})$. Moreover, by Fubini's theorem and \eqref{eq:MF-product-formula}
\begin{align*}
\MF[\omega_1*_{\MF}\omega_2](s)
=\left(\int_1^\infty\omega_1(y)P_{is-\frac{1}{2}}(y)dy\right)
\left(\int_1^\infty\omega_2(z)P_{is-\frac{1}{2}}(z)dz\right),\qquad s\in\R.
\end{align*}
This is exactly \eqref{eq:MF-convolution-transform}.
\end{proof}

We can lift the previous property to functions from the space $\mathcal{W}(\R_{\geq1})$, which also allows us to extend \eqref{eq:MF-convolution-transform} to $s\in\R$.

\begin{prop} \label{prop:transform-multiplicativity-C}
If $\omega_1,\omega_2\in\mathcal{W}(\R_{\geq1})$, then $\omega_1\ast_{\MF}\omega_2\in\mathcal{W}(\R_{\geq1})$. Moreover,
\begin{equation}
\MF[\omega_1*_{\MF}\omega_2](s)
=(\MF\omega_1)(s)(\MF\omega_2)(s),
\qquad s\in\C.
\end{equation}
\end{prop}
\begin{proof}
Denote $\omega=\omega_1\ast_{\MF}\omega_2$. By Proposition~\ref{prop:transform-multiplicativity},
$\omega\in L^1(\R_{\geq1})$ and
\[
    (\MF\omega)(s)
    =(\MF\omega_1)(s)(\MF\omega_2)(s),
    \qquad s\in\R.
\]
Proposition~\ref{prop:admissable-weight-prop}, applied to $\omega_1,\omega_2\in\mathcal{W}(\R_{\geq1})$, allows us to use the inversion formula \eqref{MF_inv} to obtain
\[
    \omega(x)
    =\int_0^\infty
    (\MF\omega_1)(s)(\MF\omega_2)(s)
    P_{is-\frac12}(x)\tanh(\pi s)s\,ds,
    \qquad \text{a.e. }x>1.
\]
The bound in Proposition~\ref{prop:Leg-der-s-est} justifies differentiation under the integral to every order, hence the right hand side defines a function in $C^{\infty}(\R_{\geq 1})$. This smooth representative agrees with the left hand side for every $x>1$. In order to show this, it suffices to show that $\om\in C^{0}(\R_{\geq 1})$. Symmetry of $\mathsf K$ and the substitution
$z=xy+\sqrt{(x^2-1)(y^2-1)}\cos\theta$ give
\[
    \omega(x)
    =\frac1\pi\int_1^\infty\omega_1(y)
    \int_0^\pi
    \omega_2\left(
    xy+\sqrt{(x^2-1)(y^2-1)}\cos\theta
    \right)\,d\theta\,dy.
\]
The decay established in the proof of Proposition~\ref{prop:admissable-weight-prop} implies that
$\omega_2$ is bounded. Since $\omega_1\in L^1(\R_{\geq 1})$, dominated convergence then shows that $\om\in C^{0}(\R_{\geq 1})$.

Since $DP_{is-\frac12}=s^2P_{is-\frac12}$, after interchanging differentiations and integration, we obtain
\[
    (D^k\omega)(x)
    =\int_0^\infty
    s^{2k}(\MF\omega_1)(s)(\MF\omega_2)(s)
    P_{is-\frac12}(x)\tanh(\pi s)s\,ds.
\]
On the other hand, Proposition~\ref{prop:admissable-weight-prop} gives
\[
    (\MF[D^k\omega_1])(s)
    =s^{2k}(\MF\omega_1)(s),\qquad s\in\C.
\]
Applying Proposition~\ref{prop:transform-multiplicativity} and \eqref{MF_inv} to $(D^k\omega_1)\ast_{\MF}\omega_2$ therefore yields
\[
    D^k\omega=(D^k\omega_1)\ast_{\MF}\omega_2,
    \qquad\text{a.e. on }(1,\infty).
\]

On the support of $\mathsf K$, we have
    $x<yz+\sqrt{(y^2-1)(z^2-1)}\leq2yz$.
Hence Tonelli's theorem and \eqref{eq:MF-kernel-mass} give, for all $k,m\in\N$,
\[
    \begin{aligned}
        \int_1^\infty |(D^k\omega)(x)|x^m\,dx
        &\leq
        \int_1^\infty\int_1^\infty
        |(D^k\omega_1)(y)\omega_2(z)|
        \left(\int_1^\infty x^m\mathsf K(x,y,z)\,dx\right)
        dy\,dz\\
        &\leq
        2^m
        \left(\int_1^\infty |(D^k\omega_1)(y)|y^m\,dy\right)
        \left(\int_1^\infty |\omega_2(z)|z^m\,dz\right).
    \end{aligned}
\]
Together with smoothness, this proves that $\omega\in\mathcal W(\R_{\geq1})$.

Finally, Proposition~\ref{prop:admissable-weight-prop} shows that $\MF\omega$, $\MF\omega_1$, and $\MF\omega_2$ are entire. The identity theorem therefore extends the factorization already established on $\R$ to all of $\C$.
\end{proof}

We are now ready to prove the main result of this section, which is Theorem \ref{thm:closure-products}:

\begin{proof}[Proof of Theorem \ref{thm:closure-products}]
Let $\rho_1,\rho_2$ denote the densities of the laws of $g_1,g_2$ respectively. By independence and \eqref{eq:group-convolution}, the product $g_1g_2$ has the bi-$K$-invariant probability density
$\rho_1*\rho_2$ on $\U(n,n)$. Denote $\omega=\omega_1*_{\MF}\omega_2$, then by Proposition \ref{prop:transform-multiplicativity-C}, $\omega\in\mathcal W(\R_{\geq1})$ and 
\[
\MF\om(s)
=(\MF\omega_1)(s)(\MF\omega_2)(s),
\qquad s\in\C.
\]
By \eqref{eq:Z-exact}, for both $j = 1, 2$,
\[
(\MF\omega_j)\bigl(-i(k+\tfrac12)\bigr)\neq0,
\qquad k=0,\ldots,n-1,
\]
hence the same holds for $\MF\om$. Furthermore, it follows from \eqref{eq:spherical-multiplicativity}, and the implication from $i)$ to $ii)$ in Proposition~\ref{thm:derivative-transform} that, for all $s\in\R^n$,
\[
\begin{aligned}
\mathcal{S}[\rho_1*\rho_2](s)
&=
 \frac{\prod_{j=1}^n(\MF\omega_1)(s_j/2)}
      {\prod_{k=0}^{n-1}(\MF\omega_1)(-i(k+\frac12))}\frac{\prod_{j=1}^n(\MF\omega_2)(s_j/2)}
      {\prod_{k=0}^{n-1}(\MF\omega_2)(-i(k+\frac12))}
      =\frac{\prod_{j=1}^n(\MF\omega)(s_j/2)}
       {\prod_{k=0}^{n-1}(\MF\omega)(-i(k+\frac12))}.
\end{aligned}
\]
It then remains to use the implication from $ii)$ to $i)$ in
Proposition~\ref{thm:derivative-transform}.
\end{proof}

\subsection{DMPK ensemble} \label{S:IPUM-DMPK}

Recall that the DMPK ensemble is defined in \eqref{eq:DMPK-density} in terms of the functions $w_{1, M},\dots w_{n,M}$ defined in \eqref{w_j}. We first prove that $w_{1,M}\in\mathcal{W}(\R_{\geq 1})$.

\begin{prop} \label{prop:gaussian-weights-admissible}
    Let $M>0$. Then $w_{1,M}\in\mathcal{W}(\R_{\geq 1})$ and
        $$ \MF w_{1,M}(s) = e^{-\frac{M}{n} s^2},\qquad s\in\C. $$
\end{prop}
\begin{proof}
The bound in Proposition~\ref{prop:Leg-der-s-est} and Gaussian decay justify differentiation under the integral to every order, hence $w_{1,M}\in C^{\infty}(\R_{\geq 1})$. Using $DP_{is-\frac12}=s^2P_{is-\frac12}$ and the connection
formula \cite[Eq.~14.20.6]{Boisvert-Clark-Lozier-Olver10}, we obtain
\begin{align*}
(D^kw_{1,M})(x)
=-\int_{\R}e^{-\frac{M}{n} s^2}s^{2k+1}Q_{is-\frac12}(x)
\frac{ds}{\pi i},
\qquad x>1.
\end{align*}
Let $A>0$ and perform the change of variables $s\mapsto t-iA$ to write this as
\[
(D^kw_{1,M})(x)
=-\int_{\R}e^{-\frac{M}{n}(t-iA)^2}(t-iA)^{2k+1}
Q_{A+it-\frac12}(x)\frac{dt}{\pi i}.
\]
Proposition \ref{prop:Leg-Q-geom-est} and the identity $|e^{-\frac{M}{n}(t-iA)^2}|=e^{\frac{M}{n} A^2}e^{-\frac{M}{n} t^2}$ then imply that
\[
|(D^kw_{1,M})(x)|
\leq C_{M,n,k,A}x^{-A-\frac12},
\qquad x\geq2.
\]
Given $m\in\N$, taking $A>m+\frac12$ and using smoothness on $[1,2]$ therefore yields
\[
\int_1^\infty |(D^kw_{1,M})(x)|x^m dx<\infty.
\]
This proves that $w_{1,M}\in\mathcal W(\R_{\geq1})$. In particular, by Proposition \ref{prop:admissable-weight-prop}, $\MF w_{1,M}$ is entire. The inversion formula \eqref{MF_inv} then implies that
\[
(\MF w_{1,M})(s)=e^{-\frac{M}{n} s^2},
\qquad\text{a.e. }s>0.
\]
Since both sides define continuous functions, this extends to all $s>0$. The identity theorem then extends it to all $s\in\C$.
\end{proof}

The previous result allows us to show that the DMPK ensemble admits the derivative type structure in \eqref{eq:radial-density-DT}.

\begin{prop} \label{prop:gaussian-derivative-type}
    The probability density in \eqref{eq:DMPK-density} can be written as
    \begin{equation} \label{E_DT}
        \mathfrak{p}_{n,M}(x) = \frac{1}{Z_{n}(w_{1,M})} \prod_{1\leq j<k\leq n} (x_k-x_j) \, \det[(D^{j-1}w_{1,M})(x_k)]_{j,k=1}^n,\quad x\in (1,\infty)^n,
    \end{equation}
    where $D$ is as in \eqref{D}, $w_{1,M}$ is as in \eqref{w_j} and $Z_{n}(w_{1,M})$ is as in \eqref{eq:Z-exact}.
\end{prop}
\begin{proof}
Similarly as in the proof of Proposition \ref{prop:gaussian-weights-admissible}, we obtain
\[
(D^{j-1}w_{1,M})(x)
=\int_0^\infty e^{-\frac{M}{n}s^2}\tanh(\pi s)
s^{2j-1}P_{is-\frac12}(x)ds,
\qquad j=1,\ldots,n,\quad x>1.
\]
Hence, by definition, $w_{j,M}=D^{j-1}w_{1,M}$. Substituting these identities into \eqref{eq:DMPK-density} gives the determinant
in \eqref{E_DT}. Moreover, $w_{1,M}\in\mathcal W(\R_{\geq1})$ by
Proposition~\ref{prop:gaussian-weights-admissible}, so Proposition~\ref{prop:Z-exact} identifies the normalizing constant with $Z_n(w_{1,M})$ as given in \eqref{eq:Z-exact}. This proves \eqref{E_DT}.
\end{proof}

\begin{rem}
    This result was already predicted in \cite[Eq. (10)]{Frahm95}.
\end{rem}

We are now ready to prove Corollary \ref{cor:DPMK-prod}.

\begin{proof}[Proof of Corollary \ref{cor:DPMK-prod}]
Propositions \ref{prop:imup-rv} and \ref{prop:gaussian-derivative-type} allow us to conclude that the DMPK ensemble models the radial variables of a $K$-invariant random pseudo-unitary matrix $\mathfrak{g}$ sampled with respect to the derivative type with respect to $w_{1,M}$. It follows from Proposition \ref{prop:gaussian-weights-admissible} that
    $$ \MF w_{1,M}(s) = e^{-\frac{M}{n}s^2} = e^{-\frac{1}{n} s^2} \dotsm e^{-\frac{1}{n} s^2} = \MF w_{1,1}(s) \dotsm \MF w_{1,1}(s), \quad s\in\C, $$
hence, whenever $M\in\N$, by Proposition \ref{prop:transform-multiplicativity-C} and injectivity of the Mehler-Fock transform, the function $w_{1,M}$ admits the decomposition
    $$ w_{1,M} = w_{1,1} \ast_{\MF} \dots \ast_{\MF} w_{1,1}, $$
into $M$ factors $w_{1,1}$. The desired result then follows after consecutive use of Theorem~\ref{thm:closure-products}.
\end{proof}

\subsection{Double integral representations} \label{S:IPUM-DIR}

Similarly as for the polynomial ensembles of derivative type in \cite{Forster-Kieburg-Kosters21}, the properties of the Mehler-Fock transform allow us to obtain a double integral representation for the correlation kernel \eqref{K} for the joint density of the radial variables of $K$-invariant random pseudo-unitary matrices. Our main result will be Theorem \ref{DIR} from which Corollary \ref{DIR_DMPK} will follow. In order to prove that Theorem \ref{DIR} holds, we require appropriate integral representations for the functions in the underlying biorthogonal system.
 \medbreak

We will focus first on the polynomials in the biorthogonal system. Conveniently, they admit integral representations in terms of the kernel $(x,s)\mapsto P_{is-\frac{1}{2}}(x)$ of the Mehler-Fock transform by exploiting the connection to the Legendre polynomials.

\begin{prop} \label{p_int}
    Let $\mathcal{C}_{-}$ be a counterclockwise contour enclosing the zeros of $(\frac{1}{2}-it)_{n+1}$ once without enclosing any of the zeros of $(\frac{1}{2}+it)_{n+1}$. Then following are equivalent:
    \begin{itemize}
        \item[$i)$] $p$ is a polynomial of degree at most $n$,
        \item[$ii)$] there exists a polynomial $\hat{p}$ of degree at most $n$, such that
        \begin{equation} \label{eq:p_n_contour_integral}
        p(x) = \oint_{\mathcal{C}_{-}} \frac{\hat{p}(t) P_{it-\frac{1}{2}}(x)}{(\frac{1}{2}-it)_{n+1}(\frac{1}{2}+it)_{n+1}} \frac{dt}{2\pi i}.
        \end{equation}
    \end{itemize}
   Furthermore, $\hat{p}$ may be replaced by any analytic function defined on and inside of $\mathcal{C}_{-}$ that takes the same values at the zeros of $(\frac{1}{2}-it)_{n+1}$.
\end{prop}
\begin{proof}
For fixed $x>1$, the function $t\mapsto P_{it-\frac12}(x)$ is entire, hence the only possible poles of the integrand inside $\mathcal{C}_{-}$ are the simple poles at $t_k=-i(k+\frac12)$ for $k=0,\dots,n$. Since
\[
\frac{d}{dt}({\textstyle\frac12}-it)_{n+1}\Big|_{t=t_k}
=-i(-1)^k k!(n-k)!,
\qquad
({\textstyle\frac12}+it_k)_{n+1}=\frac{(n+k+1)!}{k!},
\]
and $P_{it_k-\frac12}(x)=P_k(x)$, the residue theorem gives
\begin{equation}\label{eq:p_n_contour_residues}
\oint_{\mathcal{C}_{-}}
\frac{\hat p(t)P_{it-\frac12}(x)}
{(\frac12-it)_{n+1}(\frac12+it)_{n+1}}
\frac{dt}{2\pi i}
=i\sum_{k=0}^n
\frac{(-1)^k\hat p(t_k)}{(n-k)!(n+k+1)!}\,P_k(x).
\end{equation}
In particular, the right-hand side is a polynomial of degree at most $n$, proving the implication from $ii)$ to $i)$.

For the implication from $i)$ to $ii)$, suppose that $\deg p\le n$. Since $P_0,\ldots,P_n$ form a basis of the polynomials of degree at most $n$, we can write
\[
p(x)=\sum_{k=0}^n a_kP_k(x).
\]
The nodes $t_0,\ldots,t_n$ are distinct, so polynomial interpolation yields a unique polynomial $\hat p$ of degree at most $n$ satisfying
\[
\hat p(t_k)=-i(-1)^k(n-k)!(n+k+1)!\,a_k,
\qquad k=0,\ldots,n.
\]
Substituting these values into \eqref{eq:p_n_contour_residues} gives \eqref{eq:p_n_contour_integral}, proving $ii)$.

Finally, the residue calculation applies to any function analytic on a neighborhood of the contour and its interior. By \eqref{eq:p_n_contour_residues}, its contour integral depends only on its values at $t=t_k$ for $k=0,\dots,n$. This proves the last assertion.
\end{proof}

From \eqref{eq:DMPK-density} and \eqref{w_j}, it is straightforward to see that for every function $q \in \spn{D^{j-1}w_{1,M}}_{j=1}^n$, there exists a polynomial $q_n$ of degree at most $n-1$ such that
    $$ q(x) = \int_0^\infty \hat{q}_n(s^2) e^{-\frac{M s^2}{n}} \tanh(\pi s) P_{is-\frac{1}{2}}(x) sds. $$
However, due to the factor $\tanh(\pi s)$, such representations won't be convenient for the asymptotic analysis that we want to carry out (indicated already in the proof of Proposition \ref{prop:gaussian-weights-admissible}). In order to obtain suitable integral representations, we need to establish a different way to invert the Mehler-Fock transform. It turns out that we have the following alternative inversion law.

\begin{prop} \label{P_inv_mod}
    Suppose that $f\in L^1(\R_{\geq 1})$ and $\mathcal{P}f\in L^1(\R;|s|^{\frac{1}{2}}ds)$. Then, 
    \begin{equation} \label{eq:P_inv_mod}
        f(x) = -\int_{\R} (\mathcal{P}f)(s) Q_{is - \frac{1}{2}}(x) \frac{sds}{\pi i} ,\qquad \text{ a.e. } x\in\R_{>1}.
    \end{equation}
\end{prop}
\begin{proof}
It follows from \cite[Eq. 14.9.12]{Boisvert-Clark-Lozier-Olver10} that
\begin{equation} \label{eq:Q-Q=P}
  Q_{-is - \frac{1}{2}}(x) = Q_{is - \frac{1}{2}}(x) + \pi i \tanh(\pi s) P_{is - \frac{1}{2}}(x).
\end{equation}
By the original inversion law \eqref{MF_inv}, we can therefore write
$$ f(x) = \int_{0}^{\infty} (\mathcal{P}f)(s) (Q_{-is-\frac{1}{2}}(x)-Q_{is - \frac{1}{2}}(x)) \frac{sds}{\pi i} ,\qquad \text{ a.e. } x\in\R_{>1}.$$
Both parts converge due to \eqref{eq:Q_nu(x)_hard} and \eqref{eq:K0_asy}. It then remains to note that
$$\int_{0}^{\infty} (\mathcal{P}f)(s) Q_{-is-\frac{1}{2}}(x) \frac{sds}{\pi i} = -\int_{-\infty}^{0} (\mathcal{P}f)(-s) Q_{is-\frac{1}{2}}(x) \frac{sds}{\pi i},  $$
and to use the symmetry $(\mathcal{P}f)(s) = (\mathcal{P}f)(-s) $ for $s\in\R$.  
\end{proof}

The above leads to the following integral representations for functions in $\spn{D^{j-1} \om}_{j=1}^n$.

\begin{prop}\label{q_int}
Suppose that $q=\hat q(D)\om$ for some $\hat q\in\R[s]$ and
$\om\in\mathcal W(\R_{\geq1})$. Then
\begin{equation} \label{eq:q_int}
q(x)=-\int_{\R}\hat q(s^2)(\mathcal P\om)(s)
Q_{is-\frac12}(x)\frac{sds}{\pi i},\qquad x>1.
\end{equation}
\end{prop}
\begin{proof}
Evidently, $q\in\mathcal W(\R_{\geq1})$. Hence, by Proposition~\ref{prop:admissable-weight-prop},
\[
(\MF q)(s)=\hat q(s^2)(\MF\om)(s),
\qquad s\in\C.
\]
Moreover, $\MF q$ decays faster than $s\mapsto |s|^{-m}$ on $\R$ for all $m\in\N$. Thus, in particular, $\MF q\in L^1(\R;|s|^{\frac{1}{2}}ds)$. Proposition~\ref{P_inv_mod}
then gives the desired result for a.e. $x>1$. In order to prove it for all $x>1$, it suffices to show that the right hand side of \eqref{eq:q_int} defines a continuous function on $\R_{>1}$. This follows from Proposition \ref{prop:Leg-Q-cont} and dominated convergence, which imply continuity on $\R_{\geq a}$ for all $a>1$.
\end{proof}

Proposition \ref{p_int} and \ref{q_int} imply that the correlation kernel \eqref{K} for \eqref{eq:DMPK-density} has the following double integral representation
\begin{equation}
  K_n(x,y) = - \int_{\R} \oint_{\mathcal{C}_{-}} \left(\sum_{k=0}^{n-1}\frac{\hat{p}_k(t) \hat{q}_k(s^2)}{(\frac{1}{2}-it)_{k+1}(\frac{1}{2}+it)_{k+1}} \right) (\mathcal{P}\om)(s) P_{it-\frac{1}{2}}(x) Q_{is-\frac{1}{2}}(y) \frac{dt}{2\pi i} \frac{s ds}{\pi i},
\end{equation}
in terms of and polynomials $\hat{p}_0,\dots,\hat{p}_{n-1}$ and $\hat{q}_0,\dots,\hat{q}_{n-1}$. Generally, it is not clear whether the double integral representation further simplifies, e.g., for the biorthogonal system that was already derived in \cite[Eq. (16)]{Frahm95}, see also \cite[Eq. (3.79]{Beenakker97}. Preferably, we would like to choose $\hat{p}_k$ and $\hat{q}_k$ such that the inner sum factorizes. Such choice will be given below.

\begin{prop}\label{BS_0}
Let $\om\in\mathcal{W}(\R_{\geq 1})$ and set $\pi_j(s) = \prod_{k=0}^{j-1}\bigl(s+(k+{\textstyle\frac12})^2\bigr)$ for all $j\in\{0,\dots,n\}$. Let $\mathcal{C}_{-}$
be a counterclockwise contour enclosing each zero of $(\frac12-is)_n$ once and enclosing none of the zeros of $(\frac12+is)_n$. Suppose that $\mathcal{P}\om$ is non-vanishing on a neighborhood of
$\mathcal{C}_{-}$ and its interior. For all $j\in\{0,\ldots,n-1\}$, define for $x>1$,
\begin{align}
p_j(x)
= {}& \oint_{\mathcal{C}_{-}} \frac{1}{(\MF\om)(t)}
\frac{P_{it-\frac12}(x)}
{ \pi_{j+1}(t^2)}
\frac{tdt}{\pi i},
\label{eq:p_j_biorth} \\
q_j(x) = {}& (\pi_j(D)\om)(x). \label{eq:q_j_biorth}
\end{align}
Then, $p_j$ is a polynomial of degree at most $j$, and
\begin{equation} \label{eq:BS_biorth}
  \int_1^\infty p_j(x)q_k(x)\,dx=\delta_{j,k},
  \quad j,k=0,\ldots,n-1.
\end{equation}
\end{prop}

\begin{proof}
Since $q_k$ is a finite linear combination of $\om,D\om,\ldots,D^k\om$ with $\om\in\mathcal{W}(\R_{\geq 1})$, by \eqref{eq:def-W}, it has all absolute polynomial moments. Together with the bound in Proposition~\ref{prop:Leg-geom-est}, this justifies interchanging the integrals to obtain
\[
        \int_1^\infty p_j(x)q_k(x)\,dx
        =\oint_{\mathcal C_-}
        \frac{1}{(\MF\om)(t)}\frac{\MF q_k(t)}{\pi_{j+1}(t^2)}
        \frac{tdt}{\pi i}.
\]
Proposition~\ref{prop:admissable-weight-prop} gives
\[
    \MF q_k(t)
    =\pi_k(t^2)(\MF\om)(t),
    \qquad t\in\C,
\]
and thus
\[
    \begin{aligned}
        \int_1^\infty p_j(x)q_k(x)\,dx
        &=\oint_{\mathcal C_-}
        \frac{\pi_k(t^2)}{\pi_{j+1}(t^2)}
        \frac{tdt}{\pi i}.
    \end{aligned}
\]
If $k>j$, then $\pi_{j+1}$ divides $\pi_k$, so the quotient is a polynomial and the contour integral is zero. Suppose now that $k\leq j$. Under the change of variables $u=t^2$, the transplanted contour
winds once around each zero of $\pi_{j+1}$, since $\mathcal C_-$ encloses its negative imaginary preimage and excludes its positive imaginary preimage. It then follows from
\[
    \frac{\pi_k(u)}{\pi_{j+1}(u)}
    =
    \begin{cases}
        u^{-1}+\bigO(u^{-2}),&k=j,\\
        \bigO(u^{-2}),&k<j,
    \end{cases}
    \qquad u\to\infty,
\]
that the contour integral equals $\delta_{j,k}$. This proves the stated biorthogonality relations.
\end{proof}

Proposition \ref{q_int} provides the following integral representations for the $q_j$.

\begin{prop}\label{BS}
The functions $q_j$ defined in \eqref{eq:q_j_biorth} are given by
\begin{equation} \label{eq:explicit_q_j}
q_j(x)
=-\int_{\R}(\mathcal{P}\om)(s)
\pi_j(s^2)
Q_{is-\frac12}(x)\frac{sds}{\pi i}.
\end{equation}
\end{prop}
\begin{proof}
This follows immediately from Proposition \ref{q_int}.
\end{proof}

We are now ready to prove the main result.

\begin{proof}[Proof of Theorem \ref{DIR}]
The biorthogonal system in Propositions \ref{BS_0} and \ref{BS} gives rise to the following double integral representation of the correlation kernel
\begin{equation}
  K_n(x,y) = -\int_{\R} \oint_{\mathcal{C}_{-}} \frac{(\mathcal{P}\om)(s)}{(\mathcal{P}\om)(t)} \left(\sum_{k=0}^{n-1}\frac{(\frac{1}{2}-is)_{k}(\frac{1}{2}+is)_{k}}{(\frac{1}{2}-it)_{k+1}(\frac{1}{2}+it)_{k+1}} \right) P_{it-\frac{1}{2}}(x) Q_{is-\frac{1}{2}}(y) \frac{tdt}{\pi i} \frac{sds}{\pi i}.
\end{equation}
with $\mathcal{C}_{-}$ defined in Proposition \ref{BS_0}. At this point, since $s\mapsto Q_{is-\frac{1}{2}}(y)$ is meromorphic with poles at $\{i(k+\tfrac{1}{2}\}_{k\in\Z_{\geq 0}}$  for all $y\in\R_{>1}$, we may deform the contour $\R$ to any contour $\Sigma$ from $-\infty$ to $\infty$ in $\{s\in\C:\Im s \leq 0\}$ that doesn't intersect $\mathcal{C}_{-}$. Observe that
    $$ (t^2-s^2) \frac{(\frac{1}{2}-is)_{k}(\frac{1}{2}+is)_{k}}{(\frac{1}{2}-it)_{k+1}(\frac{1}{2}+it)_{k+1}} = \frac{(\frac{1}{2}-is)_{k}(\frac{1}{2}+is)_{k}}{(\frac{1}{2}-it)_{k}(\frac{1}{2}+it)_{k}} - \frac{(\frac{1}{2}-is)_{k+1}(\frac{1}{2}+is)_{k+1}}{(\frac{1}{2}-it)_{k+1}(\frac{1}{2}+it)_{k+1}}, $$
and therefore
$$ \sum_{k=0}^{n-1}\frac{(\frac{1}{2}-is)_{k}(\frac{1}{2}+is)_{k}}{(\frac{1}{2}-it)_{k+1}(\frac{1}{2}+it)_{k+1}} = \frac{1}{t^2-s^2} \left( 1-\frac{(\frac{1}{2}-is)_{n}(\frac{1}{2}+is)_{n}}{(\frac{1}{2}-it)_{n}(\frac{1}{2}+it)_{n}} \right). $$
As $\mathcal{C}_{-}$ doesn't intersect $\Sigma$, the first term doesn't contribute to the value of the integral. Hence,
\begin{equation} \label{eq:kernel_formula}
  K_n(x,y) = \int_{\Sigma} \oint_{\mathcal{C}_{-}} \frac{(\mathcal{P}\om)(s)}{(\mathcal{P}\om)(t)} \frac{(\frac{1}{2}-is)_{n}(\frac{1}{2}+is)_{n}}{(\frac{1}{2}-it)_{n}(\frac{1}{2}+it)_{n}} \frac{P_{it-\frac{1}{2}}(x) Q_{is-\frac{1}{2}}(y)}{t^2-s^2} \frac{tdt}{\pi i} \frac{sds}{\pi i}.
\end{equation}
It then remains to carry out some further simplifications. We will write
\begin{equation*}
  \frac{s}{t^2 - s^2} = \frac{1}{2} \left( \frac{1}{t-s} - \frac{1}{s + t}\right).
\end{equation*}
The change of variables $t\mapsto -t$ yields
\begin{align*}
    \oint_{\mathcal{C}_{-}} \frac{1}{(\mathcal{P}\om)(t)} \frac{1}{(\frac{1}{2}-it)_{n}(\frac{1}{2}+it)_{n}} \frac{P_{it-\frac{1}{2}}(x)}{s + t}   \frac{tdt}{\pi i} = \oint_{-\mathcal{C}_{-}} \frac{1}{(\mathcal{P}\om)(t)} \frac{1}{(\frac{1}{2}+it)_{n}(\frac{1}{2}-it)_{n}} \frac{P_{-it-\frac{1}{2}}(x)}{s - t}   \frac{tdt}{\pi i},
\end{align*} 
and thus, after using the symmetry $P_{it-\frac{1}{2}}(x) = P_{-it-\frac{1}{2}}(x)$, see  \cite[Eq. 14.9.5]{Boisvert-Clark-Lozier-Olver10}, the right hand side becomes
\begin{align*}
    \oint_{-\mathcal{C}_{-}} \frac{1}{(\mathcal{P}\om)(t)} \frac{1}{(\frac{1}{2}-it)_{n}(\frac{1}{2}+it)_{n}} \frac{P_{it-\frac{1}{2}}(x)}{s-t}   \frac{tdt}{\pi i}.
\end{align*} 
We can therefore conclude that 
$$ K_n(x,y) = \int_{\Sigma} \oint_{\contC} \frac{(\mathcal{P}\om)(s)}{(\mathcal{P}\om)(t)} \frac{(\frac{1}{2}-is)_{n}(\frac{1}{2}+is)_{n}}{(\frac{1}{2}-it)_{n}(\frac{1}{2}+it)_{n}} \frac{P_{it-\frac{1}{2}}(x) Q_{is-\frac{1}{2}}(y)}{t-s}  \frac{t dt}{\pi i} \frac{ds}{2\pi i},$$
where $\contC$ is a contour that encloses $\{\pm i(k+{\textstyle\frac{1}{2}})\}_{k=0}^{n-1}$ once.
\end{proof}

Finally, specializing Theorem \ref{DIR} to the DMPK ensemble yields Corollary \ref{DIR_DMPK}.

\begin{proof}[Proof of Corollary \ref{DIR_DMPK}]
    This follows immediately from Theorem \ref{DIR} and the fact that for the DMPK ensemble $\om\in\mathcal{W}(\R_{\geq 1})$ and $\MF\om(s)=e^{-\frac{M}{n}s^2}$, see Proposition \ref{prop:gaussian-weights-admissible}.
\end{proof}

\section{Universal conductance fluctuations} \label{sec:conductance}

The goal in this section is to prove Theorem \ref{thm:cond-mean-var}. The starting point will be the biorthogonal system provided by Propositions \ref{BS_0} and \ref{BS}. For the DMPK ensemble \eqref{E_DT}, which has $(\mathcal{P}\om)(s) = e^{-\frac{M}{n} s^2}$, the biorthogonal functions take the form
\begin{align}
    p_{j,M}(x) = {}& \oint_{\contC_{-}} e^{\frac{M}{n} t^2} \frac{P_{it-\frac{1}{2}}(x)}{(\frac{1}{2} - it)_{j + 1} (\frac{1}{2} + it)_{j + 1}} \frac{tdt}{\pi i}, \label{eq:pj-definition} \\
  q_{j,M}(x) = {}& - \int_{\R} e^{-\frac{M}{n} s^2} (\tfrac{1}{2} - is)_j (\tfrac{1}{2} + is)_j Q_{is-\frac{1}{2}}(x) \frac{sds}{\pi i} \\
  = {}& \int^{\infty}_0 e^{-\frac{M}{n} s^2} (\tfrac{1}{2} - is)_j (\tfrac{1}{2} + is)_j \tanh(\pi s) P_{is - \frac{1}{2}}(x) sds.  \label{eq:qj-definition}
\end{align}
where $\contC_{-}$ is the contour in \eqref{eq:p_j_biorth} and we used \eqref{eq:Q-Q=P} in \eqref{eq:qj-definition}. The biorthogonality relations are given by \eqref{eq:BS_biorth} with $p_j, q_k$ replaced by $p_{j, M}, q_{k, M}$.
It follows from the proof of Proposition \ref{BS_0} that the biorthogonality relations extend to all $j,k\in\N$ as well (after enlarging $\contC_{-}$ appropriately). By Corollary \ref{DIR_DMPK}, the correlation kernel is given by $K_{n,M}(x,y)$ in \eqref{eq:defn_K_in_PQ_simplified_further}.

\subsection{Variance of the conductance and resolvent evolution}

The conductance \eqref{eq:defn_conductance} arises as a scalar multiple of the resolvent
$$ R_z=\sum_{j=1}^n\frac1{x_j-z},\quad z\in\C\setminus[1,\infty),$$
at $z=-1$. The mean of the resolvent admits the following expansion, where
$$(\mathfrak{s}f)(z)=\int_1^\infty\frac{f(x)}{x-z} dx,\quad z\in\C\setminus[1,\infty),$$
denotes the Stieltjes transform of $f$.
\medbreak

We will describe an explicit expression of the variance of the resolvent $R_z$ by expressing both terms in the elementary identity
\[
    \Var_{n,M}(R_z) =\mathbb{E}_{n,M}[R_z^2]-\mathbb{E}_{n,M}[R_z]^2,
\]
in terms of the biorthogonal system. We note that for complex $z$, $\Var_{n,M}(R_z)$ is the holomorphic second cumulant, and may not be nonnegative.
\medbreak

We will first focus on $\mathbb{E}_{n,M}[R_z]$. Biorthogonality leads to the following representation.

\begin{prop}\label{prop:resolvent-exp}
For all $z\in\C\setminus[1,\infty)$ and $M>0$,
\begin{equation} \label{eq:resolvent-exp}
 \E_{n,M}[R_z]
 =\sum_{j=0}^{n-1} p_{j,M}(z) \mathfrak{s}q_{j,M}(z).    
\end{equation}
\end{prop}
\begin{proof}
    Evidently, 
        \begin{equation} \label{eq:mean-res-expl}
        \E_{n,M}[R_z] = \int_1^\infty \frac{K_{n,M}(x,x)}{x-z}\,dx = \sum_{j=0}^{n-1} \int_1^\infty \frac{p_{j,M}(x)q_{j,M}(x)}{x-z}dx.
        \end{equation}
    Since
        \[
        \frac{p_{j,M}(x)}{x-z} = \frac{p_{j,M}(x)-p_{j,M}(z)}{x-z} + \frac{p_{j,M}(z)}{x-z},
        \]
    and the first term is a polynomial of degree at most $j-1$, the biorthogonality relations imply that
    \[
    \E_{n,M}[R_z] = \int_1^\infty \frac{K_{n,M}(x,x)}{x-z}\,dx = \sum_{j=0}^{n-1} p_{j,M}(z)\int_1^\infty  \frac{q_{j,M}(x)}{x-z}dx.
    \]
    This finishes the proof.
\end{proof}

The following kernel evolution property, which decouples the variables $(x,y)$, will lead to a more suitable representation of $\E_{n,M}[R_z]$ involving fewer biorthogonal functions.

\begin{prop}\label{prop:kernel-evolution}
Let $x,y>1$. For all $M>0$,
\begin{equation}\label{eq:kernel-evolution}
 \left.\partial_u K_{n,u}(x,y)\right|_{u=M} = -\frac{1}{n}p_{n-1,M}(x)q_{n,M}(y).
\end{equation}
\end{prop}
\begin{proof}
Since the contour $\contC$ is compact and doesn't intersect $\R$, the bound for $Q_{is-\frac12}(y)$ in Proposition \ref{prop:Leg-Q-cont} and dominated convergence justify interchanging $\partial_u$ and the double integral in \eqref{eq:defn_K_in_PQ_simplified_further}. Since $(t^2-s^2)/(t-s)=t+s$, this yields
\begin{align*}
n\left.\partial_u K_{n,u}(x,y)\right|_{u=M}
=\int_{\R}\oint_{\contC}
e^{\frac{M}{n}(t^2-s^2)}(t+s)
\frac{\pi_n(s^2)}{\pi_n(t^2)}
P_{it-\frac12}(x)Q_{is-\frac12}(y)
\frac{tdt}{\pi i}\frac{ds}{2\pi i}.
\end{align*}
By symmetry \eqref{eq:symmetry_P_Q} and \eqref{eq:pj-definition},
\begin{align*}
\oint_{\contC}
e^{\frac{M}{n}t^2}(t+s)
\frac{P_{it-\frac12}(x)}{\pi_n(t^2)}
\frac{tdt}{\pi i}
=2s\oint_{\contC_-}
e^{\frac{M}{n}t^2}
\frac{P_{it-\frac12}(x)}{\pi_n(t^2)}
\frac{tdt}{\pi i}
=2s p_{n-1,M}(x).
\end{align*}
It then remains to use \eqref{eq:qj-definition}.
\end{proof}

The average resolvent evolution takes a particularly nice form. 

\begin{prop}\label{prop:resolvent-evolution}
For all $z\in\C\setminus[1,\infty)$ and $M>0$,
\begin{equation} \label{eq:first-resolvent-evolution}
 \left.\partial_u\E_{n,u}[R_z]\right|_{u=M}
 =-\frac1n p_{n-1,M}(z) (\mathfrak{s}q_{n,M})(z).    
\end{equation}
\end{prop}
\begin{proof}
Since $x\mapsto(x-z)^{-1}$ is bounded on $\R_{\geq 1}$, $p_{j, M}$ are polynomials and all $q_{j,M}\in\mathcal{W}(\R_{\geq 1})$ so that both $q_{j, M}(x)$ and $\partial_u q_{j, u}(x) = -D q_{j, u}(x)/n$ vanishes rapidly as $x \to \infty$, dominated convergence justifies interchanging $\partial_u$ and the integral in the expression on the right hand side of \eqref{eq:mean-res-expl}. Consequently,
$$\left.\partial_u\E_{n,u}[R_z]\right|_{u=M}
=\int_1^\infty\frac{\left.\partial_uK_{n,u}(x,x)\right|_{u=M}}{x-z}dx,$$
where, by Proposition \ref{prop:kernel-evolution},
$$\left.\partial_uK_{n,u}(x,x)\right|_{u=M} = -\frac{1}{n} p_{n-1,M}(x) q_{n,M}(x).$$
Similarly as in the proof of Proposition \ref{prop:resolvent-exp}, the biorthogonality relations \eqref{eq:BS_biorth} then yield \eqref{eq:first-resolvent-evolution}.
\end{proof}

This leads to the following alternative representation of $\E_{n,M}[R_z]$ that only depends on two biorthogonal functions.

\begin{prop}\label{prop:resolvent-mean}
For all $z\in\C\setminus[1,\infty)$ and $M>0$,
\begin{equation} \label{eq:resolvent-mean}
\E_{n,M}[R_z]
 =\frac{n}{1-z}- \int_0^M \frac1n p_{n-1,u}(z) (\mathfrak{s}q_{n,u})(z) du.    
\end{equation}
\end{prop}
\begin{proof}
Let $0<\epsilon<M$. Integrating \eqref{eq:first-resolvent-evolution} over $[\epsilon,M]$ gives
\[
\E_{n,M}[R_z]
=\E_{n,\epsilon}[R_z]
-\frac1n\int_\epsilon^M
p_{n-1,u}(z)(\mathfrak{s}q_{n,u})(z)\,du.
\]

Recall that the ballistic initial condition is the deterministic configuration $x_1=\cdots=x_n=1$. Thus the joint distribution of the particles converges to this configuration as $\epsilon\downarrow0$. Consequently,
\[
\lim_{\epsilon\downarrow0}\E_{n,\epsilon}[R_z]
=\frac{n}{1-z}.
\]
Letting $\epsilon\downarrow0$ in the previous identity therefore proves both convergence of the improper integral at $u=0$ and \eqref{eq:resolvent-mean}.
\end{proof}

We will now express $\mathbb{E}_{n,M}[R_z^2]$ in terms of the biorthogonal functions.

\begin{prop}\label{prop:resolvent-sq-exp}
For all $z\in\C\setminus[1,\infty)$ and $M>0$,
\begin{equation} \label{eq:resolvent-sq-exp}
\E_{n,M}[R_z^2]
=\sum_{j=0}^{n-1}
\left(
p_{j,M}(z)(\mathfrak{s}q_{j,M})'(z)
-p_{j,M}'(z)(\mathfrak{s}q_{j,M})(z)
\right).
\end{equation}
\end{prop}
\begin{proof}
Evidently,
\begin{align*}
R_z^2 = \sum_{k=1}^n\frac1{(x_k-z)^2} + \sum_{1\leq k<l\leq n}
\frac{2}{(x_k-z)(x_l-z)},
\end{align*}
hence, after a partial fraction decomposition,
\begin{align*}
R_z^2 = \sum_{k=1}^n\frac1{(x_k-z)^2} - \frac{2}{\Delta_n(x)}
\sum_{k=1}^n\frac{\partial_{x_k}\Delta_n(x)}{x_k-z}.
\end{align*}
Taking expectations, the first term becomes $\E_{n,M}[\partial_z R_z]$. Since $x\mapsto (x-z)^{-2}$ is bounded on $\R_{\geq 1}$ and all $q_{j,M}\in\mathcal{W}(\R_{\geq 1})$, dominated convergence justifies interchanging expectation and $\partial_z$. Hence,
\[
\E_{n,M}[R_z^2]
=\partial_z \E_{n,M}[R_z] - 
\E_{n,M}\left[
\frac{2}{\Delta_n(x)}
\sum_{k=1}^n\frac{\partial_{x_k}\Delta_n(x)}{x_k-z}
\right].
\]
We will show that
\[
\E_{n,M}\left[
\frac1{\Delta_n(x)}
\sum_{k=1}^n\frac{\partial_{x_k}\Delta_n(x)}{x_k-z}
\right]
=
\sum_{j=0}^{n-1}
\int_1^\infty
\frac{p_{j,M}'(x)q_{j,M}(x)}{x-z}dx.
\]
In that case, similarly as in the proof of Proposition \ref{prop:resolvent-exp}, the biorthogonality relations yield
\[
\int_1^\infty
\frac{p_{j,M}'(x)q_{j,M}(x)}{x-z}\,dx
=p_{j,M}'(z)(\mathfrak{s}q_{j,M})(z),
\]
and together with \eqref{eq:resolvent-exp} this proves \eqref{eq:resolvent-sq-exp}.

To show the remaining identity, observe that by Andr\'eief's identity and the biorthogonality relations
\begin{align}\label{eq:resolvent-sq-exp-Andreief}
\frac1{n!}&\int_{(1,\infty)^n}
\det\Big[
p_{j-1,M}(x_k)
+\epsilon\frac{p_{j-1,M}'(x_k)}{x_k-z}
\Big]_{j,k=1}^n
\det[q_{j-1,M}(x_k)]_{j,k=1}^n
dx_1\cdots dx_n \notag \\
&\quad=
\det\Big[
\delta_{j,k}
+\epsilon\int_1^\infty
\frac{p_{j,M}'(x)q_{k,M}(x)}{x-z}dx
\Big]_{j,k=0}^{n-1} .
\end{align}
Both sides are polynomials in $\epsilon$. The coefficient of $\epsilon$ in the first determinant on the left hand side  of \eqref{eq:resolvent-sq-exp-Andreief} is given by
\begin{equation} \label{eq:after_andreief}
\left.\partial_\epsilon
\det\Big[
p_{j-1,M}(x_k)
+\epsilon\frac{p_{j-1,M}'(x_k)}{x_k-z}
\Big]_{j,k=1}^n
\right|_{\epsilon=0}
=
\sum_{k=1}^n\frac1{x_k-z}
\partial_{x_k}
\det[p_{j-1,M}(x_\ell)]_{j,\ell=1}^n.
\end{equation}
The determinant on the right hand side of \eqref{eq:after_andreief} is a non-zero constant multiple
of $\Delta_n(x)$, hence it can be written as
\[
\det[p_{j-1,M}(x_\ell)]_{j,\ell=1}^n
\frac1{\Delta_n(x)}
\sum_{k=1}^n\frac{\partial_{x_k}\Delta_n(x)}{x_k-z}.
\]
Multiplying by the second determinant and integrating, the biorthogonal representation of the joint density shows that the coefficient of $\epsilon$ on the left hand side of \eqref{eq:resolvent-sq-exp-Andreief} is
\[
\E_{n,M}\left[
\frac1{\Delta_n(x)}
\sum_{k=1}^n\frac{\partial_{x_k}\Delta_n(x)}{x_k-z}
\right].
\]
On the other hand, expanding the determinant on the right hand side of \eqref{eq:resolvent-sq-exp-Andreief}, it is straightforward to see that the coefficient of $\epsilon$ is given by
\[
\sum_{j=0}^{n-1}
\int_1^\infty
\frac{p_{j,M}'(x)q_{j,M}(x)}{x-z}\,dx.
\]
Comparing both coefficients then finishes the proof.
\end{proof}

For the conductance, i.e. $z=-1$, the right hand side of \eqref{eq:resolvent-sq-exp} drastically simplifies.

\begin{prop} \label{prop:conductance-closure}
    For all $M>0$,
    \begin{equation}\label{eq:conductance-closure}
        \mathbb{E}_{n,M}[G^2] = 2p_{n-1,M}(-1)(\mathfrak{s}q_{n,M})(-1).
    \end{equation}
\end{prop}
\begin{proof}
At $z=-1$, \eqref{eq:resolvent-sq-exp} yields
\begin{equation*}
\E_{n,M}[G^2]
=4\sum_{j=0}^{n-1}
\Bigl(
p_{j,M}(-1)(\mathfrak{s}q_{j,M})'(-1)
-p_{j,M}'(-1)(\mathfrak{s}q_{j,M})(-1)
\Bigr).
\end{equation*}
Recall the differential operator $D$ defined in \eqref{D}. For notational convenience, denote $a_j = j + \frac{1}{2}$. We can use the relation
\[
((D+a_j^2)f)(-1)=2f'(-1)+j(j+1)f(-1),
\]
to express the derivatives at $-1$ as $D$-derivatives. This yields
\begin{equation} \label{eq:conductance-closure-sum}
\E_{n,M}[G^2]
=2\sum_{j=0}^{n-1}
\Bigl(
p_{j,M}(-1)((D+a_j^2)\mathfrak{s}q_{j,M})(-1)
-((D+a_j^2)p_{j,M})(-1)(\mathfrak{s}q_{j,M})(-1)
\Bigr).
\end{equation}
The identity $DP_{is-\frac12}=s^2P_{is-\frac12}$ gives
\[
(D+a_j^2)p_{j,M}=p_{j-1,M},
\qquad
(D+a_j^2)q_{j,M}=q_{j+1,M},
\]
where we define $p_{-1,M}=0$. Since $q_{j,M}\in\mathcal W(\R_{\geq1})$, smoothness at $1$ and the decay established in the proof of Proposition~\ref{prop:admissable-weight-prop} imply that
\[
\Big[
(1-x^2)\Big(
\frac{q_{j,M}'(x)}{x-z}
+\frac{q_{j,M}(x)}{(x-z)^2}
\Big)
\Big]_1^\infty=0.
\]
Using $D_x(x-z)^{-1}=D_z(x-z)^{-1}$, integration by parts therefore yields
\[
((D+a_j^2)\mathfrak{s}q_{j,M})(z)
=\mathfrak{s}[(D+a_j^2)q_{j,M}](z)
=(\mathfrak{s}q_{j+1,M})(z).
\]
This makes the sum on the right hand side of \eqref{eq:conductance-closure-sum} telescoping, leading to \eqref{eq:conductance-closure}.
\end{proof}

\begin{rem}\label{rem:Frahm-moment-identity}
Proposition \ref{prop:resolvent-evolution} at $z=-1$ and Proposition \ref{prop:conductance-closure} imply that 
$$\mathbb{E}_{n,M}[G^2] = -n \left.\partial_u\E_{n,u}[G]\right|_{u=M}.$$
This closure property of the second conductance moment $\mathbb{E}_{n,M}[G^2]$ in terms of the first moment $\mathbb{E}_{n,M}[G]$ was already predicted in
\cite[Eq. (27)]{Frahm95}.
\end{rem}

The following theorem summarizes the results so far for the conductance. For notational convenience, we define
\begin{align}\label{pq-def}
\mathfrak p_{n,M}
 =\Gamma(n+{\textstyle\frac{1}{2}})^2 p_{n-1,M}(-1),\quad
\mathfrak q_{n,M}
 = \frac{\mathfrak{s}q_{n,M}(-1)}{\Gamma(n+\frac{1}{2})^2}.
\end{align}
The additional factors will be convenient in the asymptotic analysis in the next subsection and don't change the value of the product.

\begin{thm}
For all $M>0$,
\begin{equation} \label{momentsdef}
    \mathbb E_{n,M}[G]=n-\frac2n\int_0^M \mathfrak p_{n,u} \mathfrak q_{n,u}\,du, \quad
    \operatorname{Var}_{n,M} G =2\mathfrak p_{n,M}\mathfrak q_{n,M} -\mathbb E_{n,M}[G]^2. 
\end{equation}
\end{thm}
\begin{proof}
This follows from Proposition~\ref{prop:resolvent-mean} at $z=-1$ and Proposition~\ref{prop:conductance-closure}.
\end{proof}

\subsection{Expansion of the mean and variance of the conductance}
\label{subsec:variance-asymptotics}

In order to facilitate the asymptotic analysis, we will consider the following representations of $\mathfrak p_{n,M}$ and $\mathfrak q_{n,M}$. It will be convenient to use
\begin{align}\label{kernels}
R_1(v)=\frac{\Gamma(n+\frac{1}{2})^2}{\Gamma(n+\frac{1}{2}-v)\Gamma(n+\frac{1}{2}+v)},\quad
R_2(v)=\frac{\Gamma(n+\frac{1}{2}-iv)\Gamma(n+\frac{1}{2}+iv)}{\Gamma(n+\frac{1}{2})^2}.
\end{align}
The contour $i\contC_-$ is counterclockwise, encloses $\{ k + \frac{1}{2} : k = 0, 1, \dotsc, n - 1 \}$, and excludes $\{ -k - \frac{1}{2} : k = 0, 1, \dotsc, n - 1 \}$.

\begin{prop}
\label{prop:explicit-endpoints}
For all $M>0$,
\begin{align}
  \mathfrak p_{n,M}
&=i\oint_{i\contC_-}
z e^{-\frac{M}{n}z^2}R_1(z)\tan(\pi z)dz,
 \label{eq:pi-explicit}\\
 \mathfrak{q}_{n,M}
 &=\int_0^\infty s
 e^{-\frac{M}{n}s^2} R_2(s) \tanh(\pi s) ds.
 \label{eq:chi-explicit}
\end{align}
\end{prop}
\begin{proof}
 We first focus on $\mathfrak p_{n,M}$, see \eqref{pq-def}. Taking residues in \eqref{eq:pj-definition} before evaluating at $x=-1$ allows us to replace the Legendre factor by $i\sinh(\pi t)$, since \cite[18.6.1]{Boisvert-Clark-Lozier-Olver10}
\[
P_k(-1)=(-1)^k
=\left.i\sinh(\pi t)\right|_{t=-i(k + \frac{1}{2})}.
\]
The identity
\begin{equation}
 (\frac{1}{2} - it)_n (\frac{1}{2} + it)_n
 =\frac{\cosh(\pi t)}{\pi}
 \Gamma(n+{\textstyle\frac{1}{2}}-it)\Gamma(n+{\textstyle\frac{1}{2}}+it),
 \label{eq:Aj-gamma-product}
\end{equation}
and the change of variables $t\mapsto -iz$ then prove \eqref{eq:pi-explicit}. For $\mathfrak q_{n,M}$, see \eqref{pq-def}, use \eqref{eq:qj-definition} and
\begin{equation}\label{eq:resolvent-Mehler-Fock}
\int_1^\infty\frac{P_{is-1/2}(x)}{1+x}\,dx
=\frac{\pi}{\cosh(\pi s)},\qquad s\in\R,
\end{equation}
which is \cite[Eq. 14.20.13]{Boisvert-Clark-Lozier-Olver10} with $x=1$. Together with
\eqref{eq:Aj-gamma-product}, this proves \eqref{eq:chi-explicit}.
\end{proof}

Next, we prove an expansion for $\mathfrak q_{n,M}$. Throughout what follows, $L=M+1$.

\begin{prop}\label{qfactor}
Uniformly for $1\le L\le n$ as $n\to\infty$,
\begin{equation}\label{qexp}
\mathfrak q_{n,M}
=\frac n{2L}-c_0(L)+\frac{c_{-1}(L)}{n}+O\!\left(\frac{L^2}{n^2}\right),
\end{equation}
where
\begin{equation}\label{coefficients}
\begin{aligned}
c_0(L)=\frac1{24}-\frac1{6L^3},\quad
c_{-1}(L)=\frac{7L}{960}+\frac1{24L^2}-\frac1{5L^4}+\frac1{6L^5}.
\end{aligned}
\end{equation}
\end{prop}
\begin{proof}
Decompose \eqref{eq:chi-explicit} as $\mathfrak q_{n,M}=\mathfrak q_{n,M,1}-\mathfrak q_{n,M,2}$ with
\[
\mathfrak q_{n,M,1}=\int_0^\infty s e^{-\frac{M}{n}s^2}R_2(s)ds,\quad \mathfrak q_{n,M,2}=\int_0^\infty s e^{-\frac{M}{n}s^2}R_2(s)
(1-\tanh(\pi s))ds.
\]
We will expand $\mathfrak q_{n,M,1}$ via a steepest descent analysis and then provide an appropriate estimate for $\mathfrak q_{n,M,2}$. Stirling's formula, see \cite[Eq.~5.11.8]{Boisvert-Clark-Lozier-Olver10}, gives
\[
\log\Gamma(w+\tfrac12)
=w\log w-w+\frac12\log(2\pi)
-\frac1{24w}+\bigO(|w|^{-3}),
\]
uniformly for $|\arg w|\le\pi/2$. Applying this with
$w=n(1\pm ix)$ and $w=n$, we obtain
\begin{equation} \label{eq:qfactor-proof-expansion}
e^{-Mn x^2}R_2(nx)
=\exp\left(
-n\phi_M(x)+\frac{x^2}{12n(1+x^2)}+\bigO(n^{-3}),
\right),
\end{equation}
uniformly for $x\ge0$, where
\[
\phi_M(x)=Mx^2+2x\arctan x-\log(1+x^2).
\]
In particular, 
\[
\phi_M(x)\ge
\begin{cases}
Lx^2/2,&0\le x\le1,\\
x/2,&x\ge1,
\end{cases}
\]
so that the integral representation also extends to $M=0$ by dominated convergence. Since
\[
\phi_M'(x)=2Mx+2\arctan x,\quad
\phi_M(0)=\phi_M'(0)=0,\quad
\phi_M''(0)=2L,
\]
the origin is a saddle point and the positive real axis is a steepest descent path. It will be convenient to perform the change of variables $s\mapsto \sqrt{n/L}v$, which yields
\[
\mathfrak q_{n,M,1}
=\frac nL\int_0^\infty
v e^{-\frac ML v^2}
R_2\left(\sqrt{\frac nL}\,v\right)dv.
\]
The preceding bounds show that the contribution from $v\ge n^{1/12}$ is exponentially small, uniformly for $1\le L\le n$, hence we may restrict the previous integral to its complement. 
Near the saddle,
\[
\phi_M(x)=Lx^2-\frac{x^4}{6}+\frac{x^6}{15}+\bigO(x^8).
\]
Consequently, by Taylor's theorem,
\begin{align*}
e^{-\frac ML v^2}
R_2\left(\sqrt{\frac nL}\,v\right)
=e^{-v^2}\left[
1+\frac{v^4}{6nL^2}
+\frac1{n^2}\left(
\frac{v^2}{12L}-\frac{v^6}{15L^3}
+\frac{v^8}{72L^4}
\right)
+O\left(\frac{1+v^{12}}{n^3}\right)
\right],
\end{align*}
uniformly for $0\le v\le n^{1/12}$ and $1\le L\le n$. The integrated remainder is certainly $\bigO(n^{-2})$. The integrals corresponding to the other terms can be extended again to $[0,\infty)$ at the cost of an exponentially small remainder. Using
\begin{equation} \label{eq:qfactor-proof-Gaussian-mom}
\int_0^\infty v^{2j+1}e^{-v^2}\,dv=\frac{j!}{2},
\end{equation}
we can therefore conclude that
\[
\mathfrak q_{n,M,1}
=\frac n{2L}+\frac1{6L^3}
+\frac1n\left(
\frac1{24L^2}-\frac1{5L^4}+\frac1{6L^5}
\right)+\bigO(n^{-2}).
\]
Evidently, $n^{-2}\le L^2/n^2$. It remains to estimate $\mathfrak q_{n,M,2}$. The uniform expansion \eqref{eq:qfactor-proof-expansion} and the mean value theorem imply that
\[
\left|
e^{-\frac{M}{n}s^2}R_2(s)-e^{-\frac Ln s^2}
\right|
\le\frac{C}{n^3}(1+s^2+s^4),
\qquad s\ge0.
\]
because the difference between the two phase exponents is bounded by $s^4/(6n^3)$ and the remaining exponential factor differs from one by at most
$C(1+s^2)/n^3$. Since $1-\tanh(\pi s)\le2e^{-2\pi s}$, it then follows that
\[
\mathfrak q_{n,M,2}
=\int_0^\infty s e^{-\frac Ln s^2}
(1-\tanh(\pi s))ds+\bigO(n^{-3}).
\]
Finally,
\[
\left|e^{-\frac Ln s^2}-1+\frac Ln s^2\right|
\le\frac{L^2s^4}{2n^2},
\]
and
\begin{equation} \label{eq:qfactor-proof-tanh-mom}
\int_0^\infty s(1-\tanh(\pi s))ds=\frac1{24},
\quad
\int_0^\infty s^3(1-\tanh(\pi s))ds=\frac7{960}. 
\end{equation}
Hence,
\[
\mathfrak q_{n,M,2}
=\frac1{24}-\frac{7L}{960n}
+O\left(\frac{L^2}{n^2}\right).
\]
Combining the expansions of $\mathfrak q_{n,M,1}$ and $\mathfrak q_{n,M,2}$ then yields \ref{qexp}.
\end{proof}

Next, we prove an expansion for $\mathfrak p_{n,M}$. Up to the sign of $c_{0}(L)$ (and the factor $2$), it is the same as for $\mathfrak q_{n,M}$

\begin{prop}\label{pfactor}
Uniformly for $1\le L\le n$ as $n\to\infty$,
\begin{equation}\label{pexp}
\frac{1}{2}\mathfrak p_{n,M}=\frac{n}{2L}+c_{0}(L)+\frac{c_{-1}(L)}{n}+O\!\left(\frac{L^2}{n^2}\right),
\end{equation}
where $c_{0}(L)$ and $c_{-1}(L)$ are as in \eqref{coefficients}.
\end{prop}
\begin{proof}
We first note that the contour representation \eqref{eq:pi-explicit} extends to $M=0$ via the underlying residues. The Euler product representation of the gamma function gives the elementary bound
\[
e^{-\frac{M}{n}s^2}R_1(s)
\le e^{-\frac{2L}{3n}s^2},\qquad 0\leq s<n+\tfrac{1}{2},
\]
hence the contribution of the residues that correspond to the poles $k - \frac{1}{2}\ge \lfloor n/2\rfloor$ is exponentially small. We may therefore replace the contour $i\contC_-$ in \eqref{eq:pi-explicit} by the counterclockwise boundary of the triangle with vertices $0$ and $\lfloor n/2\rfloor(1\pm i/\sqrt3)$. Similarly as in the proof of Proposition \ref{qfactor}, Stirling's formula gives
\begin{equation}\label{eq:pfactor-proof-expansion}
e^{-Mn w^2}R_1(nw)
=\exp\left(
-n\psi_M(w)+\frac{w^2}{12n(1-w^2)}+\bigO(n^{-3})
\right),
\end{equation}
uniformly for $|\Re w|\le1/2$ and $|\Im w|\le1/2$, where
\[
\psi_M(w)
=Mw^2+(1+w)\log(1+w)+(1-w)\log(1-w).
\]
Since $\psi_0(x)\ge x^2$ for real $|x|\le1/2$ and $\psi_0''(w)=2/(1-w^2)$, Taylor's theorem in the
vertical direction gives
\[
\Re\psi_M(x+iy)
\ge Lx^2-\frac43Ly^2.
\]
In particular, throughout the triangle,
\[
\left|e^{-\frac{M}{n}z^2}R_1(z)\right|
\le C e^{-\frac{5L}{9n}(\Re z)^2}.
\]
On its vertical side, we can use $\tan(\pi(\lfloor n/2\rfloor+iy))=i\tanh(\pi y)$, to show that the contribution from this side is exponentially small. Let $\gamma$ be the segment from $0$ to
$\lfloor n/2\rfloor(1+i/\sqrt3)$, oriented away from the origin.
By conjugation symmetry,
\[
\frac12\mathfrak p_{n,M}
=\Im\int_\gamma
z e^{-\frac{M}{n}z^2}R_1(z)\tan(\pi z)\,dz
+\bigO(n^{-2}).
\]
Since $z\mapsto z e^{-Mz^2/n}R_1(z)$ is entire, its integral over $\gamma$ can be deformed to $[0,\lfloor n/2\rfloor]$ and a vertical segment whose contribution is exponentially small. This leads to the decomposition
\[
\frac12\mathfrak p_{n,M}
=\mathfrak p_{n,M,1}+\mathfrak p_{n,M,2}+\bigO(n^{-2}),
\]
where
\begin{align*}
\mathfrak p_{n,M,1}
=\int_0^{\lfloor n/2\rfloor} s e^{-\frac{M}{n}s^2}R_1(s)ds,\quad
\mathfrak p_{n,M,2}
=\Im\int_\gamma
z e^{-\frac{M}{n}z^2}R_1(z)
(\tan(\pi z)-i)dz.
\end{align*}

The asymptotic expansion of $\mathfrak p_{n,M,1}$ can be derived via a steepest descent analysis, similarly as for $\mathfrak q_{n,M,1}$ in the proof of Proposition \eqref{qfactor}. Since for $0<x<1$,
\[
\psi_M'(x)=2Mx+\log\frac{1+x}{1-x}>0,
\quad
\psi_M(0)=\psi_M'(0)=0,\quad
\psi_M''(0)=2L,
\]
the origin is a saddle point and the positive real axis is a steepest descent path. The change of variables $s\mapsto \sqrt{n/L} v$ gives
\[
\mathfrak p_{n,M,1}
=\frac nL\int_0^{\lfloor n/2\rfloor \sqrt{L/n}}
v e^{-\frac ML v^2}
R_1\left(\sqrt{\frac nL}\,v\right)\,dv.
\]
The preceding bounds show that the contribution from $v\ge n^{1/12}$ is exponentially small, uniformly for $1\le L\le n$. Near the saddle,
\[
\psi_M(x)
=Lx^2+\frac{x^4}{6}+\frac{x^6}{15}+\bigO(x^8).
\]
Consequently, by
Taylor's theorem,
\begin{align*}
e^{-\frac ML v^2}
R_1\left(\sqrt{\frac nL}\,v\right)
=e^{-v^2}\left[
1-\frac{v^4}{6nL^2}
+\frac1{n^2}\left(
\frac{v^2}{12L}-\frac{v^6}{15L^3}
+\frac{v^8}{72L^4}
\right)
+O\left(\frac{1+v^{12}}{n^3}\right)
\right],
\end{align*}
uniformly for $0\le v\le n^{1/12}$ and $1\le L\le n$. The integrated remainder is certainly $\bigO(n^{-2})$. The integrals corresponding to the other terms can be extended to $[0,\infty)$ at the cost of an exponentially small remainder. Using \eqref{eq:qfactor-proof-Gaussian-mom}, we can therefore conclude that
\[
\mathfrak p_{n,M,1}
=\frac n{2L}-\frac1{6L^3}
+\frac1n\left(
\frac1{24L^2}-\frac1{5L^4}+\frac1{6L^5}
\right)+\bigO(n^{-2}).
\]

It remains to estimate $\mathfrak p_{n,M,2}$. We will do so similarly as for $\mathfrak q_{n,M,2}$ in the proof of Proposition \eqref{qfactor}. The uniform expansion \eqref{eq:pfactor-proof-expansion} and the mean value theorem imply that
\[
\left|
e^{-\frac{M}{n}z^2}R_1(z)-e^{-\frac Ln z^2}
\right|
\le\frac{C}{n^3}(1+|z|^2+|z|^4),
\qquad z\in\gamma,
\]
because $\psi_M(w)-Lw^2=\bigO(w^4)$ and the real parts of $-n\psi_M(z/n)$ and $-Lz^2/n$ are non-positive for $z\in\gamma$. Since $|\tan(\pi z)-i|\le Ce^{-c|z|}$, replacing $e^{-\frac{M}{n}z^2}R_1(z)$ by $e^{-\frac Ln z^2}$ in the integral introduces an $\bigO(n^{-3})$ error. Since $\Re z^2\ge0$ on $\gamma$, by Taylor's theorem,
\[
\left|e^{-\frac Ln z^2}-1+\frac Ln z^2\right|
\le\frac{L^2|z|^4}{2n^2}.
\]
The resulting integrated remainder is $\bigO(L^2/n^2)$. After replacing $e^{-\frac Ln z^2}$ by $1-\frac Ln z^2$, we may extend $\gamma$ to the full ray $e^{i\pi/6}\R_{\geq 0}$ at the cost of an exponentially small remainder. Consequently,
\[
\mathfrak p_{n,M,2}
=\Im\int_0^{e^{i\pi/6}\infty}
z\left(1-\frac Ln z^2\right)
(\tan(\pi z)-i)\,dz
+O\left(\frac{L^2}{n^2}\right).
\]
Rotating the ray to the positive imaginary axis and performing the change of variables $z\mapsto is$ then gives
\begin{align*}
\mathfrak p_{n,M,2}
&=\int_0^\infty
s\left(1+\frac Ln s^2\right)
(1-\tanh(\pi s))\,ds
+O\left(\frac{L^2}{n^2}\right).
\end{align*}
Hence, by \eqref{eq:qfactor-proof-tanh-mom}, 
\begin{align*}
\mathfrak p_{n,M,2}
=\frac1{24}+\frac{7L}{960n}
+O\left(\frac{L^2}{n^2}\right),
\end{align*}
Combining the expansions of $\mathfrak p_{n,M,1}$ and $\mathfrak p_{n,M,2}$ then proves \eqref{pexp}.
\end{proof}

We can use the previous expansions to obtain an expansion of the product $\mathfrak p_{n,M}\mathfrak q_{n,M}$ as well. This immediately yields an expansion of the average conductance evolution $\partial_u\E_{n,u}[G]$ and the second conductance moment $\E_{n,u}[G^2]$.

\begin{prop}\label{product}
Uniformly for $0\le M\le n-1$ as $n\to\infty$,
\begin{equation}\label{productexp}
\mathfrak p_{n,M}\mathfrak q_{n,M}
 =\frac{n^2}{2L^2}+\frac1{90}+\frac1{9L^3}-\frac2{5L^5}+\frac5{18L^6}+O\!\left(\frac Ln\right).
\end{equation}
\end{prop}
\begin{proof}
Since $1\le L\le n$, \eqref{coefficients} gives $c_0(L)=\bigO(1)$ and $c_{-1}(L)=\bigO(L)$. Thus the terms in \eqref{pexp} and \eqref{qexp} are $\bigO(n/L)$, and their remainders contribute
$\bigO(L/n)$ to the product. Multiplying these expansions therefore yields
\begin{align*}
\mathfrak p_{n,M}\mathfrak q_{n,M}
=\frac{n^2}{2L^2}
+\frac{2c_{-1}(L)}L-2c_0(L)^2
+O\left(\frac Ln\right),
\end{align*}
where we used that $c_{-1}(L)^2/n^2=\bigO(L^2/n^2)=\bigO(L/n)$. Finally, \eqref{coefficients} yields
\[
\frac{2c_{-1}(L)}L-2c_0(L)^2
=\frac1{90}+\frac1{9L^3}
-\frac2{5L^5}+\frac5{18L^6},
\]
which proves \eqref{productexp}.
\end{proof}

We are now ready to prove the main result.

\begin{proof}[Proof of Theorem \ref{thm:cond-mean-var}]
We first retain one further term in the expansion of $\mathbb E_{n,M}[G]$. By \eqref{momentsdef} and \eqref{productexp}, after the change of variables $u\mapsto v-1$,
\begin{align*}
\E_{n,M}[G]
&=\frac nL-\frac2n\int_1^L
\left(
\frac1{90}+\frac1{9v^3}
-\frac2{5v^5}+\frac5{18v^6}
\right)dv
+O\left(\frac{L^2-1}{n^2}\right)\\
&=\frac nL-\frac1n\left(
\frac L{45}-\frac1{9L^2}
+\frac1{5L^4}-\frac1{9L^5}
\right)
+O\left(\frac{L^2-1}{n^2}\right).
\end{align*}
Since $1\le L\le n$, this yields \eqref{meanexp}. Squaring the preceding expansion gives
\[
\E_{n,M}[G]^2
=\frac{n^2}{L^2}-\frac2{45}
+\frac2{9L^3}-\frac2{5L^5}+\frac2{9L^6}
+O\left(\frac Ln\right).
\]
Together with \eqref{momentsdef} and \eqref{productexp}, this implies that
\begin{align*}
\operatorname{Var}_{n,M}G
=\frac1{15}-\frac2{5L^5}+\frac1{3L^6}
+O\left(\frac Ln\right),
\end{align*}
which is precisely \eqref{varianceexp} after writing $L=M+1$. This finishes the proof.
\end{proof}

\section{Local statistics} \label{sec:LS}

For notational convenience, in this section we will denote
\begin{equation} \label{eq:Khat}
  \Khat_{n}(\xi, \eta) = \sqrt{\sinh \xi \sinh \eta} K_{n,M_n}(\cosh \xi, \cosh \eta) = - \int_{\Sigma} \frac{ds}{\pi i} \oint_{\contC} \frac{dt}{2\pi i} \frac{q_n(s; \eta) p_n(t; \xi)}{s-t},
\end{equation}
where
\begin{align} 
  q_n(s; \eta) = {}& \sqrt{\sinh \eta} e^{-\frac{M_n}{n} s^2} ({\textstyle\frac{1}{2}}-is)_{n}({\textstyle\frac{1}{2}}+is)_{n} Q_{is-\frac{1}{2}}(\cosh \eta), \\
  p_n(t; \xi) = {}& \sqrt{\sinh \xi} \frac{t e^{\frac{M_n}{n} t^2}}{(\frac{1}{2}-it)_{n}(\frac{1}{2}+it)_{n}} P_{it-\frac{1}{2}}(\cosh \xi),
\end{align}
and $M_n>0$ is chosen appropriately according to the specific regime that is being considered. In terms of the kernel $\mathbb{K}_{n,M}$ in Section \ref{sec:MR}, we have 
\[
\mathbb{K}_{n,M}(\xiscaled,\etascaled) = M \Khat_{n}(M\xiscaled,M\etascaled).
\]

To facilitate the asymptotic analysis in this section, we require a sufficient understanding of the asymptotics of the Pochhammer symbols $(\frac{1}{2} - is)_n (\frac{1}{2} + is)_n$ and $(\frac{1}{2} - it)_n (\frac{1}{2} + it)_n$, and the Legendre functions $Q_{is-\frac{1}{2}}(\cosh \eta)$ and $P_{it-\frac{1}{2}}(\cosh \xi)$. The former will often be derived with the help of the reflection formula for the gamma function, which gives
\begin{subequations} \label{eq:Pochhammer_asy}
  \begin{align}
    ({\textstyle\frac{1}{2}}-is)_{n}({\textstyle\frac{1}{2}}+is)_{n} = {}& \frac{\Gamma(\frac{1}{2} + n - is) \Gamma(\frac{1}{2} + n + is)}{\Gamma({\textstyle\frac{1}{2}} - is) \Gamma(\frac{1}{2} + is)}
    \label{eq:reflection_Gamma} \\
    = {}& (-1)^n \frac{\Gamma(n + \frac{1}{2} + is)}{\Gamma(-n + \frac{1}{2} + is)}, \label{eq:reflection_Gamma_alt}
\end{align}
\end{subequations}
and Stirling's formula. The required asymptotic properties of the Legendre functions are collected in the Appendix. The symmetry
\begin{equation} \label{eq:symmetry_P_Q}
  P_{-it - \frac{1}{2}}(z) = P_{it - \frac{1}{2}}(z), \quad p_n(t; \xi) = -p_n(-t; \xi),
\end{equation}
see \cite[Eq. (1.30) \& (1.31)]{Nemes-Olde_Daalhuis20}, will often be convenient to use as well.
\medbreak

We will treat the four regimes of the DMPK ensemble in the order of insulating, crossover II, crossover I and metallic regime.

\subsection{Insulating regime} \label{subsec:insulator_regime}

We assume that $M=M_n=\alpha_n n$ with all $\alpha_n>0$ and $\alpha_n\to\infty$ as $n \to \infty$. The method of proof is similar to that in \cite[Sect. 2.1.1]{Liu-Wang-Wang18}.

\subsubsection{Normal distribution around classical positions} \label{subsubsec:insulator_normal}

For each $m \in \{ 1, \dotsc, n \}$, recall the classical positions $\zetascaled_{n, m}$ and variables $\xiscaled$ and $\etascaled$ depending on real parameters $\tau, \sigma$ defined in Theorem \ref{thm:insulator}. We take
\begin{align} \label{eq:xi_eta_x_y_insulator}
  \zeta_{n, m} = {}& M \zetascaled, & \xi = {}& \zeta_{n, m} + \sqrt{2} \tau \alpha_n^{\frac{1}{2}} = M \xiscaled, & \eta = {}& \zeta_{n, m} + \sqrt{2} \sigma \alpha_n^{\frac{1}{2}} = M \etascaled.
\end{align}
Note that $\xi > \alpha_n/2 $ and $\eta > \alpha_n/2$ for sufficiently large $n$. We deform the contour $\contC$ into $\contC^{(\insulator, m)}_{\crit} \cup \contC^{(\insulator, m)}_{\crit\Upper} \cup \contC^{(\insulator, m)}_{\Lower} \cup \contC^{(\insulator, m)}_{\Middle} \cup \contC^{(\insulator, m)}_{\Upper}$, where for $m \in \{ 2, \dotsc, n - 1 \}$, we let
\begin{align}
  \contC^{(\insulator, m)}_{\crit} = {}& \{ -i(m - \tfrac{1}{2}) + \alpha_n^{-\frac{1}{2}} e^{i\theta} : \theta \in [0, 2\pi] \}, \notag \\
  \contC^{(\insulator, m)}_{\crit\Upper} = {}& \{ z : \bar{z} \in \contC^{(\insulator, m)}_{\crit} \}, \notag \\
  \contC^{(\insulator, m)}_{\Lower} = {}& \text{polygonal contour with vertices $-i(m + \tfrac{1}{2}) + \alpha_n^{-\frac{1}{2}} i$, $-i(m + \tfrac{1}{2}) + \alpha_n^{-\frac{1}{2}} i + \tfrac{1}{2} e^{-\frac{\pi i}{3}}$,} \notag \\
                     & \text{$-i(m + \tfrac{1}{2}) + \alpha_n^{-\frac{1}{2}} i + \tfrac{1}{2} e^{-\frac{2\pi i}{3}}$, $-i(n - \tfrac{1}{2}) + \tfrac{1}{2} e^{-\frac{\pi i}{3}}$ and $-i(n - \tfrac{1}{2}) + \tfrac{1}{2} e^{-\frac{2\pi i}{3}}$}, \label{eq:defn_C_-} \\
  \contC^{(\insulator, m)}_{\Upper} = {}& \{ z : \bar{z} \in \contC^{(\insulator, m)}_{\Lower} \}, \label{eq:defn_C_+} \\
  \contC^{(\insulator, m)}_{\Middle} = {}& \text{polygonal contour with vertices $-i(m - \tfrac{3}{2}) - \alpha_n^{-\frac{1}{2}}i$, $-i(m - \tfrac{3}{2}) - \alpha_n^{-\frac{1}{2}}i + \tfrac{1}{2} e^{\frac{\pi i}{3}}$,} \notag \\
                     & \text{$-i(m - \tfrac{3}{2}) - \alpha_n^{-\frac{1}{2}}i + \tfrac{1}{2} e^{\frac{2\pi i}{3}}$, $i(m - \tfrac{3}{2}) + \alpha_n^{-\frac{1}{2}}i$, $i(m - \tfrac{3}{2}) + \alpha_n^{-\frac{1}{2}}i + \tfrac{1}{2} e^{-\frac{\pi i}{3}}$,} \notag \\
                     & \text{$i(m - \tfrac{3}{2}) + \alpha_n^{-\frac{1}{2}}i + \tfrac{1}{2} e^{\frac{2\pi i}{3}}$}, \label{eq:defn_C_m}
\end{align}
with positive orientation. If $m = 1$, we take $\contC^{(\insulator, 1)}_{\Middle} = \emptyset$, and the other contours defined as above; if $m = n$, we take $\contC^{(\insulator, n)}_{\Lower} = \contC^{(\insulator, n)}_{\Upper} = \emptyset$, and the other contours defined as above. In Section \ref{subsubsec:voidness}, we extend the definition of $\contC^{(\insulator, m)}_{\Lower}$ and $\contC^{(\insulator, m)}_{\Upper}$ to $m = 0$. We deform the contour $\Sigma$ to $\Sigma^{(m - 1/2 + 2 \alpha_n^{-1/2})}_0 \cup \Sigma^{(m - 1/2 + 2 \alpha_n^{-1/2})}_+ \cup \Sigma^{(m - 1/2 + 2 \alpha_n^{-1/2})}_-$, where
\begin{align}
  \Sigma^{(R)}_0 = {}& \{ x - iR : x \in [-R, R] \}, \label{eq:defn_Sigma_0} \\
  \Sigma^{(R)}_+ = {}& \{ -iR + R + x e^{-\frac{\pi i}{8}} : x \in [0, \infty) \}, \\
  \Sigma^{(R)}_- = {}& \{ -iR - R + x e^{-\frac{7\pi i}{8}} : x \in [0,\infty) \}, \label{eq:defn_Sigma_-}
\end{align}
and orient it rightward. Here and later in some cases, $\Sigma$ is taken not to be a horizontal line, because we want to use some asymptotic results available only in certain regions. See Figure \ref{fig:insulator_normal} for an overview of the contours.
\medbreak

\begin{figure}[htb]
  \centering
  \includegraphics[angle=-90]{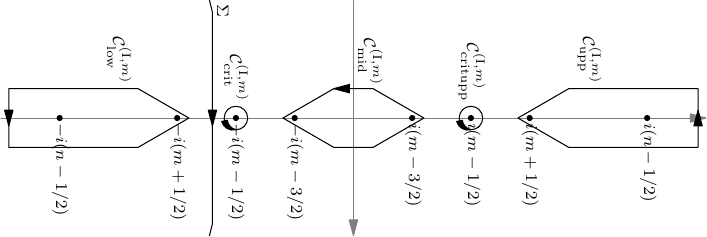}
  \caption{Contours in Section \ref{subsubsec:insulator_normal}. Figure rotated $90^{\circ}$ clockwise.}
  \label{fig:insulator_normal}
\end{figure}
\medbreak

Below we fix $m \in \{ 1, \dotsc, n \}$ and write
\begin{equation}
  \Khat_{n}(\xi, \eta) = I_{\crit} + I_{\crit\Upper} + I_{\Lower} + I_{\Middle} + I_{\Upper},
\end{equation}
where
\begin{equation}
  I_{\star} = \int_{\Sigma} \frac{-ds}{\pi i} \oint_{\contC^{(\insulator, m)}_{\star}} \frac{dt}{2\pi i} \frac{q_n(s; \eta) p_n(t; \xi)}{s-t}, \quad \star = \text{$\crit$, $\crit\Upper$, $\Lower$, $\Middle$ or $\Upper$}.
\end{equation}
We have, by \eqref{eq:asy_Q_small_s} if $\lvert s \rvert$ is small and by \eqref{eq:asy_Q_is} if $\vert s \rvert$ is large, that uniformly for $s \in \Sigma$,
\begin{equation} \label{eq:asy_Q_insulator}
  \sqrt{\sinh \eta} e^{-\frac{M}{n} s^2} Q_{is - \frac{1}{2}}(\cosh \eta) = \sqrt{\pi/2} \, e^{-\alpha_n s^2-is\eta} \frac{\Gamma(is + \frac{1}{2})}{\Gamma(is + 1)} (1 + \bigO(e^{-\eta})).
\end{equation}
We have, by \eqref{eq:asy_P_small_t}, uniformly for $t \in \contC$ with $\Im t \in [-3, 3]$,
\begin{multline} \label{eq:asy_P_insulator_between}
  \sqrt{\sinh \xi} e^{\frac{M}{n} t^2} P_{it - \frac{1}{2}}(\cosh \xi) = \\
  \frac{1}{\sqrt{2\pi}} e^{\alpha_n t^2 + it\xi} \frac{\Gamma(it)}{\Gamma(it + \frac{1}{2})} (1 + \bigO(e^{-\xi})) + \frac{1}{\sqrt{2\pi}} e^{\alpha_n t^2 - it\xi} \frac{\Gamma(-it)}{\Gamma(-it + \frac{1}{2})} (1 + \bigO(e^{-\xi})),
\end{multline}
by \eqref{eq:asy_P_it}, uniformly for $t \in \contC$ with $\Im t \leq -3$, 
\begin{equation} \label{eq:asy_P_insulator}
  \sqrt{\sinh \xi} e^{\frac{M}{n} t^2} P_{it - \frac{1}{2}}(\cosh \xi) = \frac{1}{\sqrt{2\pi}} e^{\alpha_n t^2 + it\xi} \frac{\Gamma(it)}{\Gamma(it + \frac{1}{2})} (1 + \bigO(e^{-\xi})),
\end{equation}
and, by \eqref{eq:asy_P_it} and \eqref{eq:symmetry_P_Q}, uniformly for $t \in \contC$ with $\Im t \geq 3$,
\begin{equation} \label{eq:asy_P_insulator_reflection}
  \sqrt{\sinh \xi} e^{\frac{M}{n} t^2} P_{it - \frac{1}{2}}(\cosh \xi) = \frac{1}{\sqrt{2\pi}} e^{\alpha_n t^2 - it\xi} \frac{\Gamma(-it)}{\Gamma(-it + \frac{1}{2})} (1 + \bigO(e^{-\xi})).
\end{equation}
Note that, by \eqref{eq:asy_P_insulator_between}, \eqref{eq:asy_P_insulator} also holds for $t \in \contC^{(\insulator, m)}_{\crit}$, even if $m = 1, 2, 3$ and $\Im t > -3$.

By \eqref{eq:reflection_Gamma} and Stirling's formula, we have that for all $t \in \contC$ with $\Im t \leq 0$,
\begin{equation} \label{eq:Stirling_t}
  (\tfrac{1}{2}-it)_{n}(\tfrac{1}{2}+it)_{n} = \frac{\cosh(\pi t)}{\sqrt{\pi/2}} \Gamma(\tfrac{1}{2} + n - it) e^{(n + it)(\log(n + it) - 1)} (1 + \bigO(n^{-1})),
\end{equation}
and upon substituting $t$ by $s$, \eqref{eq:Stirling_t} is also valid for $s \in \Sigma$. For $s \in \Sigma$ such that $\lvert s + in \rvert > \epsilon n$, we have
\begin{equation} \label{eq:Stirling_s}
  (\tfrac{1}{2}-is)_{n}(\tfrac{1}{2}+is)_{n} = 2\cosh(\pi s) e^{(n - is)(\log(n - is) - 1) + (n + is)(\log(n + is) - 1)} (1 + \bigO(n^{-1})),
\end{equation}
and upon substituting $s$ by $t$, \eqref{eq:Stirling_s} is also valid for all $t \in \contC$ with $\lvert \Im t \rvert \leq (1 - \epsilon/2) n$. For $t \in \contC$ with $\Im t \leq 0$, we can simply use the symmetry $t\mapsto -t$ of $(\frac{1}{2} - it)_n (\frac{1}{2} - it)_n$. In particular, if $u, v$ are in a bounded subset of $\compC$, and
\begin{align} \label{eq:change_of_variables_crossover}
  s = -im + u, \qquad t = -im + v, 
\end{align}
then for $m \in \{ 1, \dotsc, n \}$ with $m < (1 - \epsilon)n$ and sufficiently small $\epsilon \in (0, 1)$,
\begin{equation} \label{eq:Stirling_bulk}
  \frac{(\frac{1}{2}-is)_{n}(\frac{1}{2}+is)_{n}}{(\frac{1}{2}-it)_{n}(\frac{1}{2}+it)_{n}} = \frac{\cosh(\pi u)}{\cosh(\pi v)} \times
  \begin{cases}

        \frac{\Gamma(\frac{1}{2} + n - m - iu)}{\Gamma(\frac{1}{2} + n - m - iv)} \frac{\exp(iu \log(n + m)}{\exp(iv \log(n + m))} (1 + \bigO(n^{-1})), & \text{all $m = 1, \dotsc, n$}, \\
                \frac{\exp(iu \log(\frac{n + m}{n - m}))}{\exp(iv \log(\frac{n + m}{n - m}))} (1 + \bigO(n^{-1})), & \frac{m}{n} < 1 - \epsilon.
  \end{cases}
\end{equation}

Let $\zeta \in (0,\infty)$ and consider the following equation for $z$ on $i\R_{\leq 0}$:
\begin{equation} \label{eq:saddle_insulator}
    \frac{d}{dz} F(z; \zeta) = 0, \quad F(z; \zeta) = 
        \alpha_n z^2 + iz\zeta - \log \Gamma(\tfrac{1}{2} + n - iz) 
        -(n + iz)(\log(n + iz) - 1).
\end{equation}
It is easy to see that as a function of $y \in (0, n + \frac{1}{2})$, the function $i F'(-iy; \zeta)$ is real valued and increasing. We have that if $\zeta = \zeta_{n, m}$ then $z = -i(m - \frac{1}{2})$ is a solution of \eqref{eq:saddle_insulator}. Hence, by \eqref{eq:asy_Q_insulator}, \eqref{eq:asy_P_insulator_between}, \eqref{eq:asy_P_insulator} and \eqref{eq:Stirling_s}, \eqref{eq:Stirling_t}, we find that $-i(m - \frac{1}{2})$ is the saddle point for both $q_n(s; \zeta_{n, m})$ and $p_n(t; \zeta_{n, m})$. From the estimates \eqref{eq:asy_Q_insulator}--\eqref{eq:Stirling_t}, we can estimate the value of $q_n(s; \eta)$ for $s \in \Sigma$ and $p_n(t; \xi)$ for $t \in \contC^{(\insulator, m)}_{\Lower} \cup \contC^{(\insulator, m)}_{\Middle} \cup \contC^{(\insulator, m)}_{\Upper}$. As $s \in \Sigma$ moves away from $z = -i(m - \frac{1}{2})$, $\lvert q_n(s; \eta) \rvert$ vanishes exponentially fast, and as $t \in \contC \cap \compC_-$ moves away from $z = -i(m - \frac{1}{2})$, $\lvert p_n(s; \xi) \rvert$ vanishes exponentially fast. From the estimate, if $\xi$ and $\eta$ are as in \eqref{eq:xi_eta_x_y_insulator}, we have that for $s \in \Sigma$ and $t \in \contC^{(\insulator, m)}_{\Lower} \cup \contC^{(\insulator, m)}_{\Middle} \cup \contC^{(\insulator, m)}_{\Upper}$,
\begin{equation} \label{eq:est_p_n_q_n}
  \lvert q_n(s; \eta) p_n(t; \xi) \rvert = \frac{e^{(m - \frac{1}{2}) \sqrt{2} \tau \alpha_n^{\frac{1}{2}}}}{e^{(m - \frac{1}{2}) \sqrt{2} \sigma \alpha_n^{\frac{1}{2}}}} o(e^{-\epsilon \alpha_n}),
\end{equation}
where $\epsilon > 0$ is sufficiently small and independent of $m$, $n$ or $\alpha_n$ (e.g., $\epsilon = \frac{1}{10}$). Also, we have that $\lvert q_n(s; \eta) \rvert$ vanishes exponentially fast as $s$ moves away from $-i(m - \frac{1}{2})$, and $\lvert p_n(t; \xi) \rvert$ vanishes exponentially fast as $s$ moves away from $\pm (m - \frac{1}{2})$. Hence, we can conclude that
\begin{equation} \label{eq:est_I_+_I_-}
  \lvert I_{\Lower} + I_{\Middle} + I_{\Upper} \rvert = \frac{e^{(m - \frac{1}{2}) \sqrt{2} \tau \alpha_n^{\frac{1}{2}}}}{e^{(m - \frac{1}{2}) \sqrt{2} \sigma \alpha_n^{\frac{1}{2}}}} o(e^{-\epsilon \alpha_n}).
\end{equation}
It remains to evaluate the integrals $I_{\crit}$ and $I_{\crit\Upper}$. Since for  $s\in \Sigma$, $\lvert q_n(s; \eta) \rvert$ vanishes exponentially fast as $s$ moves away from $-i(m - \frac{1}{2})$, it suffices to consider the integral for $s \in \Sigma^{(m - 1/2 + 2 \alpha_n^{-1/2})}_0$. After the change of variables
\begin{align} \label{eq:chqnge_of_variable_insulator}
  s \mapsto -i(m - \tfrac{1}{2}) + \alpha_n^{-\frac{1}{2}} u,\qquad t \mapsto  -i(m - \tfrac{1}{2}) + \alpha_n^{-\frac{1}{2}} v, 
\end{align}
we have
\begin{align} \label{eq:from_Sigma_to_Sigma_0}
  \alpha_n^{\frac{1}{2}} I_{\crit} = & - \int_{\substack{\Im u = -2,\ \lvert \Re u \rvert \\ \leq (m - \frac{1}{2})M^{1/2}_n + 2}} \frac{du}{\pi i} \oint_{\lvert v \rvert = 1} \frac{dv}{2\pi i} \frac{q_n(s; \eta) p_n(t; \xi)}{u - v}
  + \frac{e^{(m - \frac{1}{2}) \sqrt{2} \tau \alpha_n^{\frac{1}{2}}}}{e^{(m - \frac{1}{2}) \sqrt{2} \sigma \alpha_n^{\frac{1}{2}}}} o(e^{-\epsilon \alpha_n}).
\end{align}
where the second term represents the contribution from $\Sigma^{(m - 1/2 + 2 \alpha_n^{-1/2})}_+ \cup \Sigma^{(m - 1/2 + 2 \alpha_n^{-1/2})}_-$, with $\epsilon > 0$ being a constant independent of $m$, $n$ or $\alpha_n$. Next, using \eqref{eq:asy_Q_insulator} and \eqref{eq:asy_P_insulator} for the Legendre functions and \eqref{eq:Stirling_bulk} for the Pochhammer symbols, after the change of variables \eqref{eq:chqnge_of_variable_insulator}, we have
\begin{equation}
  \frac{e^{(m - \frac{1}{2}) \sqrt{2} \sigma \alpha_n^{\frac{1}{2}}}}{e^{(m - \frac{1}{2}) \sqrt{2} \tau \alpha_n^{\frac{1}{2}}}} q_n(s; \eta) p_n(t; \xi) = \frac{e^{v^2 + i\sqrt{2} \tau v}}{e^{u^2 + i\sqrt{2} \sigma u}} \frac{u}{v} + \bigO(\alpha_n^{-\frac{1}{2}}), 
\end{equation}
uniformly for $u$ and $v$ in a compact subset of $\compC$. Hence, \eqref{eq:from_Sigma_to_Sigma_0} implies that
\begin{equation} \label{eq:est_I_0}
  \begin{split}
   \frac{e^{(m - \frac{1}{2}) \sqrt{2} \sigma \alpha_n^{\frac{1}{2}}}}{e^{(m - \frac{1}{2}) \sqrt{2} \tau \alpha_n^{\frac{1}{2}}}} \alpha_n^{\frac{1}{2}} I_{\crit} = {}& i \int_{\substack{\Im u = -2,\ \lvert \Re u \rvert \\ \leq (m - \frac{1}{2})M^{1/2}_n + 2}} \frac{du}{2\pi i} \oint_{\lvert v \rvert = 1} \frac{dv}{2\pi i} \frac{u}{v} \frac{e^{v^2 + i \sqrt{2} \tau v}}{e^{u^2 + i \sqrt{2} \sigma u}} \frac{1}{u - v} + \bigO(\alpha_n^{-\frac{1}{2}}) \\
    = {}& i \int^{\infty - 2i}_{-\infty - 2i} \frac{du}{2\pi i} \frac{1}{e^{u^2 + i \sqrt{2} \sigma u}} + \bigO(\alpha_n^{-\frac{1}{2}}) \\
    = {}& \frac{e^{-\frac{\sigma^2}{2}}}{2\sqrt{\pi}} + \bigO(\alpha_n^{-\frac{1}{2}}).
  \end{split}
\end{equation}
The evaluation of $I_{\crit\Upper}$ can be done similarly. Using  \eqref{eq:symmetry_P_Q} and \eqref{eq:chqnge_of_variable_insulator}, we have
\begin{align} \label{eq:est_I_0*}
     \alpha_n^{\frac{1}{2}} I_{\crit\Upper} 
    &=  \int_{\substack{\Im u = -2,\ \lvert \Re u \rvert \leq \\ (m - \frac{1}{2})M^{1/2}_n + 2}} \frac{du}{\pi i} \oint_{\lvert v \rvert = 1} \frac{dv}{2\pi i} \frac{q_n(s; \eta) p_n(t; \xi)}{-2i (m - \frac{1}{2}) \alpha_n^{\frac{1}{2}} + u + v} + \frac{e^{(m - \frac{1}{2}) \sqrt{2} \tau \alpha_n^{\frac{1}{2}}}}{e^{(m - \frac{1}{2}) \sqrt{2} \sigma \alpha_n^{\frac{1}{2}}}} o(e^{-\epsilon \alpha_n}) \notag \\ 
   & =  \frac{e^{(m - \frac{1}{2}) \sqrt{2} \tau \alpha_n^{\frac{1}{2}}}}{e^{(m - \frac{1}{2}) \sqrt{2} \sigma \alpha_n^{\frac{1}{2}}}} \bigO(\alpha_n^{-\frac{1}{2}}).
\end{align}
By \eqref{eq:est_I_+_I_-}, \eqref{eq:est_I_0} and \eqref{eq:est_I_0*}, we obtain that as $n \to \infty$,
\begin{equation}
  \frac{e^{(m - \frac{1}{2}) \sqrt{2} \sigma \alpha_n^{\frac{1}{2}}}}{e^{(m - \frac{1}{2}) \sqrt{2} \tau \alpha_n^{\frac{1}{2}}}} \sqrt{2} \alpha_n^{\frac{1}{2}} \Khat_{n}(\xi, \eta) \to \frac{e^{-\frac{\sigma^2}{2}}}{\sqrt{2 \pi}}.
\end{equation}
Since $\Khat_{n}(\xi, \eta) = (n \alpha_n)^{-1} \mathbb{K}_n(\xiscaled, \etascaled)$, we obtain part \ref{enu:thm:insulator_1} of Theorem \ref{thm:insulator} with the cocycle 
\[
\frac{e^{(m - \frac{1}{2}) \sqrt{2} \sigma \alpha_n^{\frac{1}{2}}}}{e^{(m - \frac{1}{2}) \sqrt{2} \tau \alpha_n^{\frac{1}{2}}}}.
\]

\subsubsection{Emptiness between classical positions} \label{subsubsec:voidness}

Recall the particles $\{ \lambdascaled_j \}^n_{j = 1}$ considered in Part \ref{enu:thm:insulator_2} of Theorem \ref{thm:insulator}. Let $\lambdascaledM_j = M\lambdascaled_j$. To prove part \ref{enu:thm:insulator_2} of Theorem \ref{thm:insulator}, it suffices to show that for any $\epsilon > 0$, there exists $L = L(\epsilon)>0$, such that
\begin{equation} \label{eq:gap_est}
  \mathbb{P}\{\mathrm{some}\ \lambdascaledM_j \in I_{\star}\}  \leq \E[\# \text{of $\lambdascaledM_j$'s in $I_{\star}$}] = \int_{I_{\star}} \Khat_{n}(\zeta, \zeta) d\zeta < \epsilon,
\end{equation}
for 
$I_{\star} = (0, \zeta_{n, 1} - L \alpha_n^{\frac{1}{2}})$, $I_{\star}=(\zeta_{n, m - 1} + L \alpha_n^{\frac{1}{2}}, \zeta_{n, m} - L \alpha_n^{\frac{1}{2}})$  with $m \in \{ 2, \dotsc, n \}$, or $I_{\star}=(\zeta_{n, n} + L \alpha_n^{\frac{1}{2}},\infty)$. This idea will be the same as in the proof of \cite[Thm. 2.1]{Liu-Wang-Wang18}.
\medbreak

We first consider $I_{\star} = (\zeta_{n, m - 1} + L \alpha_n^{\frac{1}{2}}, \zeta_{n, m} - L \alpha_n^{\frac{1}{2}})$. Let $\zeta \in (\zeta_{n, m - 1} + 4 \alpha_n^{\frac{1}{2}}, \zeta_{n, m} - 4 \alpha_n^{\frac{1}{2}})$, and denote
\begin{equation}
  \gap = \frac{\alpha_n^{-1}}{2} \min \{ \zeta_{n, m} - \zeta, \zeta - \zeta_{n, m - 1} \},
\end{equation}
so that $\gap \in (2 \alpha_n^{-\frac{1}{2}}, \frac{1}{2}]$. Let $z\in\C$ be the purely imaginary number between $-i(m - \frac{3}{2})$ and $-i(m - \frac{1}{2})$ that satisfies equation \eqref{eq:saddle_insulator}. Unique existence of such $z$ is due to the increasing property of $i F'(-iy; \zeta)$ stated below \eqref{eq:saddle_insulator}. We deform the contour $\contC$ into $\contC^{(\insulator, m - 1)}_{\Lower} \cup \contC^{(\insulator, m)}_{\Middle} \cup \contC^{(\insulator, m - 1)}_{\Upper}$, defined in \eqref{eq:defn_C_-}--\eqref{eq:defn_C_m}, and we deform the contour $\Sigma$ into $\Sigma^{(iz)}_0 \cup \Sigma^{(iz)}_+ \cup \Sigma^{(iz)}_-$, defined in \eqref{eq:defn_Sigma_0}--\eqref{eq:defn_Sigma_-}. See Figure \ref{fig:insulator_void} for an overview of the contours. 

\begin{figure}[htb]
  \centering
  \includegraphics[angle=-90]{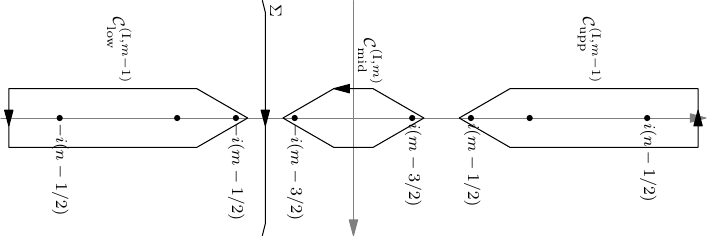}
    \caption{Contours in Section \ref{subsubsec:voidness}. Figure rotated $90^{\circ}$ clockwise.}
    \label{fig:insulator_void}
\end{figure}

We note that $z$ is the saddle point for both $q_n(s; \zeta)$ and $p_n(t; \zeta)$.  Comparing to the contour deformation in Section \ref{subsubsec:insulator_normal}, although the contour shapes are different, we still have that for all $s \in \Sigma$, the uniform estimate in \eqref{eq:asy_Q_insulator} holds with $\eta$ replaced by $\zeta$, and that for all $t \in \contC$, the uniform estimates \eqref{eq:asy_P_insulator_between}--\eqref{eq:asy_P_insulator_reflection} hold with $\xi$ replaced by $\zeta$. These estimates, together with the estimates of $(\frac{1}{2}-is)_{n}(\frac{1}{2}+is)_{n}$ for $s \in \Sigma$ and $(\frac{1}{2}-it)_{n}(\frac{1}{2}+it)_{n}$ for $t \in \contC$ by \eqref{eq:Pochhammer_asy} and Stirling's formula, we have that for some $C, \epsilon > 0$, for sufficiently large $n$,
\begin{equation} \label{eq:est_pq_gap}
  \lvert q_n(s; \zeta) p_n(t; \zeta) \rvert \leq C e^{-\epsilon \alpha_n (\lvert s - z \rvert^2 + \lvert t - z \rvert^2)}, \quad s \in \Sigma \text{ and } t \in \contC \cap \compC_-.
\end{equation}
Furthermore, by \eqref{eq:symmetry_P_Q}, we also have the estimate
\begin{equation} \label{eq:est_pq_gap_up}
  \lvert q_n(s; \zeta) p_n(t; \zeta) \rvert \leq C e^{-\epsilon \alpha_n (\lvert s - z \rvert^2 + \lvert t + z \rvert^2)}, \quad s \in \Sigma \text{ and } t \in \contC \cap \compC_+.
\end{equation}
Plugging in \eqref{eq:est_pq_gap} and \eqref{eq:est_pq_gap_up} into \eqref{eq:Khat}, we have that for some $C', \epsilon' > 0$, for sufficiently large $n$,
\begin{equation} \label{eq:gap_prob_middle}
  \left\lvert \Khat_n(\zeta, \zeta) \right\rvert \leq \frac{C'}{\gap} \alpha_n^{-1} e^{-\epsilon' \alpha_n \gap^2}.
\end{equation}
This estimate shows that if $L = L(\epsilon)$ is sufficiently large, then for sufficiently large $n$, \eqref{eq:gap_est} also holds for $I_{\star} = (\zeta_{n, m - 1} + L \alpha_n^{\frac{1}{2}}, \zeta_{n, m} - L \alpha_n^{\frac{1}{2}})$.

Similarly, let $\zeta \in (0, \zeta_{n, 1} - 4 \alpha_n^{\frac{1}{2}})$, and denote
\begin{equation}
  \gap' = \frac{\alpha_n^{-1}}{2} (\zeta_{n, 1} - \zeta).
\end{equation}
Let $z$ be the imaginary number between $0$ and $-\frac{i}{2}$ that satisfies \eqref{eq:saddle_insulator}. The unique existence of $z$ is due to the increasing property of $i F'(-iy; \zeta)$ stated below \eqref{eq:saddle_insulator}. We deform the contour $\contC$ into $\contC^{(\insulator, 0)}_{\Lower} \cup \contC^{(\insulator, 0)}_{\Upper}$, defined  in \eqref{eq:defn_C_-} and \eqref{eq:defn_C_+}, and deform the contour $\Sigma$ into $\Sigma^{(iz)}_0 \cup \Sigma^{(iz)}_+ \cup \Sigma^{(iz)}_-$. We have that \eqref{eq:est_pq_gap} and \eqref{eq:est_pq_gap_up} hold, and for some $C', \epsilon' > 0$, for sufficiently large $n$,
\begin{equation} \label{eq:gap_prob_leftmost}
  \left\lvert \Khat_n(\zeta, \zeta) \right\rvert \leq \frac{C'}{\gap'} \alpha_n^{-1} e^{-\epsilon' \alpha_n (\gap')^2}.
\end{equation}
This confirms that if $L = L(\epsilon)$ is sufficiently large, then for sufficiently large $n$, \eqref{eq:gap_est} holds for $I_{\star} = (0, \zeta_{n, 1} - L \alpha_n^{\frac{1}{2}})$. We note that if $\zeta$ is very close to $0$, we need to estimate $p_n(t; \zeta)$ and $q_n(s; \zeta)$ by \eqref{eq:P_nu(x)_hard}--\eqref{eq:K0_asy}, and \eqref{eq:gap_prob_leftmost} still holds.

Finally, let $\zeta \in (\zeta_{n, n} + 4 \alpha_n^{\frac{1}{2}},\infty)$, and denote
\begin{equation}
  \gap'' = \frac{\alpha_n^{-1}}{2} (\zeta - \zeta_{n, n}).
\end{equation}
Let $z$ be the imaginary number between $-i(n - \frac{1}{2})$ and $-i(n + \frac{1}{2})$ that satisfies \eqref{eq:saddle_insulator} if $\gap'' < \frac{1}{2}$, and take $z = -in$ if $\gap'' \geq \frac{1}{2}$. Here we note that $z$ is the saddle point for both $q_n(s; \zeta)$ and $p_n(t; \zeta)$ if $\gap'' < \frac{1}{2}$ while $z$ is not a saddle point if $\gap'' \geq \frac{1}{2}$. We deform the contour $\contC$ into $\contC^{(\insulator, 0)}_{\Lower} \cup \contC^{(\insulator, 0)}_{\Upper}$ and deform the contour $\Sigma$ into $\Sigma^{(iz)}_0 \cup \Sigma^{(iz)}_+ \cup \Sigma^{(iz)}_-$, formally the same as in last case, but now $\Sigma$ is below $\contC$, rather than through $\contC$. We have analogous to \eqref{eq:est_pq_gap} and \eqref{eq:est_pq_gap_up}, for some $C, \epsilon > 0$, for sufficiently large $n$,
\begin{equation}
  \lvert q_n(s; \zeta) p_n(t; \zeta) \rvert \leq
  \begin{cases}
    C e^{-\epsilon \alpha_n (\lvert s - z \rvert^2 + \lvert t \pm z \rvert^2)}, & s \in \Sigma, t \in \contC \cap \compC_{\pm} \text{ and } S'' < \frac{1}{2}, \\
    C e^{-\epsilon \alpha_n S'' (\lvert s - z \rvert + \lvert t \pm z \rvert)}, & s \in \Sigma, t \in \contC \cap \compC_{\pm} \text{ and } S'' \geq \frac{1}{2}, 
  \end{cases}
\end{equation}
and for some $C', \epsilon' > 0$, for sufficiently large $n$,
\begin{equation} \label{eq:gap_prob_rightmost}
  \left\lvert \Khat_n(\zeta, \zeta) \right\rvert \leq \frac{C'}{\gap''} \alpha_n^{-1} 
  \begin{cases}
        e^{-\epsilon' \alpha_n (\gap'')^2}, & \gap'' < \frac{1}{2}, \\
        e^{-\epsilon' \alpha_n \gap''}, & \gap'' \geq \frac{1}{2}.
  \end{cases}
\end{equation}
This confirms that if $L = L(\epsilon)$ is sufficiently large, then for sufficiently large $n$, \eqref{eq:gap_est} holds for $I_{\star} = (\zeta_{n, n} + L \alpha_n^{\frac{1}{2}},\infty)$.

Part \ref{enu:thm:insulator_2} of Theorem \ref{thm:insulator} is therefore proven.

\subsection{Crossover II} \label{subsec:crossoverII}

We assume that $M = M_n = n \alpha$ with fixed $\alpha > 0$. Unlike in the insulating regime, the local statistics in the crossover regime have different features in the bulk (for the particles in the middle), at the soft edge (for the rightmost particles) and at the hard edge (for the leftmost particles).

\subsubsection{Bulk} \label{subsubsec:crossover_bulk}

Let $\eps>0$ be sufficiently small and let $m$ be an integer such that $\epsilon n < m < (1 - \epsilon) n$ as in part \ref{enu:thm:crossover_1} of Theorem \ref{thm:crossover}. We take
\begin{align} \label{eq:xi_eta_crossover}
  \zeta_m = {}& 2m \alpha, &\xi = {}& \zeta_m + \tilde{\tau}, & \eta = {}& \zeta_m + \tilde{\sigma}, &\tilde{\tau} = {}& \tau + \log\frac{n + m}{n - m}, &\tilde{\sigma} = {}& \sigma + \log\frac{n + m}{n - m}.
\end{align}
Note that $\xi$ and $\eta$ are larger than $\epsilon \alpha n$. We deform the contour $\contC$ into $\contC^{(\crossover, m)}_{\Lower} \cup \contC^{(\crossover, m)}_{\Middle} \cup \contC^{(\crossover, m)}_{\Upper}$, where for $1 \leq m \leq n - 1$,
\begin{align}
  \contC^{(\crossover, m)}_{\Lower} = {}& \text{rectangular contour with vertices $-i(m + \tfrac{1}{2}) + \tfrac{1}{2} e^{\frac{\pi i}{4}}$, $-i(m + \tfrac{1}{2}) + \tfrac{1}{2} e^{\frac{3\pi i}{4}}$}, \notag \\
                                        & \text{$-i(n - \tfrac{1}{2}) + \tfrac{1}{2} e^{-\frac{\pi i}{4}}$ and $-i(n - \tfrac{1}{2}) + \tfrac{1}{2} e^{-\frac{3\pi i}{4}}$}, \\
  \contC^{(\crossover, m)}_{\Upper} = {}& \{ z : \bar{z} \in \contC^{(\crossover, m)}_{\Lower} \}, \\
  \contC^{(\crossover, m)}_{\Middle} = {}& \text{rectangular contour with vertices $i(m - \tfrac{1}{2}) + \tfrac{1}{2} e^{\frac{\pi i}{4}}$, $i(m - \tfrac{1}{2}) + \tfrac{1}{2} e^{\frac{3\pi i}{4}}$}, \notag \\
                                        & \text{$-i(m - \tfrac{1}{2}) + \tfrac{1}{2} e^{-\frac{\pi i}{4}}$ and $-i(m - \tfrac{1}{2}) + \tfrac{1}{2} e^{-\frac{3\pi i}{4}}$},
\end{align}
with positive orientation. In Section \ref{subsubsec:crossover_hard}, we will extend the definition of $\contC^{(\crossover, m)}_{\Lower}$ and $\contC^{(\crossover, m)}_{\Upper}$ to $m = 0$, and in Section \ref{subsubsec:crossover_soft}, we will extend the definition of $\contC^{(\crossover, m)}_{\Middle}$ to $m = n$. We deform the contour $\Sigma$ into  $\Sigma^{(m)}_0 \cup \Sigma^{(m)}_+ \cup \Sigma^{(m)}_-$, defined as in \eqref{eq:defn_Sigma_0}--\eqref{eq:defn_Sigma_-}. See Figure \ref{fig:crossover_bulk} for an overview of the contours.

\begin{figure}[htb]
  \centering
    \includegraphics[angle=-90]{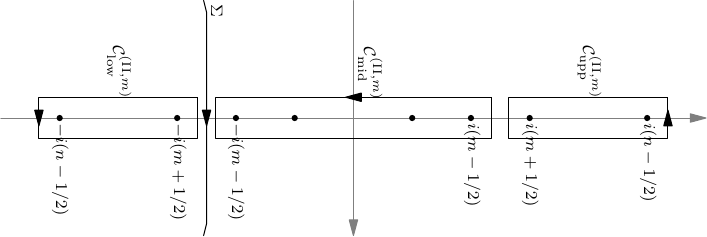}
    \caption{Contours in Section \ref{subsubsec:crossover_bulk}. Figure rotated $90^{\circ}$ clockwise.}
    \label{fig:crossover_bulk}
\end{figure}
\medbreak

Analogous to \eqref{eq:asy_Q_insulator}, we have, by \eqref{eq:asy_Q_small_s} if $\lvert s \rvert$ is small and by \eqref{eq:asy_Q_is} if $\vert s \rvert$ is large, that uniformly for $s \in \Sigma$,
\begin{equation} \label{eq:asy_Q_crossover}
  \sqrt{\sinh \eta} e^{-\frac{M}{n} s^2} Q_{is - \frac{1}{2}}(\cosh \eta) = \sqrt{\frac{\pi}{2}} \, e^{-\alpha s^2-is\eta} \frac{\Gamma(is + \frac{1}{2})}{\Gamma(is + 1)} (1 + \bigO(e^{-\eta})).
\end{equation}
Analogous to \eqref{eq:asy_P_insulator_between}--\eqref{eq:asy_P_insulator_reflection}, by \eqref{eq:asy_P_small_t}, we have, uniformly for $t \in \contC$ with $\Im t \in [-3, 3]$, 
\begin{multline} \label{eq:asy_P_crossover_between}
  \sqrt{\sinh \xi} e^{\frac{M}{n} t^2} P_{it - \frac{1}{2}}(\cosh \xi) = \\
  \frac{1}{\sqrt{2\pi}} e^{\alpha t^2 + it\xi} \frac{\Gamma(it)}{\Gamma(it + \frac{1}{2})} (1 + \bigO(e^{-\xi})) + \frac{1}{\sqrt{2\pi}} e^{\alpha t^2 - it\xi} \frac{\Gamma(-it)}{\Gamma(-it + \frac{1}{2})} (1 + \bigO(e^{-\xi})),
\end{multline}
by \eqref{eq:asy_P_it}, uniformly for $t \in \contC$ with $\Im t \leq -3$, 
\begin{equation} \label{eq:asy_P_crossover}
  \sqrt{\sinh \xi} e^{\frac{M}{n} t^2} P_{it - \frac{1}{2}}(\cosh \xi) = \frac{1}{\sqrt{2\pi}} e^{\alpha t^2 + it\xi} \frac{\Gamma(it)}{\Gamma(it + \frac{1}{2})} (1 + \bigO(e^{-\xi})),
\end{equation}
and, by \eqref{eq:asy_P_it} and \eqref{eq:symmetry_P_Q}, uniformly for $t \in \contC$ with $\Im t \geq 3$,
\begin{equation} \label{eq:asy_P_crossover_reflection}
  \sqrt{\sinh \xi} e^{\frac{M}{n} t^2} P_{it - \frac{1}{2}}(\cosh \xi) = \frac{1}{\sqrt{2\pi}} e^{\alpha t^2 - it\xi} \frac{\Gamma(-it)}{\Gamma(-it + \frac{1}{2})} (1 + \bigO(e^{-\xi})).
\end{equation}
With the same change of variables $s \mapsto -im + u$ and $t \mapsto -im + v$ as in \eqref{eq:change_of_variables_crossover}, we have then have uniformly for $u$ and $v$ in a compact subset of $\C$,
\begin{equation}
  t \sqrt{\sinh \eta} e^{-\frac{M}{n} s^2} Q_{is-\frac{1}{2}}(\cosh \eta) \sqrt{\sinh \xi} e^{\frac{M}{n} t^2} P_{it-\frac{1}{2}}(\cosh) = \frac{e^{\tilde{\tau} m}}{e^{\tilde{\sigma} m}} \frac{1}{2i} \frac{e^{\alpha v^2 + i \tilde{\tau} v}}{e^{\alpha u^2 + i \tilde{\sigma} u}} (1 + \bigO(n^{-1}).
\end{equation}
Moreover, together with the estimates of the Pochhammer symbols in \eqref{eq:Stirling_t} and \eqref{eq:Stirling_s}, we find that $q_n(s; \xi)$ vanishes exponentially fast as $s \in \Sigma$ moves away from $-im$, and, using the symmetry $t\mapsto-t$ that $p_n(t; \xi)$ vanishes exponentially fast as $t \in \contC$ moves away from $\pm im$. Hence, the integral in \eqref{eq:Khat} concentrates around $-im$ for $s$ and $\pm im$ for $t$. On the other hand, the $(s - t)^{-1}$ factor in the integral \eqref{eq:Khat} becomes larger if both $s$ and $t$ are around $-im$ and smaller if $s$ is around $-im$ and $t$ is around $irn$, so the main contribution to the integral in \eqref{eq:Khat} occurs when both $s$ and $t$ are around $-im$. Using the estimates of the Pochhammer symbols in \eqref{eq:Stirling_bulk}, we can therefore conclude that
\begin{align} \label{eq:crossover_bulk_result}
  \Khat_{n}(\xi, \eta) = &\frac{e^{\tilde{\tau} m}}{e^{\tilde{\sigma} m}} i \int^{\infty}_{-\infty} \frac{du}{2\pi i} \oint_{C^{\crossover, \infty}_{\Lower} \cup C^{\crossover, \infty}_{\Upper}} \frac{dv}{2\pi i} \frac{\cosh(\pi u)}{\cosh(\pi v)} \frac{e^{-\alpha u^2 + i \left( -\tilde{\sigma} + \log \left( \frac{n + m}{n - m} \right) \right) u}}{e^{-\alpha v^2 + i \left( -\tilde{\tau} + \log \left( \frac{n + m}{n - m} \right) \right) v}} \frac{1}{u - v} \notag \\
  & \times (1 + \bigO(n^{-1})),
\end{align}
where the contours are
\begin{align}
  C^{(\crossover, \infty)}_{\Lower} = {}& \{ -\tfrac{i}{2} + \tfrac{1}{2} e^{\frac{\pi i}{4}} + iy : y \in (-\infty, 0] \} \cup \{ -\tfrac{i(1 - \sqrt{2}/2)}{2} - x : x \in [-\tfrac{\sqrt{2}}{4}, \tfrac{\sqrt{2}}{4}] \} \notag \\
                                      & \cup \{ -\tfrac{i}{2} + \tfrac{1}{2} e^{\frac{3\pi i}{4}} - iy : y \in [0, \infty) \}, \\
  C^{(\crossover, \infty)}_{\Upper} = {}& \{ \tfrac{i}{2} + \tfrac{1}{2} e^{-\frac{3\pi i}{4}} - iy : y \in (-\infty, 0] \} \cup \{ \tfrac{i(1 - \sqrt{2}/2)}{2} + x : x \in [-\tfrac{\sqrt{2}}{4}, \tfrac{\sqrt{2}}{4}] \} \notag \\
                                      & \cup \{ -\tfrac{i}{2} + \tfrac{1}{2} e^{-\frac{\pi i}{4}} + iy : y \in [0, \infty) \}.
\end{align}
Part \ref{enu:thm:crossover_1} of Theorem \ref{thm:crossover} then follows after the change of variables $u \to -iu$ and $v \to -iv$ with cocycle $e^{\tilde{\sigma} m}/e^{\tilde{\tau} m}$.

\subsubsection{Hard edge} \label{subsubsec:crossover_hard}

We take $\xi=\tau$ and  $\eta=\sigma$ in a compact subset of $[0,\infty)$. We deform the contour $\contC$ into $\contC^{(\crossover, 0)}_{\Lower} \cup \contC^{(\crossover, 0)}_{\Upper}$, and let $\Sigma$ be the real line. We use \eqref{eq:P_nu(x)_hard} and \eqref{eq:I0_asy} to estimate $\sqrt{\sinh \xi} e^{\frac{M}{n} t^2} P_{it - \frac{1}{2}}(\cosh \xi)$ for $t \in \contC$ and use \eqref{eq:Q_nu(x)_hard} and \eqref{eq:K0_asy} to estimate $\sqrt{\sinh \eta} e^{-\frac{M}{n} s^2} Q_{is - \frac{1}{2}}(\cosh \eta)$ for $s \in \Sigma$. In addition, we estimate $(\frac{1}{2}-is)_{n}(\frac{1}{2}+is)_{n}$ for $s \in \Sigma$ and $(\frac{1}{2}-it)_{n}(\frac{1}{2}+it)_{n}$ for $t \in \contC$ in the same way as in Section \ref{subsubsec:crossover_bulk}. Then we find that $p_n(s; \xi)$ and $q_n(t; \eta)$ vanish exponentially fast as $s$ and $t$ move away from $0$. Hence, the integral in \eqref{eq:Khat} concentrates around $0$. 
\medbreak

Using the estimate \eqref{eq:Stirling_bulk} with $m = 0$, the contour integral in \eqref{eq:Khat} becomes
\begin{align*}
   \Khat_{n}(\xi, \eta) = &-2 \sqrt{\sinh \xi \sinh \eta} \int^{\infty}_{-\infty} \frac{ds}{2\pi i} \oint_{\contC^{(\crossover, 0)}_{\Lower} \cup \contC^{(\crossover, 0)}_{\Upper}} \frac{dt}{2\pi i} \frac{\cosh(\pi s)}{\cosh(\pi t)} \frac{e^{\alpha t^2}}{e^{\alpha s^2}} \frac{tP_{it - \frac{1}{2}}(\cosh\xi) Q_{is - \frac{1}{2}}(\cosh\eta)}{s - t} \\
   & + \bigO(n^{-1}),
\end{align*}
The contour $\contC^{(\crossover, 0)}_{\Lower} \cup \contC^{(\crossover, 0)}_{\Upper}$ can be replaced by $C^{(\crossover, \infty)}_{\Lower} \cup C^{(\crossover, \infty)}_{\Upper}$ without changing the estimate. After the change of variables $s \mapsto -iu$ and $t \mapsto -iv$, this proves part \ref{enu:thm:crossover_2} of Theorem \ref{thm:crossover}.

\subsubsection{Soft edge} \label{subsubsec:crossover_soft}

The proof in this section is similar to that of \cite[Theorem 1.2]{Liu-Wang-Wang18}, and the limiting result obtained is the same.

We let $\tilde{\tau}, \tilde{\sigma}$ be in a compact subset of $\realR$ and take
\begin{align}
  \xi = {}& 2n \alpha + \log(2n) + \tilde{\tau}, & \eta = {}& 2n \alpha + \log(2n) + \tilde{\sigma}, &\tilde{\tau} = {}& \tau - \alpha, &\tilde{\sigma} = {}& \sigma - \alpha.
\end{align}
We deform the contour $\contC$ into $\contC^{(\crossover, n)}_{\Middle}$, and $\Sigma$ into $\Sigma^{(n)}_0 \cup \Sigma^{(n)}_+ \cup \Sigma^{(n)}_-$. We use \eqref{eq:asy_P_crossover_between}--\eqref{eq:asy_P_crossover_reflection} to estimate $\sqrt{\sinh \xi} e^{\frac{M}{n} t^2} P_{it - \frac{1}{2}}(\cosh \xi)$ for $t \in \contC$ and \eqref{eq:asy_Q_crossover} to estimate $\sqrt{\sinh \eta} e^{-\frac{M}{n} s^2} Q_{is - \frac{1}{2}}(\cosh \eta)$ for $s \in \Sigma$. In addition, we estimate $(\frac{1}{2}-is)_{n}(\frac{1}{2}+is)_{n}$ for $s \in \Sigma$ with $\lvert \Re s \rvert \geq (\epsilon/2)n$ and $(\frac{1}{2}-it)_{n}(\frac{1}{2}+it)_{n}$ for $t \in \contC$ with $\lvert \Im t \rvert \leq (1 - \epsilon/2) n$ using \eqref{eq:Stirling_s}, and estimate $(\frac{1}{2}-is)_{n}(\frac{1}{2}+is)_{n}$ for $s \in \Sigma$ with $\lvert \Re s \rvert > (\epsilon/2)n$ and $(\frac{1}{2}-it)_{n}(\frac{1}{2}+it)_{n}$ for $t \in \contC$ with $\lvert \Im t \rvert > (1 - \epsilon/2) n$ using \eqref{eq:Stirling_t}. Then we find that as $s \in \Sigma$ moves away from $-in$, $p_n(s; \xi)$ vanishes exponentially fast, and that as $t \in \contC \cap \compC_-$ moves away from $-in$, $q_n(t; \eta)$ vanishes exponentially fast. For $t \in \contC \cap \compC_+$, we have the same property due to the symmetry $q_n(-t; \eta) = q_n(t; \eta)$. Hence, the major contribution of the integral in \eqref{eq:Khat} is from two regions: the region where both $s$ and $t$ are around $-in$, and the region where $s$ is around $-in$ and $t$ is around $in$. We note that due to the symmetry of $p_n(t; \xi)$, the factor $q_n(s; \eta) p_n(t; \xi)$ has equal contribution in the two regions, but the factor $(s - t)^{-1}$ in the integrand of \eqref{eq:Khat} is much larger in the former region than in the latter. Hence, we only need to consider the integral in the region where both $s$ and $t$ around $-in$ for the leading order estimates of $\Khat_{n}(\xi, \eta)$.
\medbreak

For $s$ around $-in$, we have the estimates
\begin{equation} \label{eq:crossover_Q_soft}
  \begin{split}
    e^{-\frac{M}{n} s^2} Q_{is - \frac{1}{2}}(y) = {}& \frac{\sqrt{\pi}}{\sqrt{2is \sinh \eta}} e^{-\alpha s^2-is\eta} (1 + \bigO(s^{-1}) + \bigO(e^{-\eta} s^{-\frac{1}{2}})) \\
    = {}& e^{-n^2 \alpha} \frac{\sqrt{\pi}}{\sqrt{2is \sinh \eta}} e^{-\alpha (s + in)^2 - i(\tilde{\sigma} + \log(2n))s} (1 + \bigO(s^{-1}) + \bigO(e^{-\eta} s^{-\frac{1}{2}})),
  \end{split}
\end{equation}
and
\begin{equation} \label{eq:crossover_P_soft}
  \begin{split}
    e^{\frac{M}{n} t^2} P_{it - \frac{1}{2}}(x) = {}& \frac{1}{\sqrt{2\pi it \sinh \xi}} e^{\alpha t^2 + it\xi} (1 + \bigO(t^{-1}) + \bigO(e^{-\xi} t^{-\frac{1}{2}})) \\
    = {}& e^{n^2 \alpha} \frac{1}{\sqrt{2\pi it \sinh \xi}} e^{\alpha (t + in)^2 + i(\tilde{\tau} + \log(2n))t} (1 + \bigO(t^{-1}) + \bigO(e^{-\xi} t^{-\frac{1}{2}})).
  \end{split}
\end{equation}
After performing the change of variables
\begin{align*}
  s \mapsto -in + u,\qquad t \mapsto -in + v.
\end{align*}
in \eqref{eq:Khat}, and using \eqref{eq:crossover_Q_soft}, \eqref{eq:crossover_P_soft} and \eqref{eq:Stirling_t}, we have
\begin{equation} \label{eq:crossover_softedge}
  \begin{split}
   \Khat_{n}(\xi, \eta) = \frac{e^{ \tilde{\tau} n}}{e^{\tilde{\sigma} n}} i \int^{\infty}_{-\infty} \frac{du}{2\pi i} \oint_{\contC^{(\crossover, \infty)}_{\Upper}} \frac{dv}{2\pi i} \frac{\cosh(\pi u) \Gamma(\frac{1}{2} -iu)}{\cosh(\pi v) \Gamma(\frac{1}{2} -iv)} \frac{e^{\alpha v^2 + i \tilde{\tau} v}}{e^{\alpha u^2 + i \tilde{\sigma} u}} \frac{1}{u - v} (1 + \bigO(n^{-1})),
  \end{split}
\end{equation}
where we deform the shape of the contours in the region $\lvert u \rvert$ or $\lvert v \rvert$ is large, and it does not affect the leading order of the contour integral. The change of variables $u \to -i(u - \frac{1}{2})$ and $v \to -i(v - \frac{1}{2})$ and the reflection formula of the gamma function, then gives part \ref{enu:thm:crossover_3} of Theorem \ref{thm:crossover} with cocycle $e^{\tilde{\sigma} (n - \frac{1}{2})}/e^{ \tilde{\tau} (n - \frac{1}{2})}$.

\subsection{Crossover I} \label{subsec:ballisitic_regime}

We assume that $M = \alpha > 0$ is fixed. We first prove a technical lemma that will be used in this subsection and Section \ref{subsec:diffusive_regime}. We define
\begin{equation} \label{eq:value_of_r_soft_edge}
  r_{\max} = \frac{1}{2\alpha} \left( \sqrt{(2\alpha + 1)^2 - 1} + \arcosh(2\alpha + 1) \right) =  \sqrt{1 + \alpha^{-1}} + \frac{1}{2\alpha} \log \frac{\sqrt{1 + \alpha^{-1}} + 1}{\sqrt{1 + \alpha^{-1}} - 1}.
\end{equation}
We defined the functions $f$ and $F$ for $z \in \compC \setminus \{iy : y \in \realR \}$, such that
\begin{align}
  f(z; r) = {}& 2\alpha z + i2r \alpha + i \log \frac{1 - iz}{1 + iz}
                \begin{cases}
                  -\pi, & \Re z > 0, \\
                  +\pi, & \Re z < 0, 
                \end{cases} \label{eq:defn_f_ballistic} \\
  F(z; r) = {}& \alpha z^2 + i2\alpha r z - (1 - iz)\log(1 - iz) - (1 + iz)\log(1 + iz)
                \begin{cases}
                  - \pi z, & \Re z > 0, \\
                  + \pi z, & \Re z < 0, 
                \end{cases} \label{eq:defn_F(z;r)}
\end{align}
where the $\log$ function takes the principal branch, $z\mapsto f(z; r)$ is analytically extended to $\{ iy : y \in (-\infty, -1) \cup (1,\infty) \} \}$. On $\{ iy: y \in (-1, 1) \}$, we have
\begin{equation}
  \lim_{\epsilon \to 0^+} f(\epsilon + iy) = \lim_{\epsilon \to 0^+} f(-\epsilon + iy) - 2\pi, \quad y \in (-1, 1).
\end{equation}
Similarly,
\begin{equation}
  \lim_{\epsilon \to 0^+} F(\epsilon + iy) = \lim_{\epsilon \to 0^+} F(-\epsilon + iy)
  \begin{cases}
    - 2\pi z, & y \in (-1, 1), \\
    + 2\pi i, & y \in (-\infty, -1), \\
    - 2\pi i, & y \in (1, \infty),
  \end{cases}
\end{equation}
and $z\mapsto e^{n F(z; r)}$ is analytic on $\compC \setminus \{ iy : y \in [-1, 1] \}$. Then we have the following technical result.

\begin{lem} \label{lem:ballistic}
  Fix $\alpha > 0$. Let $r \in (0, r_{\max})$ where $r_{\max}$ is defined in \eqref{eq:value_of_r_soft_edge}.
  \begin{enumerate}
  \item \label{enu:lem:ballistic_1}
    The equation for $z$ in the fourth quadrant of the complex plane,
    \begin{equation} \label{eq:saddle_point_eq}
      f(z; r) = 0, 
    \end{equation}
    has a unique solution $z = z(r)$, where $f(z; r)$ is defined in \eqref{eq:defn_f_ballistic}. Moreover, as $r$ runs from $0$ to $r_{\max}$, $z(r)$ moves along the curve
    \begin{equation} \label{eq:curve_C}
      C_{\alpha} = \{ x - i\sqrt{g(x)} : x \in (0, z(0)) \}, \quad \text{where} \quad g(x) = 2x \frac{\cos(2\alpha x)}{\sin(2\alpha x)} - x^2 + 1,
    \end{equation}
    from the end point $z(0) \in (0, \pi/(2\alpha))$ that is the solution of
    \begin{equation} \label{eq:defn_x_max}
      \tan(\alpha x) = \frac{1}{x},
    \end{equation}
    to the other end point
    \begin{equation}
      z(r_{\max}) = -i(1 + \alpha^{-1})^{\frac{1}{2}},
    \end{equation}
    such that $\Re z(r)$ is decreasing.
  \item \label{enu:lem:ballistic_2}
    For $r \in (0, r_{\max})$, the function for $z$ in the fourth quadrant of the complex plane,
    satisfies
    \begin{align}
      F'(z(r); r) = {}& 0, & \Re F''(z(r); r) > {}& 0,
    \end{align}
    and $\Re F(z; r)$ attains its maximum at $z(r)$ as $z \in C_{\alpha}$. Moreover, $\Re F(z; 0)$ attains its maximum at $z(0)$ as $z \in C_{\alpha}$, and $\Re F(z; 0)$ attains its maximum at $z(0)$ as $z \in C_{\alpha}$; $\Re F(z; r_{\max})$ attains its maximum at $z(r_{\max})$ as $z \in C_{\alpha}$.
  \item \label{enu:lem:ballistic_3}
    For $r \in (0, r_{\max})$, as $z$ runs on the ray $\{ x + i\Im z(r): x \in (0, \infty) \}$ from $i\Im z(r)$ to $\infty$, $\Re F(z; r)$ decreases first, attains its minimum at $z(r)$, and then increases. Moreover, as $z$ runs over the positive real line, $\Re F(z; 0)$ decreases first, attains its minimum at $z(r)$, and then increases; as $z$ on the ray $\{ x + i\Im z(r_{\max}): x \in (0, \infty) \}$ from $i\Im z(r_{\max})$ to $\infty$, $\Re F(z; r_{\max})$, attains its minimum at $z(r_{\max}) = i\Im z(r_{\max})$, and then increases.
  \end{enumerate}
\end{lem}

\begin{proof}
  First we prove part \ref{enu:lem:ballistic_1}. Let $z = x + iy$ be a possible solution of \eqref{eq:saddle_point_eq} with $r$ unspecified, where $x  > 0$ and $y < 0$, and let $w = (1 - iz)/(1 + iz)$ such that $\arg w \in [-\pi, 0]$. Let $s = \ln \lvert w \rvert$ and $t = \arg w$, and we have the relation between $s, t$ and $x, y$ that
  \begin{align} \label{eq:s,t_in_x,y}
      e^{s}\bigl[(1-y) \cos t - x\sin t \bigr] &= 1+y, &
      e^{s}\bigl[(1-y) \sin t + x\cos t \bigr] &= -x.
  \end{align}
  Then real and imaginary parts of \eqref{eq:saddle_point_eq} become a pair of real-valued equations
  \begin{align} \label{eq:real_and_imag_F'}
    2\alpha x-\pi-t = {}& 0, & 2\alpha y+2r \alpha + s = {}& 0.
  \end{align}
  Since $t \in (-\pi, 0)$, we have that $x \in (0, \pi/(2\alpha))$. Expressing $s, t$ into $x, y$ by \eqref{eq:s,t_in_x,y}, we have that \eqref{eq:real_and_imag_F'} becomes
  \begin{align}
    e^{-2r \alpha} = {}& \frac{(1 + y) e^{2\alpha y}}{-(1 - y) \cos(2\alpha x) + x \sin(2\alpha x)}, \label{eq:r_in_xy_1} \\
    e^{-2r \alpha} = {}& \frac{x e^{2\alpha y}}{(1 - y) \sin(2\alpha x) + x \cos(2\alpha x)}. \label{eq:r_in_xy}
  \end{align}
  Then by identifying the right-hand sides of \eqref{eq:r_in_xy_1} and \eqref{eq:r_in_xy}, we deduce that, with $g(x)$ defined in \eqref{eq:curve_C}
  \begin{equation} \label{eq:relation_without_r}
    y^2 = g(x).
  \end{equation}
  We note that \eqref{eq:relation_without_r} does not depend on $r$, and for any positive $r$, the solution $z = x + iy$ of \eqref{eq:saddle_point_eq} satisfies equation \eqref{eq:relation_without_r} and lies in curve $C_{\alpha}$ defined in \eqref{eq:curve_C}.
  
  By computing the derivative, we find that $g(x)$ is a decreasing function on $(0, \pi/(2\alpha))$, and we can verify that $g(z(0)) = 0$ where $z(0) \in (0, \pi/(2\alpha))$ is defined by \eqref{eq:defn_x_max}. Hence, $g(x) > 0$ for $x \in (0, z(0))$ and for all $r > 0$, possible solutions of \eqref{eq:saddle_point_eq} in the fourth quadrant are in the curve $C_{\alpha}$.

  Suppose $z = x + iy$ in $C_{\alpha}$ is a solution of \eqref{eq:saddle_point_eq} associated with $r > 0$, from \eqref{eq:r_in_xy} and \eqref{eq:relation_without_r}, we have that $r$ is a function of $x$, such that
  \begin{equation} \label{eq:r_in_x}
    r = h(x), \qquad h(x) = -y - \frac{1}{2\alpha} \ln \frac{x}{(1 - y) \sin(2\alpha x) + x \cos(2\alpha x)} \qquad y = -\sqrt{g(x)}.
  \end{equation}
  Below we show that $r$ in \eqref{eq:r_in_x} depends on $x$ monotonically. Suppose $z_0 = x_0 + iy_0 \in C_{\alpha}$ and $r_0 = h(x_0)$. Then by the implicit function theorem, for real-valued $r$ that is in a neighbourhood of $r_0$, \eqref{eq:saddle_point_eq} defines an analytic function $\tilde{z}(r) = \tilde{x}(r) + i \tilde{y}(r)$ such that $f(\tilde{z}(r); r) = 0$ and $\tilde{z}(r_0) = z_0$, with
  \begin{equation}
    \frac{d}{dr} \tilde{z}(r) = -i \frac{\alpha(1 + \tilde{z}(r)^2)}{\alpha(1 + \tilde{z}(r)^2) + 1} \neq 0,
  \end{equation}
  since $\tilde{z}(t)$ is in the fourth quadrant. Moreover,
  \begin{equation}
    \frac{d}{dr} \tilde{x}(r) = \Im \frac{\alpha(1 + \tilde{z}(r)^2)}{\alpha(1 + \tilde{z}(r)^2) + 1} = -\Im \frac{1}{\alpha} \frac{1}{1 + \alpha^{-1} + \tilde{z}(t)^2} < 0,
  \end{equation}
  we have that $\tilde{x}(r)$ is a strictly decreasing function of $r$, and $r$ is a strictly decreasing function of $\tilde{x}$. Since all solutions of \eqref{eq:saddle_point_eq} are on $C_{\alpha}$, $r = h(x)$ is the inverse function of $\tilde{x}(r)$ globally, and then it is a strictly decreasing for $x \in (0, z(0))$. Moreover, as $x \to 0^+$, $h(x) \to r_{\max}$, and as $x \to (z(0))_-$, $h(x) \to 0$. We conclude that for $r \in (0, r_{\max})$, \eqref{eq:saddle_point_eq} has a unique solution $z(r)$ on $C_{\alpha}$, and complete the proof of part \ref{enu:lem:ballistic_1}.

  Next we prove part \ref{enu:lem:ballistic_2}. We note that $F'(z; r) = f(z; r)$, so at $z(r)$, $F'(z(r); r) = f(z(r); r) = 0$, and
  \begin{equation} \label{eq:formula_F''}
    F''(z(r); r) = f'(z(r); r) = 2\alpha \frac{1 + \alpha^{-1} + z(r)^2}{1 + z(r)^2} \neq 0.
  \end{equation}
  Moreover, $\Re F''(z(r); r) > 0$ for all $r \in (0, r_{\max})$. To see it, we write $z(r) = x + iy$, and use \eqref{eq:relation_without_r} to derive that
  \begin{equation} \label{eq:2nd_derivative_F}
    \begin{split}
      \Re F''(z(r); r) = {}& 2\alpha + \frac{2(1 + x^2 - y^2)}{(1 + x^2 - y^2)^2 + (2xy)^2} \\
      = {}& 2\alpha \left( 1 - \frac{\sin(2\alpha x) \cos(2\alpha x)}{2\alpha x} \right) + \sin(2\alpha x)^2 \\
      > {}& 0.
    \end{split}
  \end{equation}
  Below we show that for a fixed $r_0 \in (0, r_{\max})$, the value of $g(r) = \Re F(z(r); r_0)$ attains its maximum at $r = r_0$ for $r \in (0, r_{\max})$. We first have that
  \begin{equation}
    g'(r) = \Re (F'(z(r); r_0) z'(r)).
  \end{equation}
  Noting that
  \begin{equation}
    F'(z(r); r_0) - F'(z(r); r) = i2\alpha(r_0 - r),
  \end{equation}
  and $F'(z(r); r) = 0$, we have $F'(z(r); r_0) = i2\alpha(r_0 - r)$. On the other hand,
  \begin{equation}
    z'(r) = -i \frac{\alpha(1 + z(r)^2)}{\alpha(1 + z(r)^2) + 1} = \frac{-i2\alpha}{F''(z(r); r)}.
  \end{equation}
  Hence,
  \begin{equation}
    g'(r) = 4\alpha^2 (r_0 - r) \Re \frac{1}{F''(z(r); r)} = 4\alpha^2 (r_0 - r) \frac{\Re F''(z(r); r)}{\lvert F''(z(r); r) \rvert^2}
    \begin{cases}
      > 0, & r \in (0, r_0), \\
      < 0, & r \in (r_0, r_{\max}).
    \end{cases}
  \end{equation}
  Hence we prove $g(r)$ attains its unique global maximum at $r = r_0$ and finish the proof of part \ref{enu:lem:ballistic_2}.

  Finally we prove part \ref{enu:lem:ballistic_3}. We denote $y(r) = \Im z(r)$ and
  \begin{equation}
    \phi(x) = \Re F(x + iy(r); r), \quad x > 0,
  \end{equation}
  and have that
  \begin{equation}
    \phi'(x) = \Re F'(x + iy(r); r) = 2\alpha x - \pi - \arg w, \quad w = \frac{1 - i(x + iy(r))}{1 + i(x + iy(r))}.
  \end{equation}
  We have that
  \begin{equation}
    \arg w = -\arccot \frac{1 - x^2 - y(r)^2}{2x} \in (-\pi, 0).
  \end{equation}
  Hence, $\phi'(x) = 0$ is equivalent to 
  \begin{equation}
    y(r) = -\sqrt{1 - x^2 + 2x \cot(2\alpha x)}, \quad 0 < x < \frac{\pi}{2\alpha}.
  \end{equation}
  This implies that $\phi'(x) = 0$ if and only if $z = x + iy(r)$ is on the intersection of $C$ and the horizontal ray $\{ x + iy(r) : x > 0 \}$, and it can only be $z(r)$. Equivalently, we have that on the domain $(0,\infty)$ $\phi'(x)$ only changes sign once at $x = \Re z(r)$. Using the second derivative inequality \eqref{eq:2nd_derivative_F} of $\Re F''(z(r); r)$, we find that $\phi'(x)$ is negative on $(0, \Re z(r))$ and positive on $(\Re z(r),\infty)$. We prove the global minimum property.
\end{proof}

For $r \in (0, r_{\max})$, let $z(r) = x(r) + iy(r)$ be defined as in Lemma \ref{lem:ballistic}. Denote $z_+ = n z(r) = n(x(r) + iy(r))$ and $z_- = n(-x(r) + iy(r))$. We introduce the contours
\begin{align*}
    \contC^{(\ballistic)}_{\Lower, \Right} = \{ nz : z \in C_{\alpha} \},\qquad
    \contC^{(\ballistic)}_{\Upper} = \{ x + iy \in \compC_- : x - iy \in \contC^{(\ballistic)}_{\Lower} \},
\end{align*}
and define $\contC^{(\ballistic)}_{\Lower, \Left}$ to be the reflection of $\contC^{(\ballistic)}_{\Lower, \Right}$ about the imaginary axis, $\contC^{(\ballistic)}_{\Lower} = \contC^{(\ballistic)}_{\Lower, \Left} \cup \contC^{(\ballistic)}_{\Lower, \Right}$ and $\contC^{(\ballistic)}_{\Outer} = \contC^{(\ballistic)}_{\Lower} \cup \contC^{(\ballistic)}_{\Upper}$. We note that as $r$ moves from $0$ to $r_{\max}$, $z_+$ moves along the curve $\contC^{(\ballistic)}_{\Lower, \Right}$ and $z_-$ moves along the curve $\contC^{(\ballistic)}_{\Lower, \Left}$. Let $\epsilon > 0$ be sufficiently small. For any $r \in (0, r_{\max})$, we let
\begin{align}
  z^{(\epsilon, \pm)}_+ = {}& z^{(\epsilon, \pm)}_+(r) = \text{the point on $\contC^{(\ballistic)}_{\Lower, \Right}$ whose imaginary part is $(\Im z_{\pm}) \pm \epsilon$}, \\
  z^{(\epsilon, \pm)}_- = {}& z^{(\epsilon, \pm)}_-(r) = \text{the point on $\contC^{(\ballistic)}_{\Lower, \Left}$ whose imaginary part is $(\Im z_{\pm}) \pm \epsilon$}.
\end{align}
We introduce the contours
\begin{align} 
\contC^{(\ballistic, r, \epsilon)}_{\Outer} = {}& \contC^{(\ballistic, r, \epsilon)}_{\Upper} \cup \contC^{(\ballistic, r, \epsilon)}_{\Lower}, \label{eq:CI-cont-1} \\
  \contC^{(\ballistic, r, \epsilon)}_{\Upper} = {}& \text{the part of $\contC^{(\ballistic)}_{\Outer}$ above $z^{(\epsilon, +)}_{\pm}$}, \\
  \contC^{(\ballistic, r, \epsilon)}_{\Lower} = {}& \text{the part of $\contC^{(\ballistic)}_{\Outer}$ below $z^{(\epsilon, -)}_{\pm}$}, \\
  \contC^{(\ballistic, r, \epsilon)}_{\Inner} = {}& \text{the line segments from $z^{(\epsilon, +)}_-$ to $z^{(\epsilon, +)}_+$ and from $z^{(\epsilon, -)}_+$ to $z^{(\epsilon, -)}_-$}. \label{eq:CI-cont-4}
\end{align}

\subsubsection{Bulk} \label{subsubsec:ballistic_bulk}
Let $r \in \left( 0, r_{\max} \right)$ and take
\begin{align}
  \zeta = {}& 2r \alpha, &\xi = {}& \zeta + \tilde{\tau} n^{-1}, &\eta = {}& \zeta + \tilde{\sigma} n^{-1}, &\tilde{\tau} = {}& \frac{\pi}{\Re z(r)} \tau, &\tilde{\sigma} = {}& \frac{\pi}{\Re z(r)} \sigma.
\end{align}
We deform the contour $\contC$ into $\contC^{(\ballistic, r, \epsilon)}_{\Outer} \cup \contC^{(\ballistic, r, \epsilon)}_{\Inner}$, see \eqref{eq:CI-cont-1}--\eqref{eq:CI-cont-4}, so that $\contC$ consists of two positively oriented closed contours. We deform the contour $\Sigma$ into the horizontal line through $z_{\pm}$. See Figure \ref{fig:ballistic_bulk} for an overview of the contours.

\begin{figure}[htb]
  \centering
  \includegraphics[angle=-90]{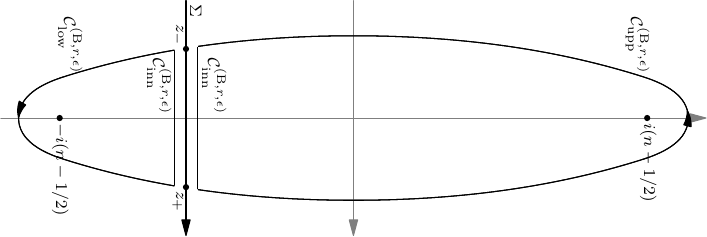}
  \caption{Contours in Section \ref{subsubsec:ballistic_bulk}. Figure rotated $90^{\circ}$ clockwise.}
  \label{fig:ballistic_bulk}
\end{figure}

We write
\begin{equation} \label{eq:K_Outer_Inner}
  \Khat_{n}(\xi, \eta) = I_{\Outer} + I_{\Inner}, 
\end{equation}
where
\begin{equation}
  I_{\star} = - \int_{\Sigma} \frac{ds}{\pi i} \oint_{\contC^{(\ballistic, r, \epsilon)}_{\star}} \frac{dt}{2\pi i} \frac{q_n(\tilde{\sigma}; \eta) p_n(t; \xi)}{s-t}, \qquad \star = \text{$\Outer$ or $\Inner$}.
\end{equation}
By \eqref{eq:Q_nu(x)_hard} and \eqref{eq:K0_asy}, uniformly for $s \in \Sigma$,
\begin{equation} \label{eq:asy_Q_ballistic}
  \begin{split}
    \sqrt{\sinh \eta} e^{-\frac{M}{n} s^2} Q_{is - \frac{1}{2}}(y) = {}& \sqrt{\eta} e^{-\frac{\alpha}{n} s^2} K_0(is \eta) (1 + \bigO(s^{-1})) \\
    = {}& \frac{\sqrt{\pi}}{\sqrt{2is}} e^{-n(\alpha (\frac{s}{n})^2 + i2\alpha r \frac{s}{n})} e^{-i\tilde{\sigma} \frac{s}{n}} (1 + \bigO(s^{-1})).
  \end{split}
\end{equation}
Let $\epsilon' > 0$ be sufficiently small. By \eqref{eq:P_nu(x)_hard} and \eqref{eq:I0_asy}, uniformly for $t \in \contC^{(\ballistic)}$ with $\Im t \leq -\epsilon' n$,
  \begin{equation} \label{eq:asy_P_ballistic}
    \begin{split}
    \sqrt{\sinh \xi} e^{\frac{M}{n} t^2} P_{it - \frac{1}{2}}(x) = {}& \sqrt{\eta} e^{\frac{\alpha}{n} t^2} I_0(it \xi) (1 + \bigO(t^{-1})) \\
    = {}& \frac{1}{\sqrt{2\pi it}} e^{n(\alpha (\frac{t}{n})^2 + i2\alpha r \frac{t}{n})} e^{i\tilde{\tau} \frac{t}{n}} (1 + \bigO(t^{-1})).
    \end{split}
  \end{equation}
We note that the first identity of \eqref{eq:asy_P_ballistic} holds for all $t \in \contC^{(\ballistic)}$ with $\Im t \leq 0$. For $t \in \contC^{(\ballistic)}$ with $\Im t > 0$, the symmetry \eqref{eq:symmetry_P_Q} together with \eqref{eq:P_nu(x)_hard} and \eqref{eq:I0_asy} yields a similar estimate. By \eqref{eq:reflection_Gamma_alt} and Stirling's formula, for $s \in \contC^{(\ballistic)}_{\Outer}$ or $s \in \Sigma$ with $\lvert \Re s \rvert > \epsilon' n$,
\begin{align} \label{eq:gamma_ballistic}
(\tfrac{1}{2}-is)_{n}(\tfrac{1}{2}+is)_{n} = &e^{-2n} \exp \left( \pi s + (n - is) \log (n - is) + (n + is) \log (n + is) \right) \notag \\
  &\times (1 + \bigO(n^{-1})),
\end{align}
and by \eqref{eq:reflection_Gamma} and Stirling's formula, for $s \in \Sigma$ with $\lvert \Re s \rvert \leq \epsilon' n$,
\begin{align} \label{eq:gamma_ballistic_inner}
  (\tfrac{1}{2}-is)_{n}(\tfrac{1}{2}+is)_{n} = &\frac{2\cosh(\pi s)}{e^{2n}} \exp \left( + (n - is) \log (n - is) + (n + is) \log (n + is) \right) \notag \\
  &\times (1 + \bigO(n^{-1})).
\end{align}
We find that for $s \in \Sigma$ with $\lvert \Re s \rvert > \epsilon' n$,
\begin{equation}
  q_n(s; \eta) = e^{-2n(\log n - 1)} \frac{\sqrt{\pi}}{\sqrt{2is}} e^{-n F(\frac{s}{n}; r)} e^{-i\tilde{\sigma} \frac{s}{n}} (1 + \bigO(n^{-1})),
\end{equation}
By part \ref{enu:lem:ballistic_3} of Lemma \ref{lem:ballistic}, we have that on the part of $\Sigma$ with $\Re s > \epsilon' n$, $q_n(s)$ vanishes exponentially fast as $s$ moves away from $z_+$, and the by symmetry, on the part of $\Sigma$ with $\Re s < -\epsilon' n$, $q_n(s)$ vanishes exponentially fast as $s$ moves away from $z_-$. For $t \in \contC^{(\ballistic)}_{\Lower} \cap \{ \Im t \leq -\epsilon' n \}$, we find that
\begin{equation}
  p_n(t; \xi) = e^{2n(\log n - 1)} \frac{\sqrt{t}}{\sqrt{2\pi i}} e^{n F(\frac{t}{n}; r)} e^{i\tilde{\tau} \frac{t}{n}} (1 + \bigO(n^{-1})).
\end{equation}
By part \ref{enu:lem:ballistic_2} of Lemma \ref{lem:ballistic}, we have that on the part of $\contC^{(\ballistic)}_{\Lower, \Right}$ with $\Im t \leq -\epsilon' n$, $p_n(t; \xi)$ vanishes exponentially fast as $t$ moves away from $z_+$, and on the part of $\contC^{(\ballistic)}_{\Lower, \Left}$ with $\Im t \leq -\epsilon' n$, $p_n(t; \xi)$ vanishes exponentially fast as $t$ moves away from $z_-$. We also have the estimate of $q_n(s; \eta)$ for $s \in \Sigma$ with $\lvert \Re s \rvert \leq \epsilon' n$ and the estimate of $p_n(t; \xi)$ for $t \in \contC^{(\ballistic)}_{\Lower}$ with $\Im t \geq -\epsilon' n$. We also find that the vanishing property remains on the extended domain $s \in \Sigma$ and $t \in \contC^{(\ballistic)}_{\Lower}$. Moreover, the asymptotic results and estimates of $p_n(t; \xi)$ can be extended to the whole $\contC^{(\ballistic)}_{\Outer}$ by symmetry. We find that the integral $I_{\Outer}$ concentrates at $s = z_{\pm} + \bigO(n^{\frac{1}{2}})$ and $t = \pm z_{\pm} + \bigO(n^{\frac{1}{2}})$. Below we estimate the part of $I_{\Outer}$ as $s$ and $t$ are concentrated there. 
\medbreak

We denote
\begin{align} \label{eq:defn_Sigma_loc_and_C_loc}
  \Sigma^{z_{\pm}}_{\Local} = {}& \Sigma \cap D_{n^{2/3}}(z_{\pm}), & \contC^{\pm z_{\pm}}_{\Local} = {}& \contC^{(\diffusive, \zeta, \epsilon)}_{\Outer} \cap D_{n^{2/3}}(\pm z_{\pm}),
\end{align}
and 
\begin{equation} \label{eq:local_diffusive}
  I^{\star, \dagger}_{\Outer} = \int_{\Sigma^{\star}_{\Local}} \frac{-ds}{\pi i} \oint_{\contC^{\dagger}_{\Local}} \frac{dt}{2\pi i} \frac{q_n(\tilde{\sigma}; \eta) p_n(t; \xi)}{s-t}, \qquad \star = z_{\pm}, \quad \dagger = \pm z_{\pm}.
\end{equation}
After the change of variables
\begin{align}
  s - z_+ \mapsto u,\qquad t - z_+ \mapsto v,
\end{align}
we have, for $s = z_+ + \bigO(n^{\frac{1}{2}})$ and $t = z_+ + \bigO(n^{\frac{1}{2}})$,
\begin{equation}
  q_n(s; \eta) p_n(t; \xi) = \frac{1}{2i} \frac{e^{i \tilde{\tau} n^{-1} z_+}}{e^{i \tilde{\sigma} n^{-1} z_+}} \frac{e^{n^{-1}(F''(z(r); r) v^2/2 + i\tilde{\tau} v)}}{e^{n^{-1}(F''(z(r); r) u^2/2 + i\tilde{\sigma} u)}} (1 + \bigO(n^{-\frac{1}{2}})),
\end{equation}
where $F''(z(r); r)$ is given in \eqref{eq:formula_F''}. Hence, as $\epsilon \to 0^+$,
\begin{equation} \label{eq:estimate_diffusive_outer}
  \begin{split}
    \lim_{\epsilon \to 0^+} I^{z_+, z_+}_{\Outer} = {}& \PV \int_{v + z_+ \in \contC^{\dagger}_{\Local}} \frac{dv}{2\pi i} \int^{n^{2/3}}_{-n^{2/3}} \frac{-du}{\pi i} \frac{1}{2i} \frac{e^{i \tilde{\tau} n^{-1} z_+}}{e^{i \tilde{\sigma} n^{-1} z_+}} \frac{e^{n^{-1}(F''(z(r); r) v^2/2 + i\tilde{\tau} v)}}{e^{n^{-1}(F''(z(r); r) u^2/2 + i\tilde{\sigma} u)}} \frac{1}{u - v} \\
    {}& \times (1 + \bigO(n^{-\frac{1}{2}})) \\
    = {}& \frac{e^{i \tilde{\tau} n^{-1} z_+}}{e^{i \tilde{\sigma} n^{-1} z_+}} \bigO(n^{\frac{1}{2}}) \\
    = {}& \exp ((\eta - \xi) \Im z_{\pm}) \bigO(n^{\frac{1}{2}}).
  \end{split}
\end{equation}
Similarly, 
\begin{equation} \label{eq:estimate_diffusive_outer_more}
  \lim_{\epsilon \to 0^+} I^{z_-, z_-}_{\Outer} = \exp ((\eta - \xi) \Im z_{\pm}) \bigO(n^{\frac{1}{2}}), 
\end{equation}
and more generally,
\begin{equation*}
  \lim_{\epsilon \to 0^+} I^{\star, \dagger}_{\Outer} = \exp ((\eta - \xi) \Im z_{\pm}) \bigO(n^{-\frac{1}{2}}),
\end{equation*}
for $(\star, \dagger) = (z_+, z_-), (z_-, z_+)$, or $\star \in \{ z_+, z_- \}$ and $\dagger \in \{ -z_+, -z_- \}$. Hence,
\begin{equation} \label{eq:ballistic_bulk_outer}
  \lim_{\epsilon \to 0^+} I_{\Outer} = \exp ((\eta - \xi) \Im z_{\pm}) \bigO(n^{\frac{1}{2}}).
\end{equation}
On the other hand, 
\begin{equation} \label{eq:sine_kernel_ballistic}
  \lim_{\epsilon \to 0^+} I_{\Inner} = - \sqrt{\sinh \xi \sinh \eta} \int^{z_+}_{z_-} s Q_{is - \frac{1}{2}}(\cosh \eta) P_{is - \frac{1}{2}}(\cosh \xi) \frac{ds}{\pi i}.
\end{equation}
By \eqref{eq:asy_Q_ballistic} and \eqref{eq:asy_P_ballistic}, we have that, for $s$ between $z_-$ and $z_+$,
\begin{equation}
    \sqrt{\sinh \xi \sinh \eta} s Q_{is - \frac{1}{2}}(\cosh \eta) P_{is - \frac{1}{2}}(\cosh \xi) = \frac{1}{2i} e^{i(\tilde{\tau} - \tilde{\sigma}) \frac{s}{n}} (1 + \bigO(n^{-1})), 
\end{equation}
and thus
\begin{equation} \label{eq:ballistic_bulk_inner}
  \begin{split}
   \lim_{\epsilon \to 0^+} I_{\Inner} = {}& - \int^{z_+}_{z_-} \frac{1}{2i s} e^{i(\tilde{\tau} - \tilde{\sigma}) \frac{s}{n}} (1 + \bigO(n^{-1})) \frac{ds}{\pi i} \\
    = {}& \exp ((\eta - \xi) \Im z_{\pm}) n \left( \frac{\sin(\Re z(r)(\tilde{\tau} - \tilde{\sigma}))}{\pi(\tilde{\tau} - \tilde{\sigma})} + \bigO(n^{-1}) \right),
  \end{split}
\end{equation}

Summing up \eqref{eq:ballistic_bulk_outer} and \eqref{eq:ballistic_bulk_inner}, we obtain part \ref{enu:thm:ballistic_1} of Theorem \ref{thm:ballistic} with cocycle  $\exp ((\xi - \eta) \Im z_{\pm})$.

\begin{rem} \label{rem:comparison}
  The value of $r_{\max}$ is related to the value of $b$, the right-end point of the support of equilibrium measure in the biorthogonal model in \cite[Eq. (1.28)]{Wang23}, by $r_{\max} \alpha = \sqrt{b}$ (after taking $\Mconst = M = \alpha$). The curve $C_{\alpha}$ is related to the curve $\gamma_+$ in \cite[Lem. 2.2]{Wang23} as follows. We denote $u = u(r) = iz(r)$ and find that $u$ satisfies the identity
  \begin{equation}
    \log \frac{u + 1}{u - 1} + 2\alpha u = 2r\alpha,
  \end{equation}
  or equivalently \cite[Eq. (2.2)]{Wang23}
  \begin{equation}
    \mathbf{J}_{2\alpha}(u^2) = r^2 \alpha^2, \quad \text{where} \quad \mathbf{J}_x(s) = \frac{1}{4} \left( \arcosh \frac{s + 1}{s - 1} + x \sqrt{s} \right)^2.
  \end{equation}
  Hence, we derive that $u(r)^2$ is the same as the $\mathbf{I}_+(r^2 \alpha^2)$ with the parameter $2\alpha$ in \cite[Eq. (2.9)]{Wang23}.

\end{rem}

\subsubsection{Hard edge} \label{subsubsec:ballistic_hard}

Recall that $z(0) \in (0, \frac{\pi}{2\alpha})$ is defined by \eqref{eq:defn_x_max}. We take
\begin{align} \label{eq:hard_edge_ballistic}
  \xi = {}& \frac{1}{z(0) n} \tilde{\tau}, & \eta = {}& \frac{1}{z(0) n} \tilde{\sigma}, &\tilde{\tau} = {}&\pi \tau, &\tilde{\sigma} = {}&\pi \sigma.
\end{align}
with $\tilde{\tau},\tilde{\sigma}$ in a compact subset of $[0,\infty)$. We deform the contour $\contC$ into $\contC^{(\ballistic, 0, \epsilon)}_{\Outer} \cup \contC^{(\ballistic, 0, \epsilon)}_{\Inner}$ with $\contC^{(\ballistic, 0, \epsilon)}_{\Outer} = \contC^{(\ballistic, 0, \epsilon)}_{\Upper} \cup \contC^{(\ballistic, 0, \epsilon)}_{\Lower}$ and $\contC^{(\ballistic, 0, \epsilon)}_{\Inner}$ as in \eqref{eq:CI-cont-1}--\eqref{eq:CI-cont-4}. Now, decompose $\contC^{(\ballistic, 0, \epsilon)}_{\Lower} = \contC^{(\ballistic, 0, \epsilon)}_{\Lower, \Right} \cup \contC^{(\ballistic, 0, \epsilon)}_{\Lower, \Left}$ into
\begin{equation}
  \contC^{(\ballistic, 0, \epsilon)}_{\Lower, \Right} = \{ z \in \compC : n^{-1}z \in C_{\alpha} \text{ and } \Im z < -\epsilon \},
\end{equation}
and its reflection $\contC^{(\ballistic, 0, \epsilon)}_{\Lower, \Left}$ about the imaginary axis, and $\contC^{(\ballistic, 0, \epsilon)}_{\Upper}$ is the reflection of $\contC^{(\ballistic, 0, \epsilon)}_{\Lower}$ about the real axis.  We take $\Sigma=\R$. See Figure \ref{fig:ballistic_hard} for an overview of the contours.

\begin{figure}[htb]
  \centering
  \includegraphics[angle=-90]{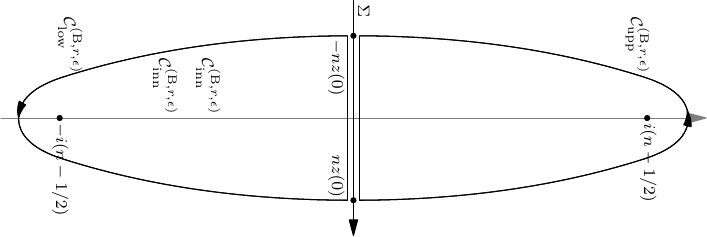}
  \caption{Contours in Section \ref{subsubsec:ballistic_hard}. Figure rotated $90^{\circ}$ clockwise.}
  \label{fig:ballistic_hard}
\end{figure}

Similarly as in \eqref{eq:K_Outer_Inner}, we have
\begin{equation} \label{eq:K_Outer_Inner_hard}
  \Khat_{n}(\xi, \eta) = I_{\Outer} + I_{\Inner}, 
\end{equation}
where
\begin{equation}
  I_{\star} = - \int_{\Sigma} \frac{ds}{\pi i} \oint_{\contC^{(\ballistic, 0, \epsilon)}_{\star}} \frac{dt}{2\pi i} \frac{q_n(s; \eta) p_n(t; \xi)}{s-t}, \qquad \star = \text{$\Outer$ or $\Inner$}.
\end{equation}

For $s \in \Sigma$, we have the estimate \eqref{eq:Q_nu(x)_hard} for $Q_{is - \frac{1}{2}}(\cosh \eta)$ and for $t \in \contC^{(\ballistic, 0, \epsilon)}_{\Outer}$, we have the estimate \eqref{eq:P_nu(x)_hard} for $P_{it - \frac{1}{2}}(\cosh \xi)$. On the other hand, for $t \in \contC^{(\ballistic, 0, \epsilon)}_{\Outer, \Lower}$ or $s \in \Sigma \cap \{ \Re s \geq \epsilon' n \}$ where $\epsilon' > 0$ is a small constant, we have that \eqref{eq:gamma_ballistic} holds, and for $s \in \Sigma \cap \{ 0 \leq \Re s \leq \epsilon' n \}$, we have that \eqref{eq:gamma_ballistic_inner} holds. Hence, we have that for $s \in \Sigma \cap \{ \Re s \geq \epsilon' n \}$, 
\begin{equation}
  q_n(s; \eta) = e^{-2n(\log n - 1)} \eta^{\frac{1}{2}} K_0 \left( i \frac{\tilde{\sigma} s}{z(0) n} \right) e^{-n F(\frac{s}{n}; 0)} (1 + \bigO(n^{-1})),
\end{equation}
We have that on this part of $\Sigma$, $q_n(s; \eta)$ vanishes exponentially fast as $s$ moves away from $n z(0)$. We also have the estimate of $q_n(s; \eta)$ for $s \in \Sigma \cap \{ 0 \leq \Re s \leq \epsilon' n \}$ which we do not write down explicitly here. This estimate shows that the value of $\lvert q_n(s; \eta) \rvert$ is much smaller than $\lvert q_n(n z(0); \eta) \rvert$, and we find that the vanishing property remains if the contours are extended to $s \in \Sigma \cap \{ \Re s \geq 0 \}$. For $t \in \contC^{(\ballistic, 0, \epsilon)}_{\Lower, \Right}$,
\begin{equation}
  p_n(t; \xi) = e^{2n(\log n - 1)} t \xi^{\frac{1}{2}} I_0 \left( i \frac{\tilde{\tau} t}{z(0) n} \right) e^{n F(\frac{t}{n}; 0)} (1 + \bigO(n^{-1})),
\end{equation}
and we have that on $\contC^{(\ballistic)}_{\Lower, \Right}$, $p_n(t; \xi)$ vanishes exponentially fast as $t$ moves away from $n z(0)$. The asymptotic results and estimates of $p_n(t; \xi)$ and $p_n(t; \xi)$ can be extended to the whole $\Sigma$ and $\contC^{(\ballistic)}_{\Outer}$ by symmetry, and we find that the integral $I_{\Outer}$ concentrates at $s = \pm n z(0) + \bigO(n^{1/2})$ and $t = \pm n z(0) + \bigO(n^{1/2})$. By computation similar to \eqref{eq:estimate_diffusive_outer} and \eqref{eq:estimate_diffusive_outer_more}, we have that
\begin{equation} \label{eq:ballistic_outer_est}
  \lim_{\epsilon \to 0^+} I_{\Outer} = \bigO(n^{\frac{1}{2}}).
\end{equation}

Now we consider $I_{\Inner}$. We have, with the line segment connecting $-n z(0)$ and $n z(0)$ being denoted as $\Sigma_{\Inner}$, and after the change of variables $s = z(0) n u$, 
\begin{equation} \label{eq:I_inner_ballistic_hard}
  \begin{split}
    \lim_{\epsilon \to 0^+} I_{\Inner} = {}& - \sqrt{\sinh \xi \sinh \eta} \int_{\Sigma_{\Inner}} \frac{ds}{\pi i} s Q_{is - \frac{1}{2}}(\cosh \eta) P_{is - \frac{1}{2}}(\cosh \xi) \\
    = {}& - \sqrt{\xi \eta} \int^{-n z(0)}_{n z(0)} \frac{ds}{\pi i} s I_0(is\xi) K_0(is\eta) (1 + \bigO(n^{-1})).
  \end{split}
\end{equation}
The second equality follows from \eqref{eq:Q_nu(x)_hard} and \eqref{eq:P_nu(x)_hard}. Since we have the freedom to deform the integral contour from $-n z(0)$ to $n z(0)$, we assume that the contour is away from $0$, and the $\bigO(s^{-1})$ error term is uniformly $\bigO(n^{-1})$.
% Noting that \cite[10.27.26 and 10.27.28]{Boisvert-Clark-Lozier-Olver10} if $\arg z \in (-\pi, 0)$, then for $\nu = 0, 1$,
% \begin{align} \label{eq:I_K_in_J_H}
%   I_{\nu}(iz) =  e^{\frac{\nu \pi i}{2}} J_{\nu}(z), \qquad K_0(iz) = -\frac{\pi i}{2} e^{-\frac{\nu \pi i}{2}} H^{(2)}_0(z), 
% \end{align}
% and \cite[10.6.3]{Boisvert-Clark-Lozier-Olver10}
% \begin{align}
%   J'_0(z) = J_1(z), \qquad (H^{(2)}_0)'(z) = H^{(2)}_1(z),
% \end{align}
Then we have that
\begin{equation} \label{eq:Bessel_integral}
  \begin{split}
    -\int^{n z(0)}_{-n z(0)} \frac{ds}{\pi i}s I_0(is \xi) K_0(is \eta) = {}& (n z(0))^2 \int^1_{-1} \frac{ds}{\pi i}s I_0(is \tilde{\tau}) K_0(is \tilde{\sigma}) \\
    = {}& (n z(0))^2 \int^1_0 \frac{ds}{\pi i}s I_0(is \tilde{\tau}) (K_0(is \tilde{\sigma}) - K_0(-is \tilde{\sigma})) \\
    = {}& (n z(0))^2 \int^1_0 s J_0(s \tilde{\tau}) J_0(s \tilde{\sigma}) ds = 2 (n z(0))^2 K_{\Bes, 0}(\tilde{\tau}^2, \tilde{\sigma}^2),
  \end{split}
\end{equation}
where we use the identities $I_0(is \tilde{\tau}) = J_0(s \tilde{\tau})$ and $K_0(is \tilde{\sigma}) - K_0(-is \tilde{\sigma}) = -\pi i J_0(s \tilde{\sigma})$ from \cite[10.27.6 and 10.27.9]{Boisvert-Clark-Lozier-Olver10}, and the integral representation of $K_{\Bes, 0}$ in \eqref{eq:defn_K_Bes}.

Summing up \eqref{eq:I_inner_ballistic_hard} and \eqref{eq:ballistic_outer_est}, we obtain part \ref{enu:thm:ballistic_2} of Theorem \ref{thm:ballistic}.

\subsubsection{Soft edge} \label{subsubsec:ballistic_soft}

The proof in this subsubsection is similar to that of \cite[Thm. 1.3, part 2]{Liu-Wang-Wang18}. The limiting result obtained is the same, and the deformation of the contours around the saddle point is different.

We consider the function $F(z; r)$ in \eqref{eq:defn_F(z;r)}, take $r = r_{\max}$ as in \eqref{eq:value_of_r_soft_edge} and set $z_*=z(r_{\max}) = -i(1 + \alpha^{-1})^{\frac{1}{2}}$. 
We take
\begin{align}
  \xi = {}& 2r_{\max} \alpha + c^{-1}_{\alpha} \tau n^{-\frac{2}{3}}, & \eta = {}& 2r_{\max} \alpha + c^{-1}_{\alpha} \sigma n^{-\frac{2}{3}}, & c_{\alpha} = {}& 2^{-\frac{1}{3}} \alpha^{-\frac{1}{2}} (\alpha + 1)^{-\frac{1}{6}}.
\end{align}
Recall that
\begin{align} \label{eq:Taylor_F_3}
  F'(z_*; r_{\max}) = {}& F''(z_*; r_{\max}) = 0, & F'''(z_*; r_{\max}) = {}& 4i \alpha^2 \sqrt{1 + \alpha^{-1}},
\end{align}
and
\begin{equation} \label{eq:Taylor_F_4}
  F^{(4)}(z_*; r_{\max}) = 2 \left( \frac{1}{(\sqrt{1 + \alpha^{-1}} - 1)^3} - \frac{1}{(\sqrt{1 + \alpha^{-1}} + 1)^3} \right) > 0.
\end{equation}
Furthermore, $z_\ast$ is the lowest endpoint of the contour $C_{\alpha}$ defined in \eqref{eq:curve_C}. Hence, for sufficiently small $\delta > 0$, the circle $\{ z \in \compC : \lvert z -z_\ast \rvert = \delta \}$ intersects with $C_{\alpha}$ at a unique point, which we denote by $\hat{z}(\delta)$. We also denote $\theta(\delta) = \arg(\hat{z}(\delta) -z_\ast) \in (0, \frac{\pi}{6})$, which has  $\theta(\delta) = (4\alpha + 3)/(6\sqrt{1 + \alpha^{-1}}) \delta + \bigO(\delta^2)$ as $\delta\to0$.

From similar arguments as in the proof of part \ref{enu:lem:ballistic_2} of Lemma \ref{lem:ballistic}, we can conclude that $F(z; r_{\max})$ attains its unique maximum at $z_*$ on $C_{\alpha}$. Moreover, using arguments from the proof of part \ref{enu:lem:ballistic_3} of Lemma \ref{lem:ballistic}, we have that $F(z; r_{\max})$ attains its minimum at $z_*$ on the ray $\{ x + i\Im z_*: x \in (0,\infty) \}$. On the other hand, we use the Taylor expansion of $F(z; r_{\max})$ at $z = z_*$ up to the fourth order, see \eqref{eq:Taylor_F_3} and \eqref{eq:Taylor_F_4}, to approximate $F(z; r_{\max})$ on $\lvert z - z_* \rvert \leq \delta$. If $\delta$ is sufficiently small, by direct calculation, we have that $F(z; r_{\max})$ attains its maximum at $z_*$ on the union of the line segment connecting $z_*$ and $z_* + \delta e^{\frac{\pi i}{6}}$ and the arc $\{ z_* + \delta e^{\theta i} : \theta(\delta) \leq \theta \leq \frac{\pi}{6} \}$, and $F(z; r_{\max})$ attains its minimum at $z_*$ on the union of the line segment from $z_*$ to $z_* + \delta e^{-\frac{\pi i}{6}}$ and the arc $\{ z_* + \delta e^{\theta i} : -\frac{\pi}{6} \leq \theta \leq 0 \}$. 
\medbreak

We deform the contour $\contC$ into $\contC^{(\ballistic, \Ai)}_{\Lower} \cup \contC^{(\ballistic, \Ai)}_{\Upper}$. Both parts consist of two components: $\contC^{(\ballistic, \Ai)}_{\Lower} =\contC^{(\ballistic, \Ai)}_{\Lower, \Right} \cup \contC^{(\ballistic, \Ai)}_{\Lower, \Left}$ and $\contC^{(\ballistic, \Ai)}_{\Upper} = \contC^{(\ballistic, \Ai)}_{\Upper, \Right} \cup \contC^{(\ballistic, \Ai)}_{\Upper, \Left}$. The components are given by 
\begin{align*}
  \contC^{(\ballistic, \Ai)}_{\Lower, \Right} &= \contC^{(\ballistic, \Ai)}_{\Local, \Right} \cup \{ n z_* + n \delta e^{\theta i} : \theta(\delta) \leq \theta \leq \tfrac{\pi}{6} \} \cup \{ nz : z \in C_{\alpha},\Im z \geq n \Im \hat{z}(\delta) \}, \\
  \contC^{(\ballistic, \Ai)}_{\Local, \Right} &= \text{union of line segments connecting $n z_* + c_{\alpha} n^{\frac{2}{3}} \tfrac{i}{2}$, $n z_* + c_{\alpha} n^{\frac{2}{3}} e^{\frac{\pi i}{6}}$ and $n z_* + n \delta e^{\frac{\pi i}{6}}$}, \\
  \contC^{(\ballistic, \Ai)}_{\Lower, \Left} &= \text{the reflection of $\contC^{(\ballistic, \Ai)}_{\Lower, \Right}$ about $i\R$}, \\
  \contC^{(\ballistic, \Ai)}_{\Upper, \Right} &= \text{the reflection of $\contC^{(\ballistic, \Ai)}_{\Lower, \Right}$ about $\R$}, \\
  \contC^{(\ballistic, \Ai)}_{\Upper, \Left} &= \text{the reflection of $\contC^{(\ballistic, \Ai)}_{\Lower, \Left}$ about $\R$}, 
\end{align*}
all oriented positively. We deform the contour $\Sigma$ into $\Sigma^{(\ballistic, \Ai)}_{\Right} \cup \Sigma^{(\ballistic, \Ai)}_{\Left}$, where
\begin{align*}
  \Sigma^{(\ballistic, \Ai)}_{\Right} &= \Sigma^{(\ballistic, \Ai)}_{\Local, \Right} \cup \{ n z_* + n \delta e^{\theta i} : -\tfrac{\pi}{6} \leq \theta \leq 0 \} \cup (n z_* + n \delta,\infty), \\
  \Sigma^{(\ballistic, \Ai)}_{\Local, \Right} &= \text{union of line segments connecting $n z_* - c_{\alpha} n^{\frac{2}{3}} \tfrac{i}{2}$, $n z_* + c_{\alpha} n^{\frac{2}{3}} e^{-\frac{\pi i}{6}}$ and  $n z_* + n \delta e^{-\frac{\pi i}{6}}$}, \\
  \Sigma^{(\ballistic, \Ai)}_{\Left} &= \text{the reflection of $\Sigma^{(\ballistic, \Ai)}_{\Right}$ about $i\R$},
\end{align*}
oriented from left to right. See Figure \ref{fig:ballistic_soft} for an overview of the contours locally around $n z_*$.

\begin{figure}
   \begin{minipage}[b]{5cm}
        \includegraphics{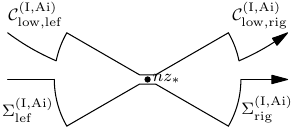}
        \caption{Contours in Section \ref{subsubsec:ballistic_soft} locally around $n z_*$.}
         \label{fig:ballistic_soft}
    \end{minipage}
    \hspace{\stretch{1}}   
    \begin{minipage}[b]{5cm}
        \includegraphics{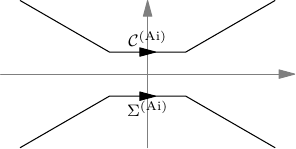}
        \caption{Shape of $\contC^{(\Ai)}$ and $\Sigma^{(\Ai)}$.}
         \label{fig:Airy_contour}
    \end{minipage}
\end{figure}

For $s \in \Sigma$ or $s \in \contC^{(\ballistic, \Ai)}_{\Lower}$, \eqref{eq:gamma_ballistic} holds uniformly; for $s \in \Sigma$, \eqref{eq:asy_Q_ballistic} holds uniformly; for $t \in \contC^{(\ballistic, \Ai)}_{\Lower}$ with $\Im t \leq -\epsilon' n$, and the first identity of \eqref{eq:asy_P_ballistic} holds uniformly for all $t \in \contC^{(\ballistic, \Ai)}_{\Lower}$. Hence, we find that, for $s \in \Sigma$, 
\begin{equation} \label{eq:est_q_ballistic_Ai}
  q_n(s; \eta) = e^{-2n(\log n - 1)} \frac{\sqrt{\pi}}{\sqrt{2is}} e^{-n F(\frac{s}{n}; r_{\max})} e^{-in^{-2/3} c^{-1}_{\alpha} \sigma s} (1 + \bigO(n^{-1})).
\end{equation}
Thus $q_n(s; \eta)$ vanishes exponentially fast as $s$ moves away from $n z_*$ along $s \in \Sigma \cap \{ \Re s \geq 0 \}$. On the other hand, for $t \in \contC^{(\ballistic, \Ai)}_{\Lower}$ with $\Im t \leq -\epsilon' n$,
\begin{equation} \label{eq:est_p_ballistic_Ai}
  p_n(t; \xi) = e^{2n(\log n - 1)} \frac{\sqrt{t}}{\sqrt{2\pi i}} e^{n F(\frac{t}{n}; r_{\max})} e^{i n^{-2/3} c^{-1}_{\alpha} \tau t} (1 + \bigO(n^{-1})).
\end{equation}
Thus, $p_n(t; \xi)$ vanishes exponentially fast as $t$ moves away from $n z_*$ along  $\contC^{(\ballistic, \Ai)}_{\Lower}$ with $\Im t \leq -\epsilon' n$. By the symmetry $t\mapsto -t$, this property also holds along $t \in \contC^{(\ballistic, \Ai)}_{\Lower} \cap \{ \Im t \leq 0 \}$. We can treat the other parts of the contour similarly, and we find that the main contribution to the integral \eqref{eq:Khat}, with $\Sigma$ and $\mathcal{C}$ deformed as above, occurs at $s = n z_* + \bigO(n^{2/3})$ and $t = \pm n z_* + \bigO(n^{2/3})$.
\medbreak

We first consider the contribution along $s \in \Sigma^{(\ballistic, \Ai)}_{\Local}$ and $t \in \contC^{(\ballistic, \Ai)}_{\Local}$, where
\begin{align*}
  \star^{(\ballistic, \Ai)}_{\Local} = \star^{(\ballistic, \Ai)}_{\Local, \Right} \cup \star^{(\ballistic, \Ai)}_{\Local, \Left}, \qquad
  \star^{(\ballistic, \Ai)}_{\Local, \Left} = \text{the reflection of } \star^{(\ballistic, \Ai)}_{\Local, \Right} \text{ about } i\R,
\end{align*}
for $\star = \Sigma \text{ or } \contC$. The previous estimates and the change of variables
\begin{align*}
  c^{-1}_{\alpha} n^{-\frac{2}{3}} (s - n z_*) \mapsto u, \qquad c^{-1}_{\alpha} n^{-\frac{2}{3}} (t - n z_*) \mapsto v,
\end{align*}
then give
\begin{subequations} \label{eq:saddle_pt_ballistic_Ai}
  \begin{align}
    & -c^{-1}_{\alpha} n^{-\frac{2}{3}} \frac{e^{in z_* \eta}}{e^{in z_* \xi}} \int_{\Sigma^{(\ballistic, \Ai)}_{\Local}} \frac{ds}{\pi i} \int_{\contC^{(\ballistic, \Ai)}_{\Local}} \frac{dt}{2\pi i} \frac{q_n(s; \eta) p_n(t; \xi)}{s-t} \label{eq:saddle_pt_ballistic_Ai_1} \\
    = {}& c^{-1}_{\alpha} n^{-\frac{2}{3}} i \int_{\Sigma^{(\ballistic, \Ai)}_{\Local}} \frac{ds}{2\pi i} \int_{\contC^{(\ballistic, \Ai)}_{\Local}} \frac{dt}{2\pi i} \frac{e^{n F(\frac{t}{n}; r_{\max}) + in^{-2/3} c^{-1}_{\alpha} \tau(t - n z_*)}}{e^{n F(\frac{s}{n}; r_{\max}) + in^{-2/3} c^{-1}_{\alpha} \sigma(s - n z_*)}} \frac{\sqrt{t/s} + \bigO(n^{-1})}{s-t} \\
    = {}& i \int_{\Sigma^{(\Ai)}}\frac{du}{2\pi i} \int_{\contC^{(\Ai)}} \frac{dv}{2\pi i} \frac{1}{u - v} \frac{e^{\frac{i}{3} v^3 + i\tau v}}{e^{\frac{i}{3} u^3 + i\sigma u}} + \bigO(n^{-\frac{1}{3}}) = K_{\Ai}(\sigma, \tau) + \bigO(n^{-\frac{1}{3}}),
  \end{align}
\end{subequations}
where the contours $\Sigma^{(\Ai)}$ and $\contC^{(\Ai)}$ are the same as in \eqref{eq:defn_K_Ai}, and are further specified as (see Figure \ref{fig:Airy_contour})
\begin{align}
  \Sigma^{(\Ai)} = {}& \text{polygonal contour connecting $e^{-\frac{5\pi i}{6}}\infty$, $e^{-\frac{5\pi i}{6}}$, $e^{-\frac{\pi i}{6}}$ and $e^{-\frac{\pi i}{6}}\infty$}, \label{eq:defn_Sigma_Ai} \\
  \contC^{(\Ai)} = {}& \text{polygonal contour connecting $e^{\frac{5\pi i}{6}}\infty$, $e^{\frac{5\pi i}{6}}$, $e^{\frac{\pi i}{6}}$ and $e^{\frac{\pi i}{6}}\infty$}. \label{eq:defn_contC_Ai}
\end{align}

Next, we consider the contribution along $s \in \Sigma^{(\ballistic, \Ai)}_{\Local}$ and $t \in -\contC^{(\ballistic, \Ai)}_{\Local}$, which is \eqref{eq:saddle_pt_ballistic_Ai_1} after replacing $t$ by $-t$. By symmetry, the latter doesn't change $q_n(s; \eta) p_n(t; \xi)$, while $(s - t)^{-1}$ becomes $(s + t)^{-1}$. However, as $s + t = \bigO(n)$, we have
\begin{equation} \label{eq:saddle_pt_ballistic_Ai_opposite}
  -2^{\frac{1}{3}} \alpha^{\frac{1}{2}} (\alpha + 1)^{\frac{1}{6}} n^{-\frac{2}{3}} \frac{e^{in z_* \eta}}{e^{in z_* \xi}} \int_{\Sigma^{(\ballistic, \Ai)}_{\Local}} \frac{ds}{\pi i} \oint_{\contC^{(\ballistic, \Ai)}_{\Local}} \frac{dt}{2\pi i} \frac{q_n(s; \eta) p_n(t; \xi)}{s + t} = \bigO(n^{-\frac{1}{3}}).
\end{equation}
We therefore obtain part \ref{enu:thm:ballistic_3} of Theorem \ref{thm:ballistic} with cocycle $\exp ((\eta - \xi) i n z_*)$.

\subsection{Metallic regime} \label{subsec:diffusive_regime}

We assume that $M_n=\alpha_n n$ with all $\alpha_n>0$, and $M_n\to\infty$ and $\alpha_n\to 0$ as $n\to\infty$. The analysis in this section is parallel to the one in Section \ref{subsec:ballisitic_regime}.

\subsubsection{Bulk and hard edge} \label{subsubsec:diffusive_bulk}

We define $\zeta_{\max}$ depending on $M_n$ like $r_{\max}$ depending on $\alpha$ in \eqref{eq:value_of_r_soft_edge}, such that
\begin{equation} \label{eq:defn_zeta_max}
  \begin{split}
    \zeta_{\max} = {}& 2 M_n(1 + M^{-1}_n)^{\frac{1}{2}} + \log \frac{(1 + M^{-1}_n)^{\frac{1}{2}} + 1}{(1 + M^{-1}_n)^{\frac{1}{2}} - 1} \\
    = {}& 2M_n + \log M_n + 2\log 2 + 1 + \bigO(M^{-1}_n).
  \end{split}
\end{equation}
For $\zeta \in (0, \zeta_{\max})$, by Lemma \ref{lem:ballistic}, the relevant equation is
\begin{equation} \label{eq:defn_eq_z(zeta)}
  f(z; \zeta) = 0, \qquad f(z; \zeta) = 2\alpha_n z + i\zeta + i\log \frac{n - iz}{n + iz}
  \begin{cases}
    - \pi, & \Re z > 0, \\
    + \pi, & \Re z < 0.
  \end{cases}
\end{equation}
Indeed, if $(z, \alpha, r)$ in \eqref{eq:saddle_point_eq} is substituted by $(n^{-1} z, M_n, \frac{1}{2} M^{-1}_n \zeta)$, \eqref{eq:saddle_point_eq} becomes \eqref{eq:defn_eq_z(zeta)}. By part \ref{enu:lem:ballistic_1} of Lemma \ref{lem:ballistic}, \eqref{eq:defn_eq_z(zeta)} has exactly one solution in the fourth quadrant, denoted by $z_+(\zeta)$, and one solution in the third quadrant, denoted by $z_-(\zeta)$. Under the same substitution, the function $F(z; r)$ in \eqref{eq:defn_F(z;r)} becomes $n^{-1} F(z; \zeta) + 2\log n$ with $F(z; \zeta)$ defined by
\begin{equation} \label{eq:defn_F_diffusive}
  F(z; \zeta) = \alpha_n z^2 + i\zeta z - (n - iz)\log(n - iz) - (n + iz)\log(n + iz)
  \begin{cases}
    - \pi z, & \Im z > 0, \\
    + \pi z, & \Im z < 0,
  \end{cases}
\end{equation}

  If $\zeta \in (0, 2(1 - \epsilon)M_n)$ for some $\epsilon > 0$, then
  \begin{align} \label{eq:est_z_+_diffusive_bulk}
    z_{\pm}(\zeta) = \frac{1}{2\alpha_n} \left( \pm \pi - i\zeta + i \log \frac{2M_n + \zeta}{2M_n - \zeta} + \bigO(M^{-1}_n) \right),
  \end{align}
  and as $\zeta \to 0^+$, $z_{\pm}(\zeta)$ approaches $z_{\pm}(0)$, where $z_+(0) = \alpha^{-1}_n (\frac{\pi}{2} + \bigO(M^{-1}_n)) \in (0,\infty)$, $z_-(0) = -z_+(0)$, and $z_{\pm}(0)$ are the solution of the equation
  \begin{equation}
    z = \frac{1}{2\alpha_n} \pi - \alpha_n \arctan\Big(\frac{z}{n}\Big).
  \end{equation}
  If $\zeta = 2M_n + \log M_n + x$ with $x$ in a compact subset of $(-\infty, 2\log 2 + 1)$, then $z_{\pm}(\zeta) = -in \pm \alpha_n w$ with $w = \bigO(1)$ and $\arg w \in (-\frac{\pi}{2}, \frac{\pi}{2})$, we have that $w$ satisfies
  \begin{equation}
    G(w) = \bigO(M^{-1}_n), \quad G(w) = 2w + i\log w - \frac{\pi}{2} - i\log 2 + xi.
  \end{equation}
  I $x < 2\log 2 + 1$, then the equation $G(w) = 0$ has a unique solution in the right-half complex plane, and as $x \to -\infty$, the real part of the solution approaches $\frac{\pi}{2}$. Moreover, as $\Re \zeta_+(\xi)$ decreases, $\xi$ runs over $(0, \zeta_{\max})$. In particular, if $\zeta \to \zeta_{\max}$ both $z_+(\zeta)$ and $z_-(\zeta)$ approach
  \begin{equation} \label{eq:defn_z(zeta_max)}
    z(\zeta_{\max}) = -in(1 + M^{-1}_n)^{\frac{1}{2}}.
  \end{equation}

We introduce the contour $\contC^{(\diffusive)}_{\Outer} = \contC^{(\diffusive)}_{\Lower} \cup \contC^{(\diffusive)}_{\Upper}$ where $\contC^{(\diffusive)}_{\dagger} = \contC^{(\diffusive)}_{\Upper, \dagger} \cup \contC^{(\diffusive)}_{\Lower, \dagger}$, for $\star = \Lower \text{ or } \Upper$, and
\begin{align}
  \contC^{(\diffusive)}_{\Lower, \Right} = {}& \{ z_+(\zeta) : \zeta \in [0, \zeta_{\max}] \}, & \contC^{(\diffusive)}_{\Upper,\Left} = {}& \{ z : -z \in \contC^{(\diffusive)}_{\Lower, \Right} \}, \\
  \contC^{(\diffusive)}_{\Lower, \Left} = {}& \{ z_-(\zeta) : \zeta \in [0, \zeta_{\max}] \}, & \contC^{(\diffusive)}_{\Upper,\Right} = {}& \{ z : -z \in \contC^{(\diffusive)}_{\Lower, \Left} \}.
\end{align}
We also introduce $\contC^{(\diffusive)}_{\star} =  \contC^{(\diffusive)}_{\Upper, \star} \cup \contC^{(\diffusive)}_{\Lower, \star}$, for $\star = \Left \text{ or } \Right$. 
\medbreak

Let $\epsilon > 0$ be sufficiently small. For any $\zeta \in (0, \zeta_{\max})$, we let
\begin{align}
  z^{(\epsilon, \pm)}_+ = {}& z^{(\epsilon, \pm)}_+(\zeta) = \text{the point on $\contC^{(\diffusive)}_{\Right}$ whose imaginary part is $(\Im z_{\pm}) \pm \epsilon$}, \\
    z^{(\epsilon, \pm)}_- = {}& z^{(\epsilon, \pm)}_-(\zeta) = \text{the point on $\contC^{(\diffusive)}_{\Left}$ whose imaginary part is $(\Im z_{\pm}) \pm \epsilon$}.
\end{align}
We deform the contour $\contC$ into $\contC^{(\diffusive, \zeta, \epsilon)}_{\Outer} \cup \contC^{(\diffusive, \zeta, \epsilon)}_{\Inner}$. Both parts consist of two components: $\contC^{(\diffusive, \zeta, \epsilon)}_{\Outer} = \contC^{(\diffusive, \zeta, \epsilon)}_{\Upper} \cup \contC^{(\diffusive, \zeta, \epsilon)}_{\Lower}$ and $\contC^{(\diffusive, \zeta, \epsilon)}_{\Inner} = \contC^{(\diffusive, \zeta, \epsilon)}_{\Inner, +} \cup \contC^{(\diffusive, \zeta, \epsilon)}_{\Inner, -}$. The components are given by
\begin{align}
  \contC^{(\diffusive, \zeta, \epsilon)}_{\Upper} = {}&  \text{the part of $\contC^{(\diffusive)}_{\Outer}$ above $z^{(\epsilon, +_)}_{\pm}$}, \label{eq:defn_contC_upp} \\
  \contC^{(\diffusive, \zeta, \epsilon)}_{\Lower} = {}& \text{the part of $\contC^{(\diffusive)}_{\Outer}$ below $z^{(\epsilon, -)}_{\pm}$}, \\
  \contC^{(\diffusive, \zeta, \epsilon)}_{\Inner, \pm} = {}& \text{line segments between $z^{(\epsilon, \pm)}_-$ to $z^{(\epsilon, \pm)}_+$}. \label{eq:defn_contC_inn}
\end{align}

We deform the contour $\Sigma$ into  $\{ z : \Im z = \Im z_{\pm}(\zeta) \}$.
\medbreak

In this section, we consider both the hard edge and bulk limits. At the hard edge, we take $\zeta=0$ and consider
\begin{align} \label{eq:hard_edge_metallic}
  \xi = {}& \frac{\pi}{z_+(0)} \tilde{\tau}, & \eta = {}& \frac{\pi}{z_+(0)} \tilde{\sigma}, &\tilde{\tau} = {}& \frac{2}{\pi} \frac{\Re z_+(\zeta)}{\alpha_n} \tau, &\tilde{\sigma} = {}& \frac{2}{\pi} \frac{\Re z_+(\zeta)}{\alpha_n} \sigma,
\end{align}
for $\tilde{\tau}, \tilde{\sigma}$ in a compact subset of $[0,\infty)$. In the bulk, we let $\eps>0$, take $\zeta \in [2\epsilon M_n, 2(1 - \epsilon) M_n]$ and consider
\begin{align} \label{eq:bulk_metallic}
  \xi = {}& \zeta + \frac{\pi}{\Re z_+(\zeta)} \tilde{\tau}, & \eta = {}& \zeta + \frac{\pi}{\Re z_+(\zeta)} \tilde{\sigma}, &\tilde{\tau} = {}& \frac{2}{\pi} \frac{\Re z_+(0)}{\alpha_n} \tau, &\tilde{\sigma} = {}& \frac{2}{\pi} \frac{\Re z_+(0)}{\alpha_n} \sigma. 
\end{align}
for $\tilde{\tau}, \tilde{\sigma}$ in a compact subset of $[0,\infty)$. In both cases, analogously to \eqref{eq:K_Outer_Inner} and \eqref{eq:K_Outer_Inner_hard}, we have
\begin{equation} \label{eq:K=outer+inner}
  \Khat_{n}(\xi, \eta) = I_{\Outer} + I_{\Inner}, 
\end{equation}
where
\begin{equation} \label{eq:defn_I_Outer_Inner}
  I_{\star} = -\int_{\Sigma} \frac{ds}{\pi i} \oint_{\contC^{(\diffusive, \zeta, \epsilon)}_{\star}} \frac{dt}{2\pi i} \frac{q_n(s; \eta) p_n(t; \xi)}{s-t}, \qquad \star = \text{$\Outer$ or $\Inner$}.
\end{equation}
Similarly as in \eqref{eq:I_inner_ballistic_hard}, we find
\begin{equation} \label{eq:sine_kernel_metallic}
  \begin{split}
    \lim_{\epsilon \to 0^+} I_{\Inner} = {}& -\sqrt{\sinh \xi \sinh \eta} \int^{z_+(\zeta)}_{z_-(\zeta)} \frac{ds}{\pi i} s Q_{is - \frac{1}{2}}(\cosh \eta) P_{is - \frac{1}{2}}(\cosh \xi) \\
    = {}& -\sqrt{\xi \eta} \int^{z_+(\zeta)}_{z_-(\zeta)} \frac{ds}{\pi i} I_0(is \xi) K_0(is \eta) (1 + \bigO(s^{-1})).
  \end{split}
\end{equation}

We first estimate $I_{\Inner}$ as $\epsilon \to 0$. If $\xi, \eta$ are at the hard edge \eqref{eq:hard_edge_metallic}, we have the freedom to deform the integral contour from $z_+(0)$ to $z_-(0)$, so that the contour is away from $0$, and the $\bigO(s^{-1})$ error term is uniformly $\bigO(\alpha_n)$. One possible deformation is as follows: We take $\delta'>0$ sufficiently small (e.g. $\delta'=\frac{1}{10}$) and deform the contour $\Sigma$ such that the line segment connecting $z_+(0)$ and $z_-(0)$ is replaced by two line segments, one connecting $z_-(0)$ to $-\delta' n^{\alpha} i$, the other connecting $-\delta' n^{\alpha} i$ to $z_+(0)$. In this way, the estimate \eqref{eq:Q_nu(x)_hard} holds uniformly for $s \in \Sigma$. Then, similarly as in \eqref{eq:I_inner_ballistic_hard} and \eqref{eq:Bessel_integral}, we can conclude that
\begin{equation} \label{eq:limit_diffusive_inner}
  \begin{split}
    \lim_{\epsilon \to 0^+} I_{\Inner} %= {}& \frac{z(0) \sqrt{\tilde{\tau} \tilde{\sigma}}}{2\pi (\tilde{\sigma}^2 - \tilde{\tau}^2)} \left[s(\tilde{\tau} J_1(s\tilde{\tau}) H^{(2)}_0(s\tilde{\sigma}) - \sigma J_0(s\tilde{\tau}) H^{(2)}_1(s\tilde{\sigma})) \right]^{s = 1}_{s = -1} + \bigO(1) \\
    = {}& 2\pi z_+(0) \left( \sqrt{\tilde{\tau} \tilde{\sigma}} K_{\Bes, 0}((\pi \tilde{\tau})^2, (\pi \tilde{\sigma})^2) + \bigO(\alpha_n) \right),
  \end{split}
\end{equation}

If $\xi, \eta$ are in the bulk \eqref{eq:bulk_metallic}, Using % \eqref{eq:I_K_in_J_H} and
the estimates \eqref{eq:I0_asy} and \eqref{eq:K0_asy}, we obtain

\begin{equation}
  \begin{split}
    \lim_{\epsilon \to 0^+} I_{\Inner} = {}& \frac{1}{2\pi i (\xi - \eta)} \left( e^{i(\xi - \eta) z_+(\zeta)} - e^{i(\xi - \eta) z_-(\zeta)} \right) + \bigO(\exp ((\eta - \xi) \Im z_{\pm}(\zeta))) \\
    = {}& \exp ((\eta - \xi) \Im z_{\pm}(\zeta)) \alpha_n \left( \frac{\sin((\tilde{\tau} - \tilde{\sigma}) \pi)}{\pi(\tilde{\tau} - \tilde{\sigma})} + \bigO(\alpha_n) \right).
  \end{split}
\end{equation}

Second, we estimate $I_{\Outer}$ as $\epsilon \to 0^+$. The estimate \eqref{eq:P_nu(x)_hard} can be applied for $t \in \contC^{(\diffusive, \zeta, \epsilon)}_{\Outer}$, and the estimate \eqref{eq:Q_nu(x)_hard} can be applied for $s \in \Sigma$ if $\lvert s \rvert$ is large. This is satisfied if $\zeta > 0$ and $\xi,\eta$ are in the bulk \eqref{eq:bulk_metallic}. If $\zeta = 0$ and $\xi,\eta$ are at the hard edge \eqref{eq:hard_edge_metallic}, we can deform contour $\Sigma$ as we do in the evaluation of \eqref{eq:sine_kernel_metallic}. In this way, the estimate \eqref{eq:Q_nu(x)_hard} holds uniformly for $s \in \Sigma$. Afterwards, we can use the estimates \eqref{eq:I0_asy} and \eqref{eq:K0_asy}. On the other hand, estimates for $(\frac{1}{2}-is)_{n}(\frac{1}{2}+is)_{n}$ for $s \in \Sigma$ and $(\frac{1}{2}-it)_{n}(\frac{1}{2}+it)_{n}$ for $t \in \contC^{(\diffusive, \zeta, \epsilon)}_{\Outer}$ can be obtained by \eqref{eq:Pochhammer_asy} and Stirling's formula. 
\medbreak

Using the previous estimates, we find that, if $\xi, \eta$ are the hard edge  \eqref{eq:hard_edge_metallic}, $q_n(s; \xi)$ vanishes exponentially fast as $s \in \Sigma$ moves away from $z_{\pm}(0)$, and $p_n(t; \xi)$ vanishes exponentially fast as $t \in \contC^{(\diffusive, 0, \epsilon)}_{\Outer}$ moves away from $z_{\pm}(0)$. Hence, the main contribution to the integral occurs at $s = z_{\pm}(0) + \bigO(\alpha_n^{-\frac{1}{2}})$ and $t = z_{\pm}(0) + \bigO(\alpha_n^{-\frac{1}{2}})$. Around these points, with the help of \eqref{eq:Stirling_s}, we obtain
\begin{equation} \label{eq:hard_edge_mettalic}
  \begin{split}
    q_n(s; \eta) p_n(t; \xi)
    = {}& t (I_0(it\xi) + \bigO(\alpha_n)) (K_0(is\eta) + \bigO(\alpha_n)) \sqrt{\xi \eta} \\
    & \times \frac{e^{-\alpha_n s^2} \cosh(\pi s) e^{(n - is)(\log(n - is) - 1) + (n + is)(\log(n + is) - 1)}}{e^{-\alpha_n t^2} \cosh(\pi t) e^{(n - it)(\log(n - it) - 1) + (n + it)(\log(n + it) - 1)}} (1 + \bigO(n^{-1})).
  \end{split}
\end{equation}
The modified Bessel functions remain bounded and
\begin{equation}
  \frac{e^{-\alpha_n s^2} \cosh(\pi s) e^{(n - is)(\log(n - is) - 1) + (n + is)(\log(n + is) - 1)}}{e^{-\alpha_n t^2} \cosh(\pi t) e^{(n - it)(\log(n - it) - 1) + (n + it)(\log(n + it) - 1)}} = \frac{e^{-\alpha_n (s - z_{\pm}(0))^2}}{e^{-\alpha_n (t - z_{\pm}(0))^2}} (1 + \bigO(\alpha_n^{\frac{1}{2}})),
\end{equation}
as $s,t=z_{\pm}(0) + \bigO(\alpha_n^{-\frac{1}{2}})$. It then follows that
\begin{equation} \label{eq:diffusive_outer_curve}
  \lim_{\epsilon \to 0^+} I_{\Outer} = \bigO(\alpha_n^{\frac{1}{2}}).
\end{equation}
In the case where $\xi, \eta$ are in the bulk \eqref{eq:bulk_metallic}, we can apply \eqref{eq:I0_asy} and \eqref{eq:K0_asy}, to obtain that uniformly for $s \in \Sigma$,
\begin{equation} \label{eq:asy_Q_it}
  \sqrt{\sinh \eta} Q_{is - \frac{1}{2}}(\cosh \eta) = \sqrt{\frac{\pi}{2is}} e^{-is\eta} (1  + \bigO(s^{-1})),
\end{equation}
and uniformly for $t \in \contC^{(\diffusive, \delta, \epsilon)}_{\Outer}$ with $\lvert \Im t \rvert > \delta \alpha_n$,
\begin{equation} \label{eq:asy_P_is}
  \sqrt{\sinh \xi} P_{it - \frac{1}{2}}(\cosh \xi) = \frac{1}{\sqrt{2\pi (\pm it)}} e^{\pm it\xi} (1 + \bigO(t^{-1})),
\end{equation}
where $\pm$ becomes $+$ if $t \in \compC_-$ and $-$ if $t \in \compC_+$. For all $t \in \contC^{(\diffusive, \delta, \epsilon)}_{\Outer}$ with $\lvert \Im t \rvert \leq \delta \alpha_n$ and sufficiently large $n$, we have the estimate
\begin{equation} \label{eq:asy_P_diffusive_between}
  \left\lvert \sqrt{\sinh \xi} P_{it - \frac{1}{2}}(\cosh \xi) \right\rvert \leq 
  \sqrt{\frac{2}{\pi |t|}} \big( e^{\Re(it\xi)} + e^{\Re(-it\xi)} \big).
\end{equation}

We can therefore conclude that uniformly for $s \in \Sigma$ and for $t \in \contC^{(\diffusive, \zeta, \epsilon)}_{\Outer}$ with $\Im t < -\delta \alpha_n$,
\begin{align}
  q_n(s; \eta) = {}& \sqrt{\frac{\pi}{2is}} e^{-F(s; \zeta) - 2n} e^{-i2s \tilde{\sigma} \alpha_n} (1 + \bigO(\alpha_n)), \\
  p_n(t; \xi) = {}& \frac{\sqrt{t}}{\sqrt{2\pi i}} e^{F(t; \zeta) + 2n} e^{i2t \tilde{\tau} \alpha_n} (1 + \bigO(\alpha_n)), \label{eq:est_p_n_metallic}
\end{align}
where $F(z; \zeta)$ is defined in \eqref{eq:defn_F_diffusive}. For $t \in \contC^{(\diffusive, \zeta, \epsilon)}_{\Outer}$ with $\Im t > \delta \alpha_n$, the estimate of $p_n(t; \xi)$ is given by \eqref{eq:est_p_n_metallic} and \eqref{eq:symmetry_P_Q}, and for $t \in \contC^{(\diffusive, \zeta, \epsilon)}_{\Outer}$ with $\lvert \Im t \rvert \leq \delta \alpha_n$, we may use \eqref{eq:asy_P_diffusive_between} instead of \eqref{eq:asy_P_is}.
\medbreak

From the previous estimates, we find that $q_n(s; \xi)$ vanishes exponentially fast as $s \in \Sigma$ moves away from $z_{\pm}(\zeta)$, and that $p_n(t; \xi)$ vanishes exponentially fast as $t \in \contC^{(\diffusive, \zeta, \epsilon)}_{\Outer}$ moves away from $\pm z_{\pm}(\zeta)$. For $\epsilon \to 0^+$, the integral $I_{\Outer}$ therefore concentrates at $s = z_{\pm}(\zeta) + \bigO(\alpha_n^{\frac{1}{2}})$ and $t = \pm z_{\pm}(\zeta) + \bigO(\alpha_n^{\frac{1}{2}})$. Together with the exponential decay of the integrand of $I_{\Outer}$ in \eqref{eq:defn_I_Outer_Inner} away from $z_{\pm}$ for $s$ and $\pm z_{\pm}$ for $t$, we have that
\begin{equation} \label{eq:diffusive_outer}
  \lim_{\epsilon \to 0^+} I_{\Outer} = \exp ((\eta - \xi) \Im z_{\pm}) \bigO(\alpha_n^{\frac{1}{2}}).
\end{equation}

Summing up \eqref{eq:diffusive_outer} and \eqref{eq:sine_kernel_metallic}, we obtain part \ref{enu:thm:metallic_1} of Theorem \ref{thm:metallic} with cocycle $\exp ((\xi - \eta) \Im z_{\pm})$. Similarly, summing up \eqref{eq:limit_diffusive_inner} and \eqref{eq:diffusive_outer_curve}, we obtain part \ref{enu:thm:metallic_2} of Theorem \ref{thm:metallic} with cocycle $\exp ((\xi - \eta) \Im z_{\pm})$.

\subsubsection{Soft edge}

Take $\zeta_{\max}$ as in \eqref{eq:defn_zeta_max} and 
\begin{align}
  \xi = {}& \zeta_{\max} + 2^{\frac{1}{3}} \alpha_n^{\frac{2}{3}} \tau, &
  \eta = {}& \zeta_{\max} + 2^{\frac{1}{3}} \alpha_n^{\frac{2}{3}} \sigma,
\end{align}
with $\tau, \sigma$ are in a compact subset of $\realR$. We consider the function  $F(z; \zeta)$ in \eqref{eq:defn_F_diffusive} and take $z_*=z(\zeta_{\max})$ as in \eqref{eq:defn_z(zeta_max)}. Note that $F(z; \zeta)$ is not defined for $z \in \{ iy : y \in \realR \}$, but $F'(z; \zeta)$ and $e^{F(z; \zeta)}$ can be defined $z \in \{ iy : \lvert y \rvert > n \}$ by analytic continuation. We have, 
\begin{align} 
  F'(z_*; \zeta_{\max}) = {}& F''(z_*; \zeta_{\max}) = 0, & F'''(z_*; \zeta_{\max}) = {}& 4i \alpha_n^2 \sqrt{1 + M^{-1}_n},
\end{align}
\begin{equation} 
  F^{(4)}(z_*; \zeta_{\max}) = 2n^{-3} \Bigg( \frac{1}{(\sqrt{1 + M^{-1}_n} - 1)^3} - \frac{1}{(\sqrt{1 + M^{-1}_n} + 1)^3} \Bigg) = 16 \alpha_n^3 (1 + \bigO(M^{-1}_n)).
\end{equation}
We denote
\begin{equation}
  c_{\alpha} = (-i F'''(z_*; \zeta_{\max})/2)^{-\frac{1}{3}} = 2^{-\frac{1}{3}} \alpha_n^{-\frac{2}{3}} (1 + M^{-1}_n)^{-\frac{1}{6}}.
\end{equation}
For sufficiently small $\delta > 0$, the circle $\{ z \in \compC : \lvert z - z_*\rvert = \delta \alpha^{-1}_n \}$ intersects with $\contC_{\Outer}$ at two points; one in the fourth quadrant, denoted by $\hat{z}_+(\delta)$, and one the the third quadrant, denoted by $\hat{z}_-(\delta)$. We also denote $\theta(\delta) = \arg(\hat{z}_+(\delta) - z_*) \in (0, \frac{\pi}{6})$, which has $\theta(\delta) = \frac{2}{3} \delta + \bigO(\delta^2)$ as $\delta\to0$.
\medbreak

We deform the contour $\contC$ into $\contC^{(\diffusive, \Ai)} = \contC^{(\diffusive, \Ai)}_{\Lower} \cup \contC^{(\diffusive, \Ai)}_{\Upper}$. Both parts of consist of two components:  $\contC^{(\diffusive, \Ai)}_{\Lower} =\contC^{(\diffusive, \Ai)}_{\Lower, \Right} \cup \contC^{(\diffusive, \Ai)}_{\Lower, \Left}$ and $\contC^{(\diffusive, \Ai)}_{\Upper} = \contC^{(\diffusive, \Ai)}_{\Upper, \Right} \cup \contC^{(\diffusive, \Ai)}_{\Upper, \Left}$. The components are given by
\begin{align}
  \contC^{(\diffusive, \Ai)}_{\Lower, \Right} &= \contC^{(\diffusive, \Ai)}_{\Local, \Right} \cup \{ z_* + \alpha_n \delta e^{\theta i} : \theta(\delta) \leq \theta \leq \tfrac{\pi}{6} \} \cup \{ z : z \in \contC^{(\diffusive)}_{\Lower, \Right} \text{ and } \Im z \geq n \Im \hat{z}(\delta) \}, \\
   \contC^{(\diffusive, \Ai)}_{\Local, \Right} &= \text{union of line segments connecting $z_* + c_{\alpha} \tfrac{i}{2}$, $z_* + c_{\alpha} e^{\frac{\pi i}{6}}$ and $z_* + \alpha_n \delta e^{\frac{\pi i}{6}}$}, \\
   \contC^{(\diffusive, \Ai)}_{\Lower, \Left} &= \text{the reflection of $\contC^{(\diffusive, \Ai)}_{\Lower, \Right}$ about $i\R$}, \\
   \contC^{(\diffusive, \Ai)}_{\Upper, \star} &= \text{the reflection of $\contC^{(\diffusive, \Ai)}_{\Lower, \star}$ about $\R$},\qquad \text{$\star = \Left$ or $\Right$},
\end{align}
oriented so that the closed contour $\contC$ is oriented positively. We deform the contour $\Sigma$ into $\Sigma^{(\diffusive, \Ai)} = \Sigma^{(\diffusive, \Ai)}_{\Right} \cup \Sigma^{(\diffusive, \Ai)}_{\Left}$, where
\begin{align}
  \Sigma^{(\diffusive, \Ai)}_{\Right} &= \Sigma^{(\diffusive, \Ai)}_{\Local, \Right} \cup \{ z_* + n \delta e^{\theta i} : -\tfrac{\pi}{6} \leq \theta \leq 0 \} \cup (z_* + \alpha_n \delta, \infty), \\
  \Sigma^{(\diffusive, \Ai)}_{\Local, \Right} &= \text{line segments connecting $z_* - c_{\alpha} \tfrac{i}{2}$, $z_* + c_{\alpha} e^{-\frac{\pi i}{6}}$ and  $z_* + \alpha_n \delta e^{-\frac{\pi i}{6}}$}, \\
  \Sigma^{(\diffusive, \Ai)}_{\Left} &= \text{the reflection of $\Sigma^{(\diffusive, \Ai)}_{\Right}$ about $i\R$},
\end{align}
oriented so that $\Sigma$ is oriented from left to right. The shape of the contours around $z_*$ is similar to the shape of the contours around $z_*$ (defined separately in Section \ref{subsubsec:ballistic_soft}) in Section \ref{subsubsec:ballistic_soft} for crossover I regime, and the shape is schematically shown in Figure \ref{fig:ballistic_soft} (with $\ballistic$ replaced by $\diffusive$). 
\medbreak

For $s \in \Sigma$, we have the estimate \eqref{eq:Q_nu(x)_hard} for $Q_{is - \frac{1}{2}}(\cosh \eta)$. Furthermore, we can use \eqref{eq:reflection_Gamma_alt} and Stirling's formula to obtain an estimate for $(\frac{1}{2}-is)_{n}(\frac{1}{2}+is)_{n}$ like \eqref{eq:gamma_ballistic}. We then find that $q_n(s; \eta)$ vanishes exponentially fast as $s \in (z_* + \alpha_n \delta,\infty) \cup (-\infty, z_* - \alpha_n \delta) \subset \Sigma$ moves away from $z_*$. For $t \in \contC$, we have the estimate \eqref{eq:P_nu(x)_hard} for $P_{it - \frac{1}{2}}(\cosh \xi)$. Moreover, we can use \eqref{eq:Pochhammer_asy} and Stirling's formula to obtain an estimate for $(\frac{1}{2}-it)_{n}(\frac{1}{2}+it)_{n}$ like \eqref{eq:gamma_ballistic} and \eqref{eq:gamma_ballistic_inner}. We then find that $p_n(t; \xi)$ vanishes exponentially fast as $t \in \{ z : z \in \contC^{(\diffusive)}_{\Lower},\, \Im z \geq \Im \hat{z}(\delta) \} \subset \contC$ moves away from $z_*$, and as $t \in \{ z : z \in \contC^{(\diffusive)}_{\Upper} \text{ and } \Im z \leq -\Im \hat{z}(\delta) \} \subset \contC$ moves away from $-z_*$. We can treat the other parts of the contour similarly as in  Section \ref{subsubsec:ballistic_soft}. We then find that the main contribution to the integral \eqref{eq:Khat}, with $\Sigma$ and $\mathcal{C}$ deformed as above, occurs at $s = z_* + \bigO(\alpha^{\frac{2}{3}}_n)$ and $t = \pm z_* + \bigO(n^{\frac{2}{3}})$.

We first consider the integral of \eqref{eq:Khat} for $s \in \Sigma^{(\diffusive, \Ai)}_{\Local}$ and $t \in \contC^{(\diffusive, \Ai)}_{\Local}$, where
\begin{equation}
  \star^{(\diffusive, \Ai)}_{\Local} = \star^{(\diffusive, \Ai)}_{\Local, \Right} \cup \text{the reflection of } \star^{(\diffusive, \Ai)}_{\Local, \Right} \text{ about } i\R, \qquad \star = \Sigma \text{ or } \contC.
\end{equation}

For $s \in \Sigma^{(\diffusive, \Ai)}_{\Local}$, we have, like \eqref{eq:est_q_ballistic_Ai},
\begin{equation}
  q_n(s; \eta) = \sqrt{\frac{\pi}{2 i s}} e^{-F(s; \zeta_{\max}) - 2n} e^{-is 2^{1/3} \alpha_n^{2/3} \sigma} (1 + \bigO(\alpha_n)).
\end{equation}
For $t \in \contC^{(\diffusive, \Ai)}_{\Local}$, we have, like \eqref{eq:est_p_ballistic_Ai},
\begin{equation}
  p_n(t; \xi) = \sqrt{\frac{\sqrt{t}}{2\pi i}} e^{F(t; \zeta_{\max}) + 2n} e^{it 2^{1/3} \alpha_n^{2/3} \tau} (1 + \bigO(\alpha_n)),
\end{equation}

The previous estimates and the change of variables 
\begin{align}
  2^{\frac{1}{3}} \alpha_n^{\frac{2}{3}}(s - z^*) \mapsto u,\qquad 2^{\frac{1}{3}} \alpha_n^{\frac{2}{3}}(t - z^*) \mapsto v,
\end{align}
then give, analogously to \eqref{eq:saddle_pt_ballistic_Ai}, that
\begin{align} \label{eq:Ai_limit_diffusive}
  &-2^{\frac{1}{3}} \alpha_n^{\frac{2}{3}} \frac{e^{iz^* \eta}}{e^{iz^* \xi}} \int_{\Sigma^{(\diffusive, \Ai)}_{\Local}} \frac{ds}{\pi i} \int_{\contC^{(\diffusive, \Ai)}_{\Local}} \frac{dt}{2\pi i} \frac{q_n(s; \eta) p_n(t; \xi)}{s-t} \\
  & \qquad =
  -\int_{\Sigma^{(\Ai)}}\frac{du}{\pi i} \int_{\contC^{(\Ai)}} \frac{dv}{2\pi i} \frac{1}{2i(u - v)} \frac{e^{\frac{i}{3} v^3 + i\tau v}}{e^{\frac{i}{3} u^3 + i\sigma u}} + \bigO(\alpha_n^{\frac{1}{3}}) + \bigO(M^{-1}_n),
\end{align}
where $\Sigma^{(\Ai)}$ and $\contC^{(\Ai)}$ are the piecewise linear contours defined in \eqref{eq:defn_Sigma_Ai} and \eqref{eq:defn_contC_Ai}. Furthermore, analogously to \eqref{eq:saddle_pt_ballistic_Ai_opposite}, we have that the integral of \eqref{eq:Khat} for $s \in \Sigma^{(\diffusive, \Ai)}_{\Local}$ and $t \in -\contC^{(\diffusive, \Ai)}_{\Local}$ satisfies
\begin{equation}
  -2^{\frac{1}{3}} \alpha_n^{\frac{2}{3}} \frac{e^{iz^* \eta}}{e^{iz^* \xi}} \int_{\Sigma^{(\diffusive, \Ai)}_{\Local}} \frac{ds}{\pi i} \oint_{\contC^{(\diffusive, \Ai)}_{\Local}} \frac{dt}{2\pi i} \frac{q_n(s; \eta) p_n(t; \xi)}{s+t} = \bigO(n^{-1} \alpha_n^{-\frac{2}{3}}) + \bigO(n^{-1} M^{-1}_n).
\end{equation}

We therefore obtain part \ref{enu:thm:metallic_3} of Theorem \ref{thm:metallic} with cocycle $e^{iz^* \eta}/e^{iz^* \xi}$.

\appendix

\section{Appendix}

In what follows, we will collect some helpful estimates and asymptotic properties of the Legendre functions. 
\medbreak

In the first part, we review some asymptotic expansions of the Legendre functions $P_{it - 1/2}(\cosh \xi)$ for $t \in \compC$ and $Q_{is - 1/2}(\cosh \eta)$ for $s \in \{z\in\C : \Im z \leq 0\}$ that are used in Section \ref{sec:LS}. First, we consider the regime where $t, s \to \infty$. By \cite[Eq. (39) \& (40)]{Jones01}, with
\begin{equation}
  a(\zeta) = \frac{1}{8} \left( \coth \zeta - \frac{1}{\zeta} \right), \quad \zeta \in (0, \infty),
\end{equation}
we have that uniformly for $\xi \in (0,\infty)$, as $t \to \infty$ in $\compC$,
\begin{equation} \label{eq:P_nu(x)_hard}
  P_{it - \frac{1}{2}}(\cosh \xi) = \left( \frac{\xi}{\sinh \xi} \right)^{\frac{1}{2}} \left[ I_0(it\xi) (1 + \bigO(t^{-2})) + \frac{I_1(it\xi)}{it} (a(\xi) + \bigO(t^{-2})) \right],
\end{equation}
and uniformly for $\eta \in (0,\infty)$, as $s \to \infty$ in $\{ \arg s \in [-\pi, 0] \}$,
\begin{equation} \label{eq:Q_nu(x)_hard}
  Q_{is - \frac{1}{2}}(\cosh \eta) = \left( \frac{\eta}{\sinh \eta} \right)^{\frac{1}{2}} \left[ K_0(is\eta) (1 + \bigO(s^{-2})) - \frac{K_1(is\eta)}{is} (a(\eta) + \bigO(s^{-2})) \right].
\end{equation}
We can further specify this using \cite[Eq. 10.40.1, 10.40.2 \& 10.40.5]{Boisvert-Clark-Lozier-Olver10}, which say that as $z \to \infty$,
\begin{align}
  I_{\nu}(z) = {}&
                   \begin{cases}
                     \frac{e^z}{\sqrt{2\pi z}} (1 + \bigO(z^{-1})), & \arg z \in (-\frac{\pi}{2} + \epsilon, \frac{\pi}{2} - \epsilon), \\
                     \frac{e^z}{\sqrt{2\pi z}} (1 + \bigO(z^{-1})) + i e^{\nu \pi i} \frac{e^{-z}}{\sqrt{2\pi z}} (1 + \bigO(z^{-1})), & \arg z \in (-\epsilon, \epsilon),
                   \end{cases}\label{eq:I0_asy} \\
  K_{\nu}(z) = {}& \left( \frac{\pi}{2z} \right)^{1/2} e^{-z} (1 + \bigO(z^{-1})), \quad \arg z \in (-\frac{3\pi}{2} + \epsilon, \frac{3\pi}{2} - \epsilon). \label{eq:K0_asy}
\end{align}
Furthermore, $I_n(-z) = (-1)^n I_n(z)$, see \cite[Eq. 10.25.2]{Boisvert-Clark-Lozier-Olver10}.

Next, we consider the regime where $t, s$ are in a compact subset of $\compC$ and $\xi, \eta \to \infty$. We have the following expressions of $P_{it - 1/2}(\cosh \xi)$ and $Q_{is - 1/2}(\cosh \eta)$ as hypergeometric functions
\begin{multline} \label{eq:P_in_hyper}
  P_{it - \frac{1}{2}}(\cosh \xi) = \frac{\Gamma(it)}{\sqrt{\pi} \Gamma(it + \frac{1}{2})} (2\cosh \xi)^{it - \frac{1}{2}} \hyperF{2}{1}{\frac{3}{4} - \frac{it}{2}, \frac{1}{4} - \frac{it}{2}}{1 - it}{\frac{1}{\cosh^2 \xi}} \\
  + \frac{\Gamma(-it)}{\sqrt{\pi} \Gamma(-it + \frac{1}{2})} (2\cosh \xi)^{-it - \frac{1}{2}} \hyperF{2}{1}{\frac{3}{4} + \frac{it}{2}, \frac{1}{4} + \frac{it}{2}}{1 + it}{\frac{1}{\cosh^2 \xi}},
\end{multline}
\begin{equation}
  Q_{is - \frac{1}{2}}(\cosh \eta) = \frac{\sqrt{\pi} \Gamma(is + \frac{1}{2})}{\Gamma(is + 1) (2\cosh \eta)^{is + 1/2}} \hyperF{2}{1}{\frac{3}{4} + \frac{is}{2}, \frac{1}{4} + \frac{is}{2}}{1 + is}{\frac{1}{\cosh^2 \eta}},
\end{equation}
see \cite[Eq. (7.6.1) \& (7.6.3)]{Lebedev65}. By the expansion formula of hypergeometric functions, if $s$ is in a compact subset of $\compC$ that excludes $i\Z_{\geq 1}$, then uniformly for $\eta \to \infty$,
\begin{equation} \label{eq:asy_Q_small_s}
  Q_{is - \frac{1}{2}}(\cosh \eta) = \frac{\sqrt{\pi} \Gamma(is + \frac{1}{2})}{\Gamma(is + 1) (2\cosh \eta)^{is + 1/2}} \left( 1 + \bigO \left( \frac{1}{\cosh^2 \eta} \right) \right).
\end{equation}
Similarly, if $t$ is in a compact subset of $\compC$ that excludes $i\Z_{\neq 0}$, then uniformly for $\xi \to \infty$,
\begin{multline} \label{eq:asy_P_small_t}
  P_{it - \frac{1}{2}}(\cosh \xi) = \frac{\Gamma(it)}{\sqrt{\pi} \Gamma(it + \frac{1}{2})} (2\cosh \xi)^{it - \frac{1}{2}} \left( 1 + \bigO \left( \frac{1}{\cosh^2 \xi} \right) \right)  \\
  + \frac{\Gamma(-it)}{\sqrt{\pi} \Gamma(-it + \frac{1}{2})} (2\cosh \xi)^{-it - \frac{1}{2}} \left( 1 + \bigO \left( \frac{1}{\cosh^2 \xi} \right) \right).
\end{multline}

We also need the following error bounds. For $t \in \compC_+$ with $\Im t > 2$ and $\xi \in [\epsilon, \infty)$, the asymptotics of $P_{it - 1/2}(\cosh \xi)$ can be obtained by \eqref{eq:asy_P_it} together with the symmetry \eqref{eq:symmetry_P_Q}.

\begin{lem}[Thm. 3.1 \& Eq. (1.10) in \cite{Nemes-Olde_Daalhuis20}]
Let $\epsilon > 0$. Then for all $\xi \in [\epsilon, \infty)$ and $\Im t < -2$, there exists $C = C(\epsilon) > 0$ such that
\begin{equation} \label{eq:asy_P_it}
 \left\lvert \frac{\sqrt{2\pi \sinh \xi}}{e^{it\xi}} P_{it - \frac{1}{2}}(\cosh \xi) - \frac{\Gamma(it)}{\Gamma(it + \frac{1}{2})} \right\rvert \leq C e^{-\xi} \frac{\Gamma(-\Im t - 1)}{\lvert \Gamma(it + \frac{1}{2}) \rvert}.
\end{equation}
\end{lem}

\begin{lem}[Thm. 3.3 \& 3.4, Eq. (1.12) in \cite{{Nemes-Olde_Daalhuis20}}]
Let $\epsilon, \delta > 0$ and $M > 0$. Then for all $\eta \in [\epsilon,\infty)$ and $s \in D_{\delta, M}$ where $D_{\delta, M} = \{ z \in \compC : \arg z \in [-\pi + \delta, -\delta] \text{ and } \lvert z \rvert \geq M \}$, there exists $C = C(\epsilon, \delta, M) > 0$ such that
  \begin{equation} \label{eq:asy_Q_is}
    \left\lvert \sqrt{\frac{2 \sinh \eta}{\pi}} e^{is\eta} Q_{is - \frac{1}{2}}(\cosh \eta) - \frac{\Gamma(is + \frac{1}{2})}{\Gamma(is + 1)} \right\rvert \leq C e^{-\eta} \lvert s \rvert^{-\frac{3}{2}}.
  \end{equation}    
\end{lem}

In the second part, we collect some estimates that are used in Section \ref{sec:IPUM}. All the propositions below can be derived from estimates \eqref{eq:P_nu(x)_hard}--\eqref{eq:K0_asy}. They are also direct consequences of the integral representations
\begin{equation} \label{eq:Leg-int-repr}
    P_{is-\frac12}(x)
    =\frac1\pi\int_0^\pi
    \left(x+\sqrt{x^2-1}\cos\theta\right)^{is-\frac12}
    \,d\theta,
    \qquad x\geq1,\quad s\in\C,
\end{equation}
\begin{equation}\label{eq:Leg-Q-int-repr}
Q_\nu(x)=\int_0^\infty
\left(x+\sqrt{x^2-1}\cosh u\right)^{-\nu-1}du,
\qquad x>1,\qquad \Re\nu>-1,
\end{equation} 
see \cite[Eqs.~14.12.7, 14.12.6 and 14.3.10]{Boisvert-Clark-Lozier-Olver10}.

\begin{prop}\label{prop:Leg-geom-est}
Let $A\geq 0$. There exists $C_A > 0$ such that, for all $x\geq1$ and $s\in\C$ with $|\Im s|\leq A$, 
\[
|P_{is-\frac12}(x)|\leq C_A x^{A+\frac12}.
\]
\end{prop}

\begin{prop}\label{prop:Leg-der-s-est}
Let $m\in\N$ and $R>1$. There exists $C_{m,R}>0$ such that,
for all $s\in\R$,
\[
    \sup_{1\leq x\leq R}
    |P_{is-\frac12}^{(m)}(x)|
    \leq C_{m,R}(1+|s|)^{2m}.
\]
\end{prop}

\begin{prop}\label{prop:Leg-Q-geom-est}
Let $A>0$. There exists $C_{A}>0$ such that, for all $x\geq 2$, 
\[
|Q_{A+it-\frac12}(x)|
\leq C_Ax^{-A-\frac12},
\quad t\in\R.
\]
\end{prop}

\begin{prop}\label{prop:Leg-Q-cont}
    Let $a>1$ and $s\in\R$. Then $Q_{is-\frac12}\in C^{0}(\R_{\geq a})$ and 
    $$ \|Q_{is-\frac12}\|_{\infty}
    % < Q_{-\frac12}(a)
    <\infty.$$
\end{prop}

\section*{Acknowledgments}

T.~W.~was supported by the grant KAW 2023.0216 from the Knut and Alice Wallenberg Foundation. D.~W.~was partially supported by the National Natural Science Foundation of China under grant number 12271502, and the University of Chinese Academy of Sciences start-up grant 118900M043. We would like to thank the MATRIX Institute for hosting the research program \textit{Log-gases in Caeli Australi} in Creswick (Australia) during the first half of August 2025, where this project was started. We are also grateful to Sergey Berezin for several helpful discussions regarding the local statistics obtained at crossover II, and to Maurice Duits for exploring a potential connection to his work \cite{Breuer-Duits13}. We acknowledge the use of GPT-6 Astra in the final stages of writing the manuscript as a tool to resolve some technicalities involving the set $\mathcal{W}(\R_{\geq 1})$ in the proofs of Propositions \ref{prop:admissable-weight-prop}, \ref{prop:transform-multiplicativity-C} and \ref{prop:gaussian-weights-admissible}, and to obtain appropriate remainder estimates in the proofs of Propositions \ref{qfactor} and \ref{pfactor}. The AI was used under the authors' strict mathematical guidance. All arguments were rigorously reviewed, verified and substantially revised.

\def\cydot{\leavevmode\raise.4ex\hbox{.}}
  \newcommand{\ssf}[1]{\mathsf{#1}}\renewcommand{\Bbb}[1]{\mathbb{#1}}

% \bibliographystyle{plain}
% \bibliography{bibliography.bib}

\begin{thebibliography}{10}

\bibitem{Ahn23}
Andrew Ahn.
\newblock Extremal singular values of random matrix products and {B}rownian
  motion on {$\ssf{GL} (N,\Bbb {C})$}.
\newblock {\em Probab. Theory Related Fields}, 187(3-4):949--997, 2023.

\bibitem{Akemann-Burda-Kieburg18}
Gernot Akemann, Zdzislaw Burda, and Mario Kieburg.
\newblock From integrable to chaotic systems: Universal local statistics of
  {L}yapunov exponents.
\newblock {\em Europhys. Lett. {EPL}}, 126(4):40001, 7pp, 2019.

\bibitem{Akemann-Burda-Kieburg20}
Gernot Akemann, Zdzislaw Burda, and Mario Kieburg.
\newblock Universality of local spectral statistics of products of random
  matrices.
\newblock {\em Phys. Rev. E}, 102(5):052134, 27, 2020.

\bibitem{Akemann-Ipsen15}
Gernot Akemann and Jesper~R. Ipsen.
\newblock Recent exact and asymptotic results for products of independent
  random matrices.
\newblock {\em Acta Phys. Polon. B}, 46(9):1747--1784, 2015.

\bibitem{Akemann-Kieburg-Wei13}
Gernot Akemann, Mario Kieburg, and Lu~Wei.
\newblock Singular value correlation functions for products of {W}ishart random
  matrices.
\newblock {\em J. Phys. A}, 46(27):275205, 22, 2013.

\bibitem{Anker-Orsted05}
Jean-Philippe Anker and Bent Orsted, editors.
\newblock {\em Lie theory}, volume 229 of {\em Progress in Mathematics}.
\newblock Birkh\"auser Boston, Inc., Boston, MA, 2005.
\newblock Unitary representations and compactifications of symmetric spaces.

\bibitem{Assiotis23}
Theodoros Assiotis.
\newblock Exact solution of interacting particle systems related to random
  matrices.
\newblock {\em Comm. Math. Phys.}, 402(3):2641--2690, 2023.

\bibitem{Bachmann-Butz-De_Roeck12}
Sven Bachmann, Maximilian Butz, and Wojciech De~Roeck.
\newblock Disordered quantum wires: microscopic origins of the {DMPK} theory
  and {O}hm's {L}aw.
\newblock {\em J. Stat. Phys.}, 148(1):164--189, 2012.

\bibitem{Baik-Ben_Arous-Peche05}
Jinho Baik, G{\'e}rard Ben~Arous, and Sandrine P{\'e}ch{\'e}.
\newblock Phase transition of the largest eigenvalue for nonnull complex sample
  covariance matrices.
\newblock {\em Ann. Probab.}, 33(5):1643--1697, 2005.

\bibitem{Beenakker97}
C.~W.~J. Beenakker.
\newblock Random-matrix theory of quantum transport.
\newblock {\em Rev. Mod. Phys.}, 69:731--808, Jul 1997.

\bibitem{Beenakker-Rejaei93}
C.~W.~J. Beenakker and B.~Rejaei.
\newblock Nonlogarithmic repulsion of transmission eigenvalues in a disordered
  wire.
\newblock {\em Phys. Rev. Lett.}, 71:3689--3692, Nov 1993.

\bibitem{Beenakker-Rejaei94}
C.~W.~J. Beenakker and B.~Rejaei.
\newblock Exact solution for the distribution of transmission eigenvalues in a
  disordered wire and comparison with random-matrix theory.
\newblock {\em Phys. Rev. B}, 49(11):7499, 1994.

\bibitem{Berezin-Strahov22}
Sergey Berezin and Eugene Strahov.
\newblock Gap probability for products of random matrices in the critical
  regime.
\newblock {\em J. Approx. Theory}, 274:Paper No. 105687, 29, 2022.

\bibitem{Breuer-Duits13}
Jonathan Breuer and Maurice Duits.
\newblock Central limit theorems for biorthogonal ensembles and asymptotics of
  recurrence coefficients.
\newblock {\em J. Amer. Math. Soc.}, 30(1):27--66, 2017.

\bibitem{Cepa-Lepingle97}
Emmanuel C\'epa and Dominique L\'epingle.
\newblock Diffusing particles with electrostatic repulsion.
\newblock {\em Probab. Theory Related Fields}, 107(4):429--449, 1997.

\bibitem{Claeys-Zhang26}
Tom Claeys and Jiyuan Zhang.
\newblock Biorthogonal ensembles of derivative type, 2026.
\newblock arXiv:2601.18427.

\bibitem{Daems-Kuijlaars-Veys08}
Evi Daems, Arno B.~J. Kuijlaars, and Willem Veys.
\newblock Asymptotics of non-intersecting {B}rownian motions and a {$4\times
  4$} {R}iemann-{H}ilbert problem.
\newblock {\em J. Approx. Theory}, 153(2):225--256, 2008.

\bibitem{Delvaux-Kuijlaars-Roman-Zhang12}
S.~Delvaux, A.~B.~J. Kuijlaars, P.~Rom\'an, and L.~Zhang.
\newblock Non-intersecting squared {B}essel paths with one positive starting
  and ending point.
\newblock {\em J. Anal. Math.}, 118(1):105--159, 2012.

\bibitem{Delvaux-Kuijlaars-Zhang11}
Steven Delvaux, Arno B.~J. Kuijlaars, and Lun Zhang.
\newblock Critical behavior of nonintersecting {B}rownian motions at a tacnode.
\newblock {\em Comm. Pure Appl. Math.}, 64(10):1305--1383, 2011.

\bibitem{Dorokhov82}
O.N. Dorokhov.
\newblock Transmission coefficient and the localization length of an electron
  in $n$ bound disordered chains.
\newblock {\em JETP Lett.}, 36:318, 1982.

\bibitem{Duits18}
Maurice Duits.
\newblock On global fluctuations for non-colliding processes.
\newblock {\em Ann. Probab.}, 46(3):1279--1350, 2018.

\bibitem{Dumaz-Labbe20}
Laure Dumaz and Cyril Labb\'e.
\newblock Localization of the continuous {A}nderson {H}amiltonian in 1-{D}.
\newblock {\em Probab. Theory Related Fields}, 176(1-2):353--419, 2020.

\bibitem{Dumaz-Labbe23}
Laure Dumaz and Cyril Labb\'e.
\newblock The delocalized phase of the {A}nderson {H}amiltonian in 1-{D}.
\newblock {\em Ann. Probab.}, 51(3):805--839, 2023.

\bibitem{Dumaz-Labbe24}
Laure Dumaz and Cyril Labb\'e.
\newblock Localization crossover for the continuous {A}nderson {H}amiltonian in
  1-d.
\newblock {\em Invent. Math.}, 235(2):345--440, 2024.

\bibitem{Edelman-Sungwoo23}
Alan Edelman and Sungwoo Jeong.
\newblock Fifty three matrix factorizations: a systematic approach.
\newblock {\em SIAM J. Matrix Anal. Appl.}, 44(2):415--480, 2023.

\bibitem{Forster-Kieburg-Kosters21}
Yanik-Pascal F\"orster, Mario Kieburg, and Holger K\"osters.
\newblock Polynomial ensembles and {P}\'olya frequency functions.
\newblock {\em J. Theoret. Probab.}, 34(4):1917--1950, 2021.

\bibitem{Frahm95}
K.~Frahm.
\newblock Equivalence of the {F}okker-{P}lanck approach and the nonlinear
  $\ensuremath{\sigma}$ model for disordered wires in the unitary symmetry
  class.
\newblock {\em Phys. Rev. Lett.}, 74:4706--4709, Jun 1995.

\bibitem{Glaeske-Prudnikov-Skornik06}
H.-J. Glaeske, A.~P. Prudnikov, and K.~A. Sk\`ornik.
\newblock {\em Operational calculus and related topics}, volume~10 of {\em
  Analytical Methods and Special Functions}.
\newblock Chapman \& Hall/CRC, Boca Raton, FL, 2006.

\bibitem{Graczyk-Malecki14}
Piotr Graczyk and Jacek Ma{\l}ecki.
\newblock Strong solutions of non-colliding particle systems.
\newblock {\em Electron. J. Probab.}, 19:no. 119, 21, 2014.

\bibitem{Helgason00}
Sigurdur Helgason.
\newblock {\em Groups and geometric analysis}, volume~83 of {\em Mathematical
  Surveys and Monographs}.
\newblock American Mathematical Society, Providence, RI, 2000.
\newblock Integral geometry, invariant differential operators, and spherical
  functions, Corrected reprint of the 1984 original.

\bibitem{Helgason01}
Sigurdur Helgason.
\newblock {\em Differential geometry, {L}ie groups, and symmetric spaces},
  volume~34 of {\em Graduate Studies in Mathematics}.
\newblock American Mathematical Society, Providence, RI, 2001.
\newblock Corrected reprint of the 1978 original.

\bibitem{Hoogenboom82}
Bob Hoogenboom.
\newblock Spherical functions and invariant differential operators on complex
  {G}rassmann manifolds.
\newblock {\em Ark. Mat.}, 20(1):69--85, 1982.

\bibitem{Johansson01a}
Kurt Johansson.
\newblock Universality of the local spacing distribution in certain ensembles
  of {H}ermitian {W}igner matrices.
\newblock {\em Comm. Math. Phys.}, 215(3):683--705, 2001.

\bibitem{Johansson03}
Kurt Johansson.
\newblock Discrete polynuclear growth and determinantal processes.
\newblock {\em Comm. Math. Phys.}, 242(1-2):277--329, 2003.

\bibitem{Johansson06}
Kurt Johansson.
\newblock Random matrices and determinantal processes.
\newblock In {\em Mathematical statistical physics}, pages 1--55. Elsevier B.
  V., Amsterdam, 2006.

\bibitem{Johansson13}
Kurt Johansson.
\newblock Non-colliding {B}rownian motions and the extended tacnode process.
\newblock {\em Comm. Math. Phys.}, 319(1):231--267, 2013.

\bibitem{Jones01}
D.~S. Jones.
\newblock Asymptotics of the hypergeometric function.
\newblock volume~24, pages 369--389. 2001.
\newblock Applied mathematical analysis in the last century.

\bibitem{Kieburg-Kosters16}
Mario Kieburg and Holger K\"osters.
\newblock Exact relation between singular value and eigenvalue statistics.
\newblock {\em Random Matrices Theory Appl.}, 5(4):1650015, 57, 2016.

\bibitem{Kieburg-Kosters19}
Mario Kieburg and Holger K\"osters.
\newblock Products of random matrices from polynomial ensembles.
\newblock {\em Ann. Inst. Henri Poincar\'e{} Probab. Stat.}, 55(1):98--126,
  2019.

\bibitem{Kieburg-Li-Zhang-Forrester23}
Mario Kieburg, Shi-Hao Li, Jiyuan Zhang, and Peter~J. Forrester.
\newblock Cyclic {P}\'olya ensembles on the unitary matrices and their spectral
  statistics.
\newblock {\em Constr. Approx.}, 57(3):1063--1108, 2023.

\bibitem{Kuijlaars-Martinez_Finkelshtein-Wielonsky09}
A.~B.~J. Kuijlaars, A.~Mart{\'{\i}}nez-Finkelshtein, and F.~Wielonsky.
\newblock Non-intersecting squared {B}essel paths and multiple orthogonal
  polynomials for modified {B}essel weights.
\newblock {\em Comm. Math. Phys.}, 286(1):217--275, 2009.

\bibitem{Kuijlaars-Martinez_Finkelshtein-Wielonsky11}
A.~B.~J. Kuijlaars, A.~Mart{\'{\i}}nez-Finkelshtein, and F.~Wielonsky.
\newblock Non-intersecting squared {B}essel paths: critical time and double
  scaling limit.
\newblock {\em Comm. Math. Phys.}, 308(1):227--279, 2011.

\bibitem{Kuijlaars-Roman19}
Arno B.~J. Kuijlaars and Pablo Rom\'an.
\newblock Spherical functions approach to sums of random {H}ermitian matrices.
\newblock {\em Int. Math. Res. Not. IMRN}, (4):1005--1029, 2019.

\bibitem{Kuijlaars-Zhang14}
Arno B.~J. Kuijlaars and Lun Zhang.
\newblock Singular values of products of {G}inibre random matrices, multiple
  orthogonal polynomials and hard edge scaling limits.
\newblock {\em Comm. Math. Phys.}, 332(2):759--781, 2014.

\bibitem{Lebedev65}
N.~N. Lebedev.
\newblock {\em Special functions and their applications}.
\newblock Prentice-Hall, Inc., Englewood Cliffs, NJ, english edition, 1965.

\bibitem{Liu-Wang-Wang18}
Dang-Zheng Liu, Dong Wang, and Yanhui Wang.
\newblock Lyapunov exponent, universality and phase transition for products of
  random matrices.
\newblock {\em Comm. Math. Phys.}, 399(3):1811--1855, 2023.

\bibitem{Liu-Wang-Zhang14}
Dang-Zheng Liu, Dong Wang, and Lun Zhang.
\newblock Bulk and soft-edge universality for singular values of products of
  {G}inibre random matrices.
\newblock {\em Ann. Inst. Henri Poincar\'e Probab. Stat.}, 52(4):1734--1762,
  2016.

\bibitem{Macedo94a}
A.~M.~S. Mac{\^e}do.
\newblock Universal parametric correlations in the transmission eigenvalue
  spectra of disordered conductors.
\newblock {\em Phys. Rev. B}, 49:16841(R)--16844(R), Jun 1994.

\bibitem{Mehta04}
Madan~Lal Mehta.
\newblock {\em Random matrices}, volume 142 of {\em Pure and Applied
  Mathematics (Amsterdam)}.
\newblock Elsevier/Academic Press, Amsterdam, third edition, 2004.

\bibitem{Mello-Pereyra-Kumar88}
PA~Mello, P~Pereyra, and N~Kumar.
\newblock Macroscopic approach to multichannel disordered conductors.
\newblock {\em Annals of Physics}, 181(2):290--317, 1988.

\bibitem{Nemes-Olde_Daalhuis20}
Gerg\H{o} Nemes and Adri~B. Olde~Daalhuis.
\newblock Large-parameter asymptotic expansions for the {L}egendre and allied
  functions.
\newblock {\em SIAM J. Math. Anal.}, 52(1):437--470, 2020.

\bibitem{Neuschel-Venker24}
Thorsten Neuschel and Martin Venker.
\newblock Boundary asymptotics of non-intersecting {B}rownian motions:
  {P}earcey, {A}iry and a transition.
\newblock {\em Nonlinearity}, 37(3):Paper No. 035008, 47, 2024.

\bibitem{Boisvert-Clark-Lozier-Olver10}
Frank W.~J. Olver, Daniel~W. Lozier, Ronald~F. Boisvert, and Charles~W. Clark,
  editors.
\newblock {\em N{IST} handbook of mathematical functions}.
\newblock U.S. Department of Commerce, National Institute of Standards and
  Technology, Washington, DC; Cambridge University Press, Cambridge, 2010.
\newblock With 1 CD-ROM (Windows, Macintosh and UNIX).

\bibitem{Stanton-Tomas79}
Robert~J. Stanton and Peter~A. Tomas.
\newblock Pointwise inversion of the spherical transform on {$L\sp{p}(G/K)$},
  {$1\leq p<2$}.
\newblock {\em Proc. Amer. Math. Soc.}, 73(3):398--404, 1979.

\bibitem{Tracy-Widom94}
Craig~A. Tracy and Harold Widom.
\newblock Level-spacing distributions and the {A}iry kernel.
\newblock {\em Comm. Math. Phys.}, 159(1):151--174, 1994.

\bibitem{Tracy-Widom94b}
Craig~A. Tracy and Harold Widom.
\newblock Level spacing distributions and the {B}essel kernel.
\newblock {\em Comm. Math. Phys.}, 161(2):289--309, 1994.

\bibitem{Wang23}
Dong Wang and Dong Yao.
\newblock Biorthogonal polynomials related to quantum transport theory of
  disordered wires.
\newblock {\em Comm. Math. Phys.}, 407(2):Paper No. 24, 79, 2026.

\end{thebibliography}

\end{document}